\documentclass[11pt]{amsart}

\usepackage[left=2.8cm,right=2.8cm,top=2.8cm,bottom=2.8cm]{geometry}

\usepackage[utf8]{inputenc}
\usepackage[T1]{fontenc}
\usepackage{ae}

\usepackage[english]{babel}

\usepackage[colorlinks=true, urlcolor=blue, citecolor=red, linkcolor=blue, linktocpage, pdfpagelabels, bookmarksnumbered, bookmarksopen]{hyperref}

\usepackage{amsmath}
\usepackage{amssymb}
\usepackage{amsfonts}
\usepackage{amsthm}
\usepackage{bbm}
\usepackage{accents}
\usepackage{mathtools}
\usepackage{enumitem}
\usepackage{cite}

\theoremstyle{plain}
\newtheorem{theorem}{Theorem}[section]
\newtheorem{proposition}[theorem]{Proposition}
\newtheorem{lemma}[theorem]{Lemma}
\newtheorem{corollary}[theorem]{Corollary}
\newtheorem{claim}{Claim}
\theoremstyle{definition}
\newtheorem{definition}[theorem]{Definition}
\newtheorem{remark}[theorem]{Remark}

\numberwithin{equation}{section}

\newcounter{hypothesisnumber}
\newenvironment{hypothesis}{\stepcounter{hypothesisnumber}\begin{equation*}\tag{H\thehypothesisnumber}}{\end{equation*}\ignorespacesafterend}

\newcommand{\deb}{\rightharpoonup}
\newcommand{\suchthat}{\,:\,}
\newcommand{\diff}{\,\mathrm{d}}
\newcommand{\case}[1]{\medskip\noindent\emph{\textbullet{} #1}\nopagebreak}

\DeclareMathOperator{\spt}{spt}
\DeclareMathOperator{\diam}{diam}
\DeclareMathOperator{\convex}{conv}
\DeclareMathOperator{\interior}{int}

\DeclareMathOperator{\trace}{Tr}

\DeclarePairedDelimiter\abs{\lvert}{\rvert}
\DeclarePairedDelimiter\norm{\lVert}{\rVert}

\let\originalleft\left
\let\originalright\right
\renewcommand{\left}{\mathopen{}\mathclose\bgroup\originalleft}
\renewcommand{\right}{\aftergroup\egroup\originalright}

\begin{document}

\setlength{\parskip}{0pt plus 4pt minus 0pt} % Corrects spacing with amsart.

\setlist[enumerate, 1]{label={\textnormal{(\alph*)}}, ref={(\alph*)}, leftmargin=0pt, itemindent=*}
\setlist[itemize, 1]{label={\textbullet}, leftmargin=0pt, itemindent=*}
\setlist[description, 1]{leftmargin=0pt, itemindent=*}
\setlist[enumerate, 2]{label={\textnormal{(\roman*)}}, ref={(\roman*)}}

\title[Sharp semi-concavity and $L^p$ estimates in an optimal-exit MFG]{Sharp semi-concavity in a non-autonomous control problem and $L^p$ estimates in an optimal-exit MFG}
\author{Samer Dweik}
\address{Laboratoire de Math\'ematiques d'Orsay, Universit\'e Paris-Sud, CNRS, Universit\'e Paris-Saclay, 91405 Orsay, France.}
\email{samer.dweik@math.u-psud.fr}
\author{Guilherme Mazanti}
\address{Laboratoire de Math\'ematiques d'Orsay, Universit\'e Paris-Sud, CNRS, Universit\'e Paris-Saclay, 91405 Orsay, France.}
\email{guilherme.mazanti@math.u-psud.fr}
\date{\today}

\thanks{This work was partially supported by a public grant as part of the ``Investissement d'avenir'' project, reference
ANR-11-LABX-0056-LMH, LabEx LMH, PGMO project VarPDEMFG. The first author was also partially supported by the by the French ANR project ``GEOMETRYA'', reference ANR-12-BS01-0014, and by the R\'egion Ile-de-France. The second author was also partially supported by the French ANR project ``MFG'', reference ANR-16-CE40-0015-01, and by the Hadamard Mathematics LabEx (LMH) through the grant number ANR-11-LABX-0056-LMH in the ``Investissement d'avenir'' project.}



\begin{abstract}
This paper studies a mean field game inspired by crowd motion in which agents evolve in a compact domain and want to reach its boundary minimizing the sum of their travel time and a given boundary cost. Interactions between agents occur through their dynamic, which depends on the distribution of all agents.

We start by considering the associated optimal control problem, showing that semi-concavity in space of the corresponding value function can be obtained by requiring as time regularity only a lower Lipschitz bound on the dynamics. We also prove differentiability of the value function along optimal trajectories under extra regularity assumptions.

We then provide a Lagrangian formulation for our mean field game and use classical techniques to prove existence of equilibria, which are shown to satisfy a MFG system. Our main result, which relies on the semi-concavity of the value function, states that an absolutely continuous initial distribution of agents with an $L^p$ density gives rise to an absolutely continuous distribution of agents at all positive times with a uniform bound on its $L^p$ norm. This is also used to prove existence of equilibria under fewer regularity assumptions on the dynamics thanks to a limit argument.
\end{abstract}

\maketitle

\tableofcontents

\section{Introduction}

Mean field games (MFGs for short) are differential games with a continuum of rational players, assumed to be indistinguishable, individually neglectable, and influenced only by the average behavior of other players through a mean-field type interaction. Introduced independently around 2006 by Jean-Michel Lasry and Pierre-Louis Lions \cite{Lasry2006, Lasry2006Jeux, Lasry3} and by Peter E.\ Caines, Minyi Huang, and Roland P.\ Malhamé \cite{Huang2006, Huang2007, Huang2003Individual}, these models have since been studied from several perspectives, including approximation of games with a large number of players by MFGs \cite{Kolokoltsov2014Rate, Cardaliaguet2017Convergence}, numerical methods for approximating MFG equilibria \cite{Gueant2012New, Carlini2014Fully, Achdou2010Mean, Achdou2016Convergence}, games with large time horizon \cite{Cardaliaguet2013Long1, Cardaliaguet2013Long2}, variational mean field games \cite{Cardaliaguet2016First, BenCarSan, Meszaros2015Variational, Prosinski2017Global}, games on graphs or networks \cite{Gueant2015Existence, Camilli2015Model, Gomes2013Continuous, Cacace2017Numerical}, or the characterization of equilibria using the master equation \cite{CardaliaguetMaster, Bensoussan2017Interpretation, Carmona2014Master}. We refer to \cite{Gomes2014Mean, Gueant2011Mean, CardaliaguetNotes} for more details and further references on mean field games. The words ``player'' and ``agent'' are used interchangeably in this paper to refer to those taking part in a game.

This paper continues the analysis of the mean field game model introduced in \cite{MazantiMinimal}, which considers players evolving in a compact domain $\Omega \subset \mathbb R^d$, their goal being to reach the boundary $\partial\Omega$. The distribution of players at time $t \geq 0$ is described by a Borel probability measure $\rho_t \in \mathcal P(\Omega)$ and, as in \cite{MazantiMinimal}, we assume that the interaction between players occur through their dynamics, the trajectory $\gamma: [0, +\infty) \to \Omega$ of a given player being described by the control system $\gamma^\prime(t) = k(\rho_t, \gamma(t)) u(t)$, where the control $u: \mathbb [0, +\infty) \to \mathbb R^d$ satisfies $\abs{u(t)} \leq 1$ for every $t \geq 0$ and the function $k: \mathcal P(\Omega) \times \Omega \to \mathbb [0, +\infty)$ describes the maximal speed $k(\mu, x)$ an agent may have when their position is $x$ and agents are distributed according to $\mu$. A player chooses their control $u$ in order to minimize the sum of their travel time to $\partial\Omega$ with a boundary cost $g(z)$ on their arrival position $z \in \partial\Omega$, which is a generalization of the time-minimization criterion of \cite{MazantiMinimal}.

The above mean field game is proposed as a simple model for crowd motion, in which the crowd, modeled macroscopically by the measures $\rho_t \in \mathcal P(\Omega)$, evolves in $\Omega$ and wishes to leave this domain through its boundary $\partial\Omega$ while optimizing exit time and position. Crowd motion has been extensively studied from a mathematical point of view, with a wide range of models being used to describe crowd behavior, ranging from Maxwell--Boltzmann models \cite{Henderson1971Statistics}, particle systems \cite{Helbing2000Simulating}, granular media \cite{Faure2015Crowd}, scalar conservation laws \cite{Colombo2005Pedestrian}, time-varying measures \cite{Piccoli2011Time}, or models based on gradient flows \cite{Maury2010Macroscopic}. Several works also address the question of controlling crowd behavior \cite{Cristiani2015Modeling, Albi2016Invisible}. Mean field games have already been used to model crowd motion, for instance in \cite{Lachapelle2011Mean, BenCarSan, Cardaliaguet2016First, Burger}, however these models differ from ours since they consider a fixed final time, identical for all agents, and no constraints on the control, which is instead penalized on the cost function. As detailed in \cite{MazantiMinimal}, the model we consider here is also closely related to Hughes' model for crowd motion \cite{Hughes2002Continuum, Hughes2003Flow} and our notion of equilibrium is related to the standard notion of Wardrop equilibria in non-atomic congestion games \cite{Wardrop1952Theoretical, Carlier2011Continuous, Carlier2008Optimal, Cristiani2015Destination}.

The function $k$ in our model is intended to represent congestion, i.e., the difficulty of moving in high-density areas. Several mean field games with congestion have been previously considered \cite{Evangelista2018First, Achdou2018Mean, Prosinski2017Global, Gomes2015Short, Gomes2015Existence, Gueant2015Existence}, their common feature being to model congestion as a penalization in the cost function of each agent when passing through crowded regions. The penalization term is usually chosen as a negative power of the density, which introduces a singularity in the Hamilton--Jacobi equation of the corresponding optimal control problem. Our model considers instead that, in some crowd motion situations, an agent may not be able to move faster by simply paying some additional cost, since the congestion provoked by other agents may work as a physical barrier for the agent to increase their speed. Hence, we model congestion as a constraint on the maximal speed, given by $k$. Singularities on the Hamilton--Jacobi equation are avoided by assuming that $k$ is upper bounded.

In order to properly model congestion, $k$ should compute $k(\mu, x)$ by evaluating $\mu$ at or around $x$ and giving as a result some non-increasing function of this evaluation, meaning that the maximal speed of an agent is a non-increasing function of some ``average density'' around $x$. This is the case, for instance, when $k$ is given by
\begin{equation}
\label{IntroK}
k(\mu, x) = V\left(\int_\Omega \chi(x - y) \psi(y) \diff \mu(y)\right),
\end{equation}
where $\chi: \mathbb R^d \to [0, +\infty)$ is a convolution kernel, $\psi: \mathbb R^d \to [0, +\infty)$ may serve as a weight on $\Omega$ or as a cut-off function to discount some part of $\Omega$, and the non-increasing function $V: [0, +\infty) \to [0, +\infty)$ provides the maximal speed in terms of the average density computed by the integral. Even though the results of this paper do not assume a particular form for $k$, we make use of \eqref{IntroK} to justify some of our assumptions in Section \ref{SectionMFG} and we also verify that our main results apply when $k$ is given by \eqref{IntroK} and $V$, $\chi$, and $\psi$ satisfy suitable regularity assumptions.

We are interested in describing equilibria of the above mean field game, which correspond, roughly speaking, to evolutions $t \mapsto \rho_t$ for which almost every agent satisfies their optimization criterion. Contrarily to the classical approach for mean field games consisting on describing equilibria in terms of $\rho_t$, we adopt here a Lagrangian approach, which amounts to describing the motion of agents as a measure on the set of all possible trajectories. This classical approach in optimal transport \cite{Santambrogio2015Optimal, Ambrosio2005Gradient, Villani2009Optimal} has been used in some recent works on mean field games \cite{MazantiMinimal, BenCarSan, CanCap, Cardaliaguet2016First, Cardaliaguet2015Weak}.

In order to analyze the above mean field game, we start by considering the corresponding optimal control problem. Assuming that $k$ is a given function depending on time instead of the measure $\rho_t$, we start by obtaining some properties of optimal trajectories using classical optimal control techniques, such as Pontryagin Maximum Principle. We then prove our two main results for the value function $\varphi$ of this optimal control problem: if the time derivative of $k$ is lower bounded, then $\varphi$ is semi-concave in space (Theorem \ref{Theorem semiconcavity}) and, if $k$ is $C^{1, 1}$, then $\varphi$ is differentiable along optimal trajectories (Theorem \ref{Differentiability of the value function}).

After this preliminary study of the optimal control problem, we turn to the analysis of the mean field game itself. We start by proving existence of equilibria in a Lagrangian setting and obtaining the corresponding MFG system of PDEs on $\rho_t$ and the value function $\varphi$ using arguments similar to \cite{MazantiMinimal}. We then prove that an absolutely continuous initial distribution of agents with an $L^p$ density gives rise to an absolutely continuous distribution of agents at all positive times with a uniform bound on its $L^p$ norm (Theorem \ref{L^p MFG}) and use this result and a limit procedure to obtain existence of equilibria and the corresponding MFG system for a less regular model to which the arguments of \cite{MazantiMinimal} do not apply (Theorem \ref{existence of an equilibrium in less regular case}).

The paper is organized as follows. Useful notations and definitions used throughout the paper are provided in Section \ref{SecNotation}. Section \ref{SecOptimalControl} considers the optimal control problem corresponding to our mean field game, remarking first on Sections \ref{SecOC-First} and \ref{SecOC-PMP} that classical results for autonomous systems can be easily extended to a time-dependent framework with very few assumptions on the time regularity of the dynamics, before proving our main results in Sections \ref{Semi-concavity} and \ref{SecDifferentiabilityVarphi}. The mean field game is finally considered in Section \ref{SectionMFG}, with existence of equilibria and the MFG system being considered in Section \ref{SecExistenceMFG}, before the $L^p$ estimates of Section \ref{SecLp} and the ensuing results for a less regular mean field game in Section \ref{SecExistenceLessRegular}.

\section{Notations and definitions}
\label{SecNotation}

Let us set the main notation used in this paper. We let $\mathbb R^+ = [0, +\infty)$ and denote the usual Euclidean scalar product and norm in $\mathbb R^d$ by $x \cdot y$ and $\abs{x}$, respectively, for $x, y \in \mathbb R^d$. The closure, interior, convex hull, and diameter of a set $A \subset \mathbb R^d$ are denoted by $\overline A$, $\accentset\circ{A}$, $\convex A$, and $\diam(A)$, respectively, with $\accentset\circ{A}$ also denoted by $\interior A$. The open and closed Euclidean balls of center $x$ and radius $r$ in $\mathbb R^d$ are denoted respectively by $B(x, r)$ and $\bar B(x, r)$.

For $A \subset B$, the function $\mathbbm 1_A: B \to \mathbb \{0, 1\}$ denotes the characteristic function of $A$, i.e., $\mathbbm 1_A(x) = 1$ if and only if $x \in A$. Given two sets $A, B$, the notation $f : A \rightrightarrows B$ indicates that $f$ is a set-valued map from $A$ to $B$, i.e., $f$ maps a point $x \in A$ to a subset $f(x) \subset B$. The maps $\Pi_t: \mathbb R \times \mathbb R^d \to \mathbb R$ and $\Pi_x: \mathbb R \times \mathbb R^d \to \mathbb R^d$ denote the canonical projections into the factors of the product $\mathbb R \times \mathbb R^d$.

If $f$ is a function defined on (a subset of) $\mathbb R \times \mathbb R^d$, we use the notations $Df$, $\partial_t f$, and $\nabla f$ to denote its differential with respect to, respectively, both variables, its first variable, and its second variable. Similar notations are used for related notions, such as super and subdifferentials.

For a given metric space $X$ with metric $\mathbf d$, the notation $\mathbf d(x, A)$ for $x \in X$ and $A \subset X$ is defined as $\mathbf d(x, A) = \inf_{y \in A} \mathbf d(x, y)$. We denote by $\mathcal M(X)$ the set of all Borel nonnegative measures on $X$ endowed with the topology of weak convergence of measures, and $\mathcal P(X) \subset \mathcal M(X)$ denotes the subset of probability measures. The support of a measure $\eta \in \mathcal M(X)$ is denoted by $\spt(\eta)$. For $\eta \in \mathcal M(X)$ and $Y \subset X$ a Borel set, we denote by $\eta\rvert_{Y} \in \mathcal M(Y)$ the restriction of $\eta$ to $Y$. When $X \subset \mathbb R^d$ is Borel and $\eta \in \mathcal M(X)$ is absolutely continuous with respect to the Lebesgue measure, we use the same notation $\eta$ for its density.

Let $\Omega \subset \mathbb R^d$ be a compact domain. We denote by $C(\mathbb R^+, \Omega)$ the space of all continuous curves from $\mathbb{R}^+$ to $\Omega$, equipped with the topology of uniform convergence on compact sets, with respect to which $C(\mathbb R^+, \Omega)$ is a Polish space (see, for instance, \cite[Chapter X]{Bourbaki}). Whenever needed, we fix a metric $\mathbf d$ on $C(\mathbb R^+, \Omega)$; for instance, one may take $\mathbf d$ defined for $\gamma_1, \gamma_2 \in C(\mathbb R^+, \Omega)$ by
\begin{equation}
\label{DefiDistGamma}
\mathbf d(\gamma_1, \gamma_2) = \sum_{n=1}^\infty \frac{1}{2^n} \sup_{t \in [0, n]} \abs{\gamma_1(t) - \gamma_2(t)}.
\end{equation}
Recall that, if $L > 0$ is fixed, the set of all $L$-Lipschitz continuous curves in $C(\mathbb R^+, \Omega)$ is compact thanks to Arzelà--Ascoli Theorem.

We recall that, for $B \subset \mathbb R^d$, a function $u: B \to \mathbb R$ is called \emph{semi-concave} if it is continuous in $B$ and there exists $C \geq 0$ such that, for every $x, h \in \mathbb R^d$ with $[x - h, x + h] \subset B$, one has
\[
u(x - h) + u(x + h) - 2 u(x) \leq C \abs{h}^2.
\]
The constant $C$ is called a \emph{semi-concavity constant} for $u$. When $A \subset \mathbb R$, $B \subset \mathbb R^d$, and $\varphi: A \times B \to \mathbb R$, we say that $\varphi$ is semi-concave with respect to $x$, uniformly in $t$, if $x \mapsto \varphi(t, x)$ is semi-concave for every $t \in A$ with a semi-concavity constant independent of $t$.

We shall also need in this paper the classical notions of generalized gradients and some of their elementary properties, which we now recall, following the presentation from \cite[Chapter 3]{CanSin}.

\begin{definition}
Let $A \subset \mathbb R$, $B \subset \mathbb R^d$, $\varphi: A \times B \to \mathbb R$, and $(t, x) \in A \times B$. The sets
\begin{align*}
D^+ \varphi(t,x) & := \left\{(h,p) \in \mathbb{R} \times \mathbb{R}^d \suchthat \limsup_{(s,y) \to (t,x)} \frac{\varphi(s,y) - \varphi(t,x) - (h,p) \cdot (s-t,y-x)}{\abs{(s-t,y-x)}} \leq 0\right\}, \displaybreak[0] \\
D^- \varphi(t,x) & := \left\{(h,p) \in \mathbb{R} \times \mathbb{R}^d \suchthat \liminf_{(s,y) \to (t,x)} \frac{\varphi(s,y) - \varphi(t,x) - (h,p) \cdot (s-t,y-x)}{\abs{(s-t,y-x)}} \geq 0\right\},
\end{align*}
are called, respectively, the \emph{superdifferential} and \emph{subdifferential} of $\varphi$ at $(t, x)$. We also define the superdifferential and subdifferential of $\varphi$ with respect to $x$ by
\begin{align*}
\nabla^+ \varphi(t,x) & := \left\{p \in \mathbb{R}^d \suchthat \limsup_{y \to x} \frac{\varphi(t,y) - \varphi(t,x) - p \cdot (y-x)}{\abs{y-x}} \leq 0\right\}, \displaybreak[0] \\
\nabla^- \varphi(t,x) & := \left\{p \in \mathbb{R}^d \suchthat \liminf_{y \to x} \frac{\varphi(t,y) - \varphi(t,x) - p \cdot (y-x)}{\abs{y-x}} \geq 0\right\},
\end{align*}
respectively. Finally, we say that a vector $(h,p) \in \mathbb{R} \times \mathbb{R}^d$ is a \emph{reachable gradient} of $\varphi$ at $(t,x) \in A \times B$ if there is a sequence $\{(t_k,x_k)\}_k$ in $A \times B$ such that $\varphi$ is differentiable at $(t_k, x_k)$ for every $k \in \mathbb{N}$, and
$$\lim_{k \to \infty} (t_k,x_k) = (t,x),\qquad \lim_{k \to \infty} D\varphi(t_k,x_k) = (h, p).$$
The set of all reachable gradients of $\varphi$ at $(t,x)$ is denoted by $D^\star \varphi(t,x)$.
\end{definition}

As a simple consequence of the definitions of $D^+ \varphi$, $D^- \varphi$, $\nabla^+ \varphi$, and $\nabla^- \varphi$, we have the inclusions $\Pi_x(D^+ \varphi(t,x)) \subset \nabla^+ \varphi(t, x)$ and $\Pi_x(D^- \varphi(t,x)) \subset \nabla^- \varphi(t, x)$. Moreover, if $\varphi$ is Lipschitz continuous, then $D^\star \varphi(t,x)$ is a compact set: it is closed by definition and it is bounded since $\varphi$ is Lipschitz. From Rademacher's theorem it follows that $D^\star \varphi(t,x) \neq \emptyset$ for every $(t,x) \in \overline{\interior(A \times B)}$. We gather in the next proposition some classical additional properties for semi-concave functions (see, e.g., \cite[Propositions 3.3.1 and 3.3.4, Theorem 3.3.6, and Lemma 3.3.16]{CanSin}).

\begin{proposition}
\label{PropSemiconcave}
Let $A \subset \mathbb R$, $B \subset \mathbb R^d$, $\varphi: A \times B \to \mathbb R$ be semi-concave, and $(t, x) \in \interior(A \times B)$. Then
\begin{enumerate}
\item $D^\star \varphi(t, x) \subset \partial D^+ \varphi(t, x)$, where $\partial D^+ \varphi(t, x)$ denotes the topological boundary of $D^+ \varphi(t,x)$ in $\mathbb{R}^d$;
\item $D^+ \varphi(t, x) \neq \emptyset$;
\item if $D^+ \varphi(t, x)$ is a singleton, then $\varphi$ is differentiable at $(t, x)$;
\item $D^+ \varphi(t, x) = \convex D^\star \varphi(t, x)$;
\item $\Pi_x(D^+ \varphi(t,x)) = \nabla^+ \varphi(t, x)$;
\item if $C > 0$ is a semi-concavity constant for $\varphi$, a vector $p \in \mathbb R^d$ belongs to $\nabla^+ \varphi(t, x)$ if and only if
\[
\varphi(t, y) - \varphi(t, x) - p \cdot (y - x) \leq C \abs{y - x}^2
\] 
for every $y \in B$ such that $[x, y] \subset B$.
\end{enumerate}
\end{proposition}

\section{Exit-time optimal control problem}
\label{SecOptimalControl}

As a preliminary step for the study of our mean field game model, we consider in this section the optimal control problem solved by each agent of the game. We assume that each agent is subjected to a non-autonomous control system, the time-dependence of the dynamic being a consequence of the interaction between agents. The optimization criterion takes into account the time to reach a certain target set, considered as an \emph{exit}, and a cost on the position at which the agent reaches the exit. For this reason, our optimal control problem is qualified as ``exit-time''. In our setting, all agents evolve in a given compact set $\Omega \subset \mathbb R^d$ and the exit is assumed to be $\partial\Omega$. Notice that the particular case where the cost on the exit position is identically zero corresponds to the problem of reaching the target set in minimal time.

We start the section by providing a precise definition of our optimal control problem and recalling some well-known facts, in particular concerning its value function, while also exploiting the consequences of the first order optimality conditions from Pontryagin Maximum Principle. We then turn to the two main results of this section. The first one concerns the semi-concavity of the value function with respect to the space variable $x$ under weak assumptions on the smoothness of the dynamic with respect to time. Our second main result shows that the value function is differentiable along optimal trajectories, except possibly their endpoints.

\subsection{Definition, existence, and first properties}
\label{SecOC-First}

We consider control systems whose state equation is of the form
\begin{equation} \label{control system}
\left\{
\begin{aligned}
{\gamma}^\prime(t) & = k(t,\gamma(t)) u(t), &\quad \text{for a.e.\ }t \geq t_0, \\
\gamma(t_0) & = x_0,
\end{aligned}
\right.
\end{equation}
where $\gamma(t) \in \mathbb R^d$ is the state, the continuous function $k: \mathbb{R}^+ \times \mathbb{R}^d \to \mathbb{R}^+$ is called the \emph{dynamic} of the system, $t_0 \in \mathbb R^+$, $x_0 \in \mathbb R^d$, and $u: [t_0,\infty) \to \bar{B}(0,1)$ is a measurable function (which is called a \emph{control}).

We list some basic assumptions on the dynamic $k$:
\begin{hypothesis}
\label{lower bound on the dynamic} 0 < k_{\min}:=\inf k \leq k_{\max}:=\sup k < +\infty,
\end{hypothesis}
 \begin{hypothesis} 
\label{Hy1}
\exists L_1 > 0 \text{ such that }\abs{k(t,x_1) - k(t,x_2)} \leq L_1 \abs{x_1 - x_2} \quad \text{for all } x_1, x_2 \in \mathbb{R}^d \text{ and } t \in \mathbb{R}^+.
\end{hypothesis}
Notice that \eqref{Hy1} ensures the existence of a unique global solution to the state equation \eqref{control system} for any choice of $t_0$, $x_0$ and $u$. We denote the solution of \eqref{control system} by $\gamma^{t_0,x_0,u}$ and we call it an \emph{(admissible) trajectory} of the system, corresponding to the initial condition $\gamma(t_0)=x_0$ and to the control $u$.

 Let $\Omega$ be a compact domain in $\mathbb{R}^d$: for a given trajectory $\gamma=\gamma^{t_0,x_0,u}$ of \eqref{control system}, we set
$$\tau^{t_0,x_0,u}=\inf\{\tau \geq 0 \suchthat \gamma^{t_0,x_0,u}(t_0 + \tau) \in \partial\Omega\},$$
with the convention $\tau^{t_0,x_0,u}=+\infty$ if $\gamma^{t_0,x_0,u}(t_0 + \tau) \notin \partial\Omega$ for all $\tau \in \mathbb R^+$. This means that we consider $\partial\Omega$ as the target set. We call $\tau^{t_0,x_0,u}$ the \emph{exit time} of the trajectory. If $\tau^{t_0,x_0,u} < +\infty$, we set for simplicity
$$\gamma^{t_0,x_0,u}_{\tau}:=\gamma^{t_0,x_0,u}(t_0+\tau^{t_0,x_0,u})$$
to denote the point where the trajectory reaches the target $\partial\Omega$. As $k_{\min} >0$, one can see easily that, for every $(t_0, x_0) \in \mathbb R^+ \times \Omega$, there is always some control $u$ such that $\tau^{t_0,x_0,u} < +\infty$.

An optimal control problem consists of choosing the control strategy $u$ in the state equation \eqref{control system} in order to minimize a given functional. Let $g: \partial\Omega \to \mathbb{R}^+$ be a given continuous function. For every $(t_0, x_0) \in \mathbb R^+ \times \Omega$, we minimize the cost
\begin{equation} \label{controlquantity}
\tau^{t_0,x_0,u} + g(\gamma^{t_0,x_0,u}_{\tau})
\end{equation}
among all controls $u$. A control $u$ and the corresponding trajectory $\gamma^{t_0,x_0,u}$ are called \emph{optimal} for the point $x_0$ at time $t_0$ if $u$ minimizes \eqref{controlquantity}. Remark that optimal controls $u:[t_0,\infty) \to \bar{B}(0,1)$ are arbitrary for $t>t_0 + \tau^{t_0,x_0,u}$ and so, as a convention and unless otherwise stated, we choose $u(t) = 0$ for $t>t_0 + \tau^{t_0,x_0,u}$. Now, suppose that
\begin{hypothesis} \label{H2}
\exists \lambda \in \left(0, \frac{1}{k_{\max}}\right) \text{ s.t. } \abs{g(x) - g(y)} \leq \lambda \abs{x-y} \quad \text{for all } x, y \in \partial\Omega.
\end{hypothesis}
This is a standard assumption in exit-time optimal control problems with boundary costs (see, e.g., \cite[(8.6) and Remark 8.1.5]{CanSin} and \cite{Dweik}), its importance being the following property, whose proof is straightforward.

\begin{lemma}
\label{LemmTimePlusGIncreases}
Let $g: \partial\Omega \to \mathbb R^+$ satisfy \eqref{H2} and $\gamma: \mathbb R^+ \to \Omega$ be $k_{\max}$-Lipschitz. If $t_1, t_2 \in \mathbb R^+$ are such that $t_1 < t_2$ and $\gamma(t_1), \gamma(t_2) \in \partial\Omega$, then
\[
t_1 + g(\gamma(t_1)) < t_2 + g(\gamma(t_2)).
\]
\end{lemma}

Under assumptions \eqref{lower bound on the dynamic}, \eqref{Hy1} and \eqref{H2}, we have the following existence result, whose proof can be carried out by classical arguments (similar, for instance, to those given in \cite[Theorem 8.1.4]{CanSin} for the autonomous case).

\begin{proposition} \label{PropExistOptim}
For every $(t_0, x_0) \in \mathbb R^+ \times \Omega$, there exists an optimal control $u$ for the cost \eqref{controlquantity}.
\end{proposition}

We note that the condition \eqref{H2} is crucial for this result. Without this condition, one should replace the cost in \eqref{controlquantity} by $\inf\{t + g(\gamma(t_0 + t))\suchthat \gamma(t_0+t) \in \partial\Omega\}$.

Another easily obtained property, stated in the next result, is that the restriction of an optimal control is still optimal.

\begin{proposition}
\label{PropRestriction}
Let $(t_0, x_0) \in \mathbb R^+ \times \Omega$, $u$ be an optimal control for $x_0$, at time $t_0$, $\gamma = \gamma^{t_0, x_0, u}$, and $\tau_0 = \tau^{t_0, x_0, u}$. Then, for every $t \in [t_0, t_0 + \tau_0)$, $u|_{[t, t_0 + \tau_0]}$ is an optimal control for $\gamma(t)$, at time $t$.
\end{proposition}

The value function $\varphi:\mathbb{R}^+ \times \Omega \to \mathbb{R}^+$ of the above optimal control problem is defined by
\begin{equation} \label{value function}
\varphi(t,x)=\min\{\tau^{t,x,u} + g(\gamma^{t,x,u}_{\tau}) \suchthat u \text{ is a control}\}, \quad t \in \mathbb{R}^+,\; x \in \Omega.
\end{equation}
The first important fact is that the value function $\varphi$ satisfies the so-called \emph{dynamic programming principle} stated in the next lemma, which can be proved by standard techniques in optimal control (see, for instance, \cite[(8.4)]{CanSin}):

\begin{lemma} \label{dynamic programming principle}
For any $t_0 \in \mathbb{R}^+$, $x_0 \in \Omega$ and any control $u:[t_0, \infty) \to \bar{B}(0,1)$, we have
$$ \varphi(t_0,x_0) \leq t - t_0 + \varphi(t, \gamma^{t_0,x_0,u}(t)), \quad \text{for all } t \in [t_0, t_0 + \tau^{t_0,x_0,u}],$$
with equality if $u$ is optimal.
\end{lemma}

One can also use standard techniques in optimal control, similar to those, e.g., in \cite[Theorem 8.1.8]{CanSin}, to show that the value function $\varphi$ is a viscosity solution of a suitable partial differential equation.

\begin{proposition}
\label{MainTheoHJ}
The value function $\varphi$ is a viscosity solution of the following Hamilton--Jacobi equation 
\begin{equation}
\label{EqHJB}
-\partial_t \varphi(t, x) + k(t, x) \abs{\nabla \varphi(t, x)} - 1 = 0, \quad (t,x) \in \mathbb{R}^+ \times \accentset\circ{\Omega}.
\end{equation}
Moreover, one has $\varphi(t, x) = g(x)$ for every $(t, x) \in \mathbb R^+ \times \partial\Omega$.
\end{proposition}

Our next result shows that, if we consider points along optimal trajectories different from the endpoints, we can prove that the elements of $D^+\varphi$ also satisfy \eqref{EqHJB}. The proof of this result is omitted here since it can be easily obtained by adapting the arguments of the classical proof in the autonomous case (see, e.g., \cite[Proposition 8.1.9]{CanSin}).

\begin{proposition} \label{Hamilton in the interior}
Let $\gamma:[t_0,t_0+\tau_0] \to \Omega$ be an optimal trajectory for $x_0$, at time $t_0$, where $\tau_0=\tau^{t_0,x_0,u}$ and $u$ is the associated optimal control. Then, for every $t \in (t_0,t_0+\tau_0)$, we have
$$-p_t + k(t,\gamma(t)) \abs{p_x} - 1 = 0, \quad \text{for all } (p_t, p_x) \in D^+\varphi(t,\gamma(t)).$$
\end{proposition}

We now want to provide an upper bound on the optimal exit time $\tau^{t_0, x_0, u}$, where $u$ is an optimal control for $x_0$, at time $t_0$. To do so, we compare $\tau^{t_0, x_0, u}$ with the minimal time needed to reach $\partial\Omega$ from $x_0$ at time $t_0$. Let us introduce the \emph{minimal-time function} $T: \mathbb R^+ \times \Omega \to \mathbb R^+$ defined by
\begin{equation}
\label{EqMinimalTime}
T(t,x)=\inf\{\tau^{t,x,u} \suchthat u \text{ is a control}\}, \quad (t,x) \in \mathbb{R}^+ \times \Omega,
\end{equation}
which corresponds to taking $g = 0$ in \eqref{controlquantity}. An optimal control $u$ for the optimization problem of $T(t,x)$ is called a \emph{minimal-time control}. We then have the following result, whose proof, omitted here, is similar in spirit to that of \cite[Lemma 8.2.4]{CanSin}.

\begin{proposition}
\label{PropBoundTau}
For every $(t,x) \in \mathbb{R}^+ \times \Omega$, one has $T(t,x) \leq k_{\min}^{-1} \mathbf d(x,\partial\Omega)$. Moreover, if $u$ is an optimal control for \eqref{controlquantity} for $x$ at time $t$, then one has
$$\tau^{t,x,u} \leq \frac{1+\lambda k_{\max}}{1 - \lambda k_{\max}} T(t,x).$$
In particular,
$$\tau^{t,x,u} \leq \frac{k_{\min}^{-1}(1+\lambda k_{\max})}{1 - \lambda k_{\max}} \mathbf d(x, \partial\Omega).$$
\end{proposition}

The next property we present is the Lipschitz continuity of the value function $\varphi$. Lipschitz continuity of $\varphi$ with respect to $x$ can be proved following the same lines of the proof of \cite[Proposition 8.2.5]{CanSin}, whereas Lipschitz continuity with respect to $t$ can be obtained by standard arguments using the dynamic programming principle from Lemma \ref{dynamic programming principle}.

\begin{proposition} \label{LipValueFunction}
Let our system satisfy properties \eqref{lower bound on the dynamic}, \eqref{Hy1}, $\&$ \eqref{H2}. Then the value function $\varphi$ is Lipschitz continuous in $\mathbb{R}^+ \times \Omega$.
\end{proposition}

The last preliminary result we present in this subsection provides a lower bound on the variation in time of $\varphi$.

\begin{proposition}
\label{PropMonotoneOptimalTime}
Assume that \eqref{lower bound on the dynamic}, \eqref{Hy1}, and \eqref{H2} hold. Then there exists $c > 0$ depending only on $k_{\min}$, $k_{\max}$, $\diam(\Omega)$, $\lambda$, and $L_1$ such that, for every $x \in \Omega$ and $t_0, t_1 \in \mathbb R^+$ with $t_0 \neq t_1$,
\begin{equation}
\label{DtTauQGreaterThanMinusOne}
\frac{\varphi(t_1, x) - \varphi(t_0, x)}{t_1 - t_0} \geq c-1.
\end{equation}
\end{proposition}

\begin{proof}
Suppose, without loss of generality, that $t_0 < t_1$. Let $\gamma_1$ be an optimal trajectory for $x$, at time $t_1$, and $u_1$ be the associated optimal control. Define $\phi: [t_0,+\infty) \to \mathbb [t_1,+\infty)$ as a function satisfying
\begin{equation} \label{Cauchy}
\left\{
\begin{aligned}
\phi^\prime(t) & = \frac{k(t, \gamma_1(\phi(t)))}{k(\phi(t), \gamma_1(\phi(t)))}, \\
\phi(t_0) & = t_1.
\end{aligned}
\right.
\end{equation}
Notice that, since $k$ is only continuous with respect to its first variable, $\phi$ is not unique a priori. Set $\gamma_0(t) = \gamma_1(\phi(t))$ for all $t \geq t_0$. By construction of $\phi$, it is clear that there is a control $u_0$ such that $\gamma_0=\gamma^{t_0,x,u_0}$ (more precisely, $u_0(t)=u_1(\phi(t))$ for $t \geq t_0$). Moreover, we have $\tau_0:=\tau^{t_0,x,u_0} = \phi^{-1}(t_1 + \tau_1) - t_0$, where $\tau_1:=\tau^{t_1,x,u_1}$. So, $\phi(t_0 + \tau_0) = t_1 + \tau_1$ and $\phi(t_0 + \tau_0) + g(\gamma_0(t_0 + \tau_0)) = t_1 + \varphi(t_1,x)$. On the other hand, from \eqref{Cauchy}, it is easy to see that, for all $t, \bar{t} \geq t_0$, one has 
$$\int^{\phi(\bar{t})}_{\phi(t)} k(s,\gamma_1(s))\diff s = \int^{\bar{t}}_{t} k(s,\gamma_1(\phi(s)))\diff s.$$
Now, set 
$$G(\theta)=\int_{\theta}^{\phi(\bar{t})} k(s,\gamma_1(s))\diff s, \quad \forall \theta \in \mathbb{R}^+,$$
where we extend $\gamma_1$ to $\mathbb R^+$ by setting $\gamma_1(s) = \gamma_1(t_1)$ for $s \in [0, t_1)$. Then, using that $G$ is bi-Lipschitz, we have 
\begin{align*}
\abs{\phi(t) - t} & = \abs[\bigg]{G^{-1} \biggl(\int^{\bar{t}}_{t} k(s,\gamma_1(\phi(s)))\diff s\biggr) - G^{-1} \biggl(\int^{\phi(\bar{t})}_{t} k(s,\gamma_1(s))\diff s\biggr)} \displaybreak[0] \\
& \leq C \abs[\bigg]{\int^{\bar{t}}_{t} k(s,\gamma_1(\phi(s)))\diff s - \int^{\phi(\bar{t})}_{t} k(s,\gamma_1(s))\diff s}\displaybreak[0] \\
& \leq  C\biggl(\abs{\phi(\bar{t}) - \bar{t}} + \int^{\bar{t}}_{t} \abs{k(s,\gamma_1(\phi(s))) - k(s,\gamma_1(s))}\diff s\biggr) \displaybreak[0] \\
& \leq  C\abs{\phi(\bar{t}) - \bar{t}} + C \int^{\bar{t}}_{t} \abs{\phi(s) - s}\diff s, 
\end{align*}
where $C > 0$ denotes a constant depending only on $k_{\min}$, $k_{\max}$, and $L_1$, whose value may change from one line to the other. Using the fact that $\phi(t_0)=t_1 > t_0$, we infer that $\phi(t) > t$ for all $t \geq t_0$. Now, if $\bar{t}=t_0 + \tau_0$, we get, using Gronwall's inequality, that
\[
\phi(t) - t \leq C e^{C\abs{t_0+\tau_0-t}}(\phi(t_0 + \tau_0) - (t_0 +\tau_0)).
\]
Setting $t=t_0$, one has
\begin{equation*}
 c (t_1 - t_0) \leq \phi(t_0 + \tau_0) - (t_0 + \tau_0) = t_1 + \varphi(t_1, x) - g(\gamma_0(t_0 + \tau_0)) - t_0 - \tau_0,
\end{equation*}
where we use Proposition \ref{PropBoundTau} to provide an upper bound on $\tau_0$ and $c > 0$ only depends on $k_{\min}$, $k_{\max}$, $\diam(\Omega)$, $\lambda$, and $L_1$. Then
\[
(c-1) (t_1 - t_0) \leq \varphi(t_1, x) - g(\gamma_0(t_0 + \tau_0)) - \tau_0 = \varphi(t_1, x) - \varphi(t_0, x),
\]
as required.
\end{proof}

\begin{remark}
The analogue of Proposition \ref{PropMonotoneOptimalTime} was already proved in \cite[Proposition 4.5]{MazantiMinimal} for the minimal-time function $T$. Even though the proof of \cite{MazantiMinimal} could be easily adapted to our setting, it would require Lipschitz continuity of $k$ with respect to $t$. Our proof refines that of \cite{MazantiMinimal} and does not require such an assuption. The fact that $c$ does not depend on any Lipschitz behavior of $k$ with respect to $t$ will be a key property for the results in Sections \ref{SecLp} and \ref{SecExistenceLessRegular}.
\end{remark}

Proposition \ref{PropMonotoneOptimalTime} yields a lower bound on the time derivative of the value function $\varphi$, which can be used to obtain information on the gradient of $\varphi$ thanks to the Hamilton--Jacobi equation \eqref{EqHJB}.

\begin{corollary}
\label{CoroGradNotZero}
There exists $c>0$ (which only depends on $k_{\min}$, $k_{\max}$, $\diam(\Omega)$, $\lambda$, and $L_1$) such that $\partial_t \varphi(t, x) \geq c-1$ and $\abs{\nabla \varphi(t, x)} \geq c$ for all $(t,x) \in \mathbb{R}^+ \times \Omega$ where $\varphi$ is differentiable.
\end{corollary}

\subsection{Pontryagin Maximum Principle and its consequences}
\label{SecOC-PMP}

In this subsection, we use the necessary optimality conditions of Pontryagin Maximum Principle to obtain further properties of optimal trajectories and the value function $\varphi$. In addition to \eqref{lower bound on the dynamic}, \eqref{Hy1}, and \eqref{H2}, we also assume that
\begin{hypothesis} \label{Smoothness of the boundary}
\partial\Omega \text{ is of class }C^{1,1},
\end{hypothesis}
\begin{hypothesis} \label{Smoothness of the gradient of the dynamic}
\nabla k \in C(\mathbb{R}^+ \times \Omega),
\end{hypothesis} 
\begin{hypothesis} \label{Smoothness of the boundary cost}
g \in C^1(\partial\Omega).
\end{hypothesis}
In order to state a version of Pontryagin Maximum Principle, we start with a preliminary result (see \cite[Lemma 8.4.2]{CanSin}).

\begin{lemma} \label{8.4.2}
Given $z \in \partial\Omega$, let $\mathbf{n}$ be the outer normal to $\partial\Omega$ at $z$. Then, for every $t \in \mathbb{R}^+$, there exists a unique $\mu > 0$ such that $k(t,z) \abs{\nabla g(z) - \mu \mathbf{n}} - 1 = 0$.
\end{lemma}

We are now ready to state Pontryagin Maximum Principle for this control problem.

\begin{proposition} \label{maximum principle}
Let properties \eqref{lower bound on the dynamic}, \eqref{Hy1}, \eqref{H2}, \eqref{Smoothness of the boundary}, \eqref{Smoothness of the gradient of the dynamic}, and \eqref{Smoothness of the boundary cost} hold, let $(t_0,x_0) \in \mathbb{R}^+ \times \Omega$ and let $\bar{u}$ be an optimal control for $x_0$, at time $t_0$. Set for simplicity
$$\gamma := \gamma^{t_0,x_0,\bar{u}},\qquad \tau_0:= \tau^{t_0,x_0,\bar{u}},\qquad z:=\gamma^{t_0,x_0,\bar{u}}_\tau,$$
and denote by $\mathbf{n}$ the outer normal to $\partial\Omega$ at $z$. Let $\mu > 0$ be such that $k(t_0 + \tau_0,z) \abs{\nabla g(z) - \mu\mathbf{n}} - 1 = 0$ ($\mu$ is uniquely determined by the previous lemma). Let $p: [t_0,t_0 + \tau_0] \to \mathbb{R}^d$ be the solution to the system
\begin{equation} \label{8.60}
\begin{cases}
p^\prime(t)= - \nabla k (t,\gamma(t)) \bar{u}(t) \cdot p (t),\\ 
p(t_0 + \tau_0) = \nabla g(z) - \mu \mathbf{n}.
\end{cases}
\end{equation}
Then, for a.e.\ $t \in [t_0,t_0+\tau_0]$,
$$- p(t) \cdot \bar{u}(t) = \max_{u \in \bar{B}(0,1)} - p(t) \cdot u.$$
\end{proposition}

We refer the reader to \cite[Lemma 8.4.2 and Theorem 8.4.3]{CanSin} for proofs of the above results. Even though the proofs in \cite{CanSin} only consider the case of autonomous dynamics, their extension to our non-autonomous setting is straightforward.

As a consequence of Proposition \ref{maximum principle}, we get the following.

\begin{proposition}
\label{RegCurve}
Let $(t_0, x_0) \in \mathbb R^+ \times \Omega$ and $\bar u$, $\gamma$, $\tau_0$, and $p$ be as in the statement of Proposition \ref{maximum principle}. Then $p$ is non-zero in $[t_0, t_0 + \tau_0]$, $\bar u(t) = -\frac{p(t)}{\abs{p(t)}}$ for every $t \in [t_0, t_0 + \tau_0]$, $\bar u$ is $L_1$-Lipschitz continuous on $[t_0,t_0+\tau_0]$ and satisfies
\[\bar u^\prime(t) = - \nabla k(t, \gamma(t)) + \bar{u}(t) \cdot \nabla k(t, \gamma(t)) \bar{u}(t) \quad \text{for a.e.\ } t \in [t_0, t_0 + \tau_0],\]
and $\gamma$ is $C^1$ on $[t_0,t_0+\tau_0]$. Moreover, $\gamma \in C^{1,1}([t_0, t_0 + \tau_0], \Omega)$ as soon as $k$ is Lipschitz in $t$.
\end{proposition}

The proof of Proposition \ref{RegCurve} is omitted here since it can be obtained using analogous arguments to \cite[Lemma 4.13 and Corollary 4.14]{MazantiMinimal}, which provide similar results for the minimal-time problem \eqref{EqMinimalTime}.

From now on, we suppose also that
\begin{hypothesis} \label{Lipschitz regularity of the gradient with respect TO x}
\exists L_2 > 0 \text{ such that }\abs{\nabla k(t,x_0) - \nabla k(t,x_1)} \leq L_2 \abs{x_0 - x_1} \quad \text{for all } x_0, x_1 \in \Omega,\;t \in \mathbb{R}^+.
\end{hypothesis}
Then, under the assumptions of Proposition \ref{RegCurve} and \eqref{Lipschitz regularity of the gradient with respect TO x}, $(\gamma,\bar{u})$ is the unique solution on $[t_0, t_0 + \tau_0]$ of
\begin{equation}\label{SystTrajCont}
\left\{
\begin{aligned}
\gamma^\prime(t)& =k(t,\gamma(t)) u(t),\\
u^{\prime}(t) & = - \nabla k(t, \gamma(t)) + u(t) \cdot \nabla k(t, \gamma(t)) u(t),\\
\gamma(t_0)& =x_0,\\
u(t_0)& =\bar{u}(t_0).
\end{aligned}
\right.
\end{equation}
Now, let us introduce the following lemma, which shows that the uniform limit of optimal trajectories is an optimal trajectory. Its proof, omitted here, follows the same lines as the proof of \cite[Theorem 8.1.7]{CanSin}.

\begin{lemma} \label{Uniform convergence trajectory}
Assume that \eqref{lower bound on the dynamic}---\eqref{Smoothness of the boundary cost} hold. Let $(t_n,x_n)_n$ be a sequence in $\mathbb{R}^+ \times \Omega$ such that $t_n \to t$ and $x_n \to x$. For each $n$, let $\gamma_n$ be an optimal trajectory for $x_n$, at time $t_n$, and $u_n$ be the associated optimal control. Then, up to extracting subsequences, there exist $\gamma$ and $u$ such that $\gamma_n \rightarrow \gamma$ and $u_n \rightarrow u$ uniformly, where $\gamma$ is an optimal trajectory for $x$, at time $t$, and $u$ is its associated optimal control.
\end{lemma}

On the other hand, we have the following result about the uniqueness of optimal control at any interior point of an optimal trajectory. 

\begin{proposition}
\label{PropSingletonAfterStartingTime}
Assume that \eqref{lower bound on the dynamic}---\eqref{Lipschitz regularity of the gradient with respect TO x} hold. Let $\gamma$ be an optimal trajectory for $x_0$ at time $t_0$, and set $\tau_0 = \tau^{t_0, x_0,u}$, where $u$ is the associated optimal control. Then, for every $t \in (t_0, t_0 + \tau_0)$, $u$ is the unique optimal control for $\gamma(t)$, at time $t$.
\end{proposition}

\begin{proof}
Fix $t \in (t_0, t_0 + \tau_0)$ and let $v$ be an optimal control for $x:=\gamma(t)$, at time $t$. Set 
$$\widetilde{u}(s) = 
\begin{dcases*}
u(s), & if $s < t$, \\
v(s), & if $s \geq t$.
\end{dcases*}
$$
Then $\widetilde{u}$ is an optimal control for $x_0$, at time $t_0$. Indeed, using the optimality of $v$, we have $\varphi(t_0,x_0) \leq \tau^{t_0,x_0,\widetilde{u}} + g(\gamma^{t_0,x_0,\widetilde{u}}_\tau) = t - t_0 + \varphi(t,x)$. On the other hand, since $u$ is optimal, one obtains from Lemma \ref{dynamic programming principle} that $\varphi(t_0, x_0) = t - t_0 + \varphi(t, x)$. Then $\varphi(t_0, x_0) = \tau^{t_0,x_0,\widetilde{u}} + g(\gamma^{t_0,x_0,\widetilde{u}}_\tau)$, and so the control $\widetilde{u}$ is optimal. Hence, by Proposition \ref{RegCurve}, $\widetilde{u}$ is continuous, which proves that $u(t) = v(t):=q$. The fact that $u(s)=v(s)$, for all $s \geq t$, follows from the uniqueness of solutions to the system \eqref{SystTrajCont} with initial conditions $\gamma(t)=x$ and $u(t)=q$.
\end{proof}

Given an optimal trajectory $\gamma$ for $x_0$ at time $t_0$, we will say that $p$ is a \emph{dual arc} associated with $\gamma$ if it satisfies the properties of Proposition \ref{maximum principle}, that is, if it solves \eqref{8.60}. Our next result states that the dual arc $p$ is included in the superdifferential of the value function $\varphi$ with respect to $x$, $\nabla^+ \varphi$.

\begin{proposition} \label{propagation of surdiff}
Under the assumptions of Proposition \ref{maximum principle}, the arc $p$ solution of \eqref{8.60} satisfies
$$p(t) \in \nabla^+ \varphi(t,\gamma(t)), \quad \text{for all } t \in [t_0,t_0+\tau_0).$$
\end{proposition}

The proof of Proposition \ref{propagation of surdiff} can be obtained by easily adapting the proof of \cite[Theorem 8.4.4]{CanSin} to our non-autonomous setting, and is omitted here for simplicity. Similarly, one can obtain an analogous property for the subdifferential by an immediate adaptation of the techniques from \cite[Theorem 7.3.4]{CanSin}.

\begin{proposition} \label{propagation of subdiff}
Let $u$ be an optimal control for $(t_0,x_0) \in \mathbb{R}^+ \times \Omega$ and $\gamma$ be its associated optimal trajectory. Let $p:[t_0,t_0 + \tau_0] \to \mathbb{R}^d$ be any solution of the adjoint equation
\begin{equation} \label{adjoint equation}
p^\prime(t)=-\nabla k(t,\gamma(t)) u(t) \cdot p(t), \quad t \in [t_0,t_0 + \tau_0],
\end{equation}
 where $\tau_0=\tau^{t_0,x_0,u}$. Suppose that $p(t_0) \in {\nabla}^- \varphi(t_0,x_0)$, then $p(t) \in {\nabla}^- \varphi(t,\gamma(t))$, for all $t \in [t_0,t_0 + \tau_0)$.
\end{proposition}

As a consequence of the previous results, we can show that the existence of $\nabla\varphi$ at some point $(t_0, x_0)$ is sufficient to ensure uniqueness of the optimal trajectory for $x_0$, at time $t_0$.

\begin{proposition}
\label{UniqueOptimalTrajectory}
Let $(t_0, x_0)  \in \mathbb R^+ \times \Omega$ and assume that $\nabla\varphi(t_0, x_0)$ exists. Then there exists a unique trajectory $\gamma$ which is optimal for $x_0$, at time $t_0$.
\end{proposition}

\begin{proof}
Assume that $\gamma_1, \gamma_2$ are optimal trajectories for $x_0$, at time $t_0$, and denote the respective optimal controls by $u_1, u_2$. For $i \in \{1, 2\}$, write $\tau_i = \tau^{t_0, x_0, u_i}$ and let $p_i: [t_0, t_0 + \tau_i] \to \mathbb R^d$ be a dual arc associated with $\gamma_i$. By Proposition \ref{RegCurve}, $p_i$ is non-zero and $u_i$ is Lipschitz continuous on $[t_0, t_0 + \tau_i]$, with $u_i(t) = -\frac{p_i(t)}{\abs{p_i(t)}}$ for every $t \in [t_0, t_0 + \tau_i]$. By Proposition \ref{propagation of surdiff}, $p_i(t) \in \nabla^+ \varphi(t, \gamma_i(t))$ for every $t \in [t_0, t_0 + \tau_i)$. In particular, since $\nabla\varphi(t_0, x_0)$ exists, one has $p_1(t_0) = p_2(t_0) = \nabla\varphi(t_0, x_0)$, yielding that $\nabla\varphi(t_0, x_0) \neq 0$ and $u_1(t_0) = u_2(t_0) = \frac{\nabla\varphi(t_0, x_0)}{\abs{\nabla\varphi(t_0, x_0)}}$. This means that both $(\gamma_1, u_1)$ and $(\gamma_2, u_2)$ solve \eqref{SystTrajCont} with the same initial conditions $\gamma_1(t_0) = \gamma_2(t_0) = x_0$ and $u_1(t_0) = u_2(t_0) = \frac{\nabla\varphi(t_0, x_0)}{\abs{\nabla\varphi(t_0, x_0)}}$, yielding, by uniqueness of the solutions of \eqref{SystTrajCont}, that $\gamma_1 = \gamma_2$.
\end{proof}

To conclude this subsection, we prove that, when optimal trajectories are close enough to the boundary, they always move towards the boundary, in the sense that the scalar product between the direction of the trajectory and some normal direction is lower bounded by a positive constant. To do so, we make use of the \emph{signed distance} to $\partial\Omega$, which is the function $d^{\pm}: \mathbb R^d \to \mathbb R$ defined for $x \in \mathbb R^d$ by
\begin{equation}
\label{EqDefiDPM}
d^{\pm}(x) = 
\begin{dcases*}
\mathbf d(x, \partial\Omega), & if $x \notin \Omega$, \\
-\mathbf d(x, \partial\Omega), & otherwise.
\end{dcases*}
\end{equation}
Recall that, thanks to \eqref{Smoothness of the boundary}, $d^{\pm}$ is $1$-Lipschitz on $\mathbb R^d$, $C^{1, 1}$ in a neighborhood of the boundary, and, if $x \in \partial\Omega$, then $\nabla d^{\pm}(x)$ is the outer normal to $\partial\Omega$ at $x$ (see, e.g., \cite{Delfour1994Shape}).

\begin{proposition}
\label{PropLowerBoundScalarProduct}
There exist $c > 0$ (depending only on $k_{\min}$, $k_{\max}$, and $\lambda$) and $\delta > 0$ (depending only on $k_{\min}$, $k_{\max}$, $\lambda$, $\diam(\Omega)$, and the curvature of $\partial\Omega$) such that, for every $(t_0, x_0) \in \mathbb R^+ \times \accentset\circ\Omega$, if $u$ is an optimal control for $x_0$, at time $t_0$, $\gamma := \gamma^{t_0, x_0, u}$ is the corresponding optimal trajectory, and $\tau_0 := \tau^{t_0, x_0, u}$, then, for every $t \in [t_0, t_0 + \tau_0]$ such that $\mathbf d(\gamma(t), \partial\Omega) \leq \delta$, one has
\begin{equation}
\label{LowerBoundScalarProduct}
\nabla d^{\pm}(\gamma(t)) \cdot u(t) \geq c.
\end{equation}
In particular, if $\mathbf d(x_0, \partial\Omega) \leq \delta$ and $\nabla\varphi(t_0, x_0)$ exists, then $\nabla\varphi(t_0, x_0) \neq 0$ and
\begin{equation}
\label{LowerBoundScalarProduct-Varphi}
-\nabla d^{\pm}(x_0) \cdot \frac{\nabla\varphi(t_0, x_0)}{\abs{\nabla\varphi(t_0, x_0)}} \geq c.
\end{equation}
\end{proposition}

\begin{proof}
Let $p: [t_0, t_0 + \tau_0] \to \mathbb R^d$ be a dual arc associated with $\gamma$. By Proposition \ref{RegCurve}, $p$ is non-zero and $u$ is Lipschitz continuous on $[t_0, t_0 + \tau_0]$, with $u(t) = -\frac{p(t)}{\abs{p(t)}}$ for every $t \in [t_0, t_0 + \tau_0]$. In particular, $u(t_0 + \tau_0) = \frac{\mu \mathbf{n} - \nabla g(z)}{\abs{\mu \mathbf{n} - \nabla g(z)}}$, where $z$, $\mathbf n$, and $\mu$ are as in the statement of Proposition \ref{maximum principle}.

We first prove \eqref{LowerBoundScalarProduct} at the final time $t_0 + \tau_0$. Recalling that $k(t_0 + \tau_0, z) \abs{\nabla g(z) - \mu \mathbf{n}} = 1$, one has
\[
\frac{1}{k(t_0 + \tau_0, z)^2} = \abs{\nabla g(z) - \mu \mathbf{n}}^2 = \abs{\nabla g(z)}^2 - 2 \mu \nabla g(z) \cdot \mathbf{n} + \mu^2,
\]
and thus
\begin{equation}
\label{MuMinusNablaGZCdotN}
2 \mu (\mu - \nabla g(z) \cdot \mathbf n) = \frac{1}{k(t_0 + \tau_0, z)^2} - \abs{\nabla g(z)}^2 + \mu^2.
\end{equation}
On the other hand, one also has that
\[
\frac{1}{k(t_0 + \tau_0, z)} = \abs{\nabla g(z) - \mu \mathbf{n}} \geq \mu - \abs{\nabla g(z)},
\]
and thus
\[\mu \leq \frac{1}{k(t_0 + \tau_0, z)} + \abs{\nabla g(z)}.\]
Combining this with \eqref{MuMinusNablaGZCdotN}, one gets that
\begin{align*}
\mu - \nabla g(z) \cdot \mathbf n & = \frac{\frac{1}{k(t_0 + \tau_0, z)^2} - \abs{\nabla g(z)}^2 + \mu^2}{2 \mu} > \frac{\frac{1}{k(t_0 + \tau_0, z)^2} - \abs{\nabla g(z)}^2}{2 \left(\frac{1}{k(t_0 + \tau_0, z)} + \abs{\nabla g(z)}\right)} \\
& = \frac{1}{2} \left(\frac{1}{k(t_0 + \tau_0, z)} - \abs{\nabla g(z)}\right) \geq \frac{1}{2} \left(\frac{1}{k_{\max}} - \lambda\right) > 0.
\end{align*}
Hence, recalling that $\abs{\mu \mathbf{n} - \nabla g(z)} = \frac{1}{k(t_0 + \tau_0, z)} \leq \frac{1}{k_{\min}}$, one obtains that
\begin{equation}
\label{LowerBoundScalarProduct-FinalTime}
\nabla d^{\pm}(z) \cdot u(t_0 + \tau_0) = \mathbf n \cdot \frac{\mu \mathbf{n} - \nabla g(z)}{\abs{\mu \mathbf{n} - \nabla g(z)}} = \frac{\mu - \nabla g(z) \cdot \mathbf{n}}{\abs{\mu \mathbf{n} - \nabla g(z)}} \geq \frac{k_{\min}}{2} \left(\frac{1}{k_{\max}} - \lambda\right),
\end{equation}
which corresponds to \eqref{LowerBoundScalarProduct} at the final time $t_0 + \tau_0$.

Now, let $\delta_0 > 0$ be such that $d^{\pm}$ is $C^{1, 1}$ on the set $\{x \in \mathbb R^d \suchthat \mathbf d(x, \partial\Omega) \leq \delta_0\}$ and $L_d > 0$ be a Lipschitz constant for $\nabla d^{\pm}$ on this set. By Proposition \ref{RegCurve}, $u$ is $L_1$-Lipschitz on $[t_0, t_0 + \tau_0]$. Take
\begin{align*}
c & = \frac{k_{\min}}{4} \left(\frac{1}{k_{\max}} - \lambda\right), \\
\delta & = \min\left\{\delta_0, \frac{k_{\min}^2 (1 - \lambda k_{\max})^2}{4 k_{\max} (1 + \lambda k_{\max}) (L_d k_{\max} + L_1)}\right\}.
\end{align*}
Let $t \in [t_0, t_0 + \tau_0)$ be such that $\mathbf d(\gamma(t), \partial\Omega) \leq \delta$. By Proposition \ref{PropRestriction}, $u|_{[t, t_0 + \tau_0]}$ is an optimal control for $\gamma(t)$, at time $t$, and thus, by Proposition \ref{PropBoundTau}, one obtains that
\[
t_0 + \tau_0 - t = \tau^{t, \gamma(t), u|_{[t, t_0 + \tau_0]}} \leq \frac{(1 + \lambda k_{\max}) \delta}{(1 - \lambda k_{\max}) k_{\min}} \leq \frac{k_{\min} (1 - \lambda k_{\max})}{4 k_{\max} (L_d k_{\max} + L_1)}.
\]
Hence, by the previous inequality and \eqref{LowerBoundScalarProduct-FinalTime}, one has
\begin{align*}
\nabla d^{\pm}(\gamma(t)) \cdot u(t) & = \nabla d^{\pm}(z) \cdot u(t_0 + \tau_0) + \left(\nabla d^{\pm}(\gamma(t)) - \nabla d^{\pm}(z)\right) \cdot u(t) \\
& \hphantom{{} = \nabla d^{\pm}(z) \cdot u(t_0 + \tau_0)} + \nabla d^{\pm}(z) \cdot \left(u(t) - u(t_0 + \tau_0)\right) \\
& \geq \frac{k_{\min}}{2} \left(\frac{1}{k_{\max}} - \lambda\right) - L_d \abs{\gamma(t) - z} - L_1 \abs{t_0 + \tau_0 - t} \\
& \geq \frac{k_{\min}}{2} \left(\frac{1}{k_{\max}} - \lambda\right) - (L_d k_{\max} + L_1) (t_0 + \tau_0 - t) \\
& \geq \frac{k_{\min}}{4} \left(\frac{1}{k_{\max}} - \lambda\right) = c,
\end{align*}
concluding the proof of \eqref{LowerBoundScalarProduct}.

Concerning the last part of the statement, notice that, as a consequence of Proposition \ref{propagation of surdiff} and the fact that $\nabla \varphi(t_0, x_0)$ exists, one deduces that $p(t_0) = \nabla\varphi(t_0, x_0)$, yielding that $\nabla\varphi(t_0, x_0) \neq 0$ and $u(t_0) = -\frac{\nabla\varphi(t_0, x_0)}{\abs{\nabla\varphi(t_0, x_0)}}$. Hence \eqref{LowerBoundScalarProduct-Varphi} follows from \eqref{LowerBoundScalarProduct}.
\end{proof}

\subsection{Sharp semi-concavity}
\label{Semi-concavity}

In this subsection, we investigate the hypotheses under which the value function $\varphi$ of our exit-time optimal control problem is semi-concave with respect to $x$. A semi-concavity result for autonomous exit-time optimal control problems is provided in \cite[Theorem 8.2.7]{CanSin} and, up to performing a classical state augmentation technique to regard \eqref{control system} as an autonomous system (which consists of considering $z(t) = (t, \gamma(t))$ as the state), one can readily obtain the semi-concavity of $\varphi$ with respect to $(t, x)$ provided that $k \in C^{1, 1}(\mathbb R^+ \times \Omega)$.

By looking at the proof of \cite[Theorem 8.2.7]{CanSin}, one can also notice that immediate adaptations of the proof allow one to obtain semi-concavity of $\varphi$ with respect to $x$ as soon as $k$ is $C^{1, 1}$ with respect to $x$ and Lipschitz continuous in $t$. It turns out that, in our setting, we can refine the proof of \cite[Theorem 8.2.7]{CanSin} to show that semi-concavity of $\varphi$ with respect to $x$ can be obtained under a weaker assumption on the behavior of $k$ with respect to $t$, namely that $\partial_t k$ is lower bounded. This is the main result of this subsection, proved in Theorem \ref{Theorem semiconcavity}.

We note that semi-concavity of $\varphi$ is related not only to the regularity of $k$, but also to the smoothness of the target $\partial\Omega$. We also make use of the fact that the distance function $\mathbf d(\cdot, \overline{\mathbb R^d \setminus \Omega})$ is semi-concave in $\Omega$, which is a consequence of \eqref{Smoothness of the boundary} (or more generally, a uniform exterior ball condition on $\Omega$). Notice that this distance function coincides with the value function $\varphi$ in the particular case $k \equiv 1$ and $g \equiv 0$, justifying the importance of its properties in the proof of Theorem \ref{Theorem semiconcavity}.

We first introduce the following estimates on the trajectories, which will be repeatedly used in our analysis.

\begin{proposition} \label{estimates trajectory}
Assume that \eqref{Hy1} and \eqref{Lipschitz regularity of the gradient with respect TO x} hold and let $t_0, t \in \mathbb R^+$. Then there exists $c > 0$, depending only on $t - t_0$, $L_1$, and $L_2$, such that, for every $x_0, x_1 \in \Omega$ and every control $u:[t_0,\infty) \to \bar{B}(0,1)$, one has
$$\abs{\gamma^{t_0,x_0,u}(t) - \gamma^{t_0,x_1,u}(t)} \leq c \abs{x_0 - x_1}$$
and 
$$\abs*{\gamma^{t_0,x_0,u}(t) + \gamma^{t_0,x_1,u}(t) - 2 \gamma^{t_0,\frac{x_0 + x_1}{2},u}(t)} \leq c\abs{x_0 - x_1}^2.$$
\end{proposition}

Proposition \ref{estimates trajectory} can be proved exactly as in \cite[Lemma 7.1.2]{CanSin} and thus its proof is omitted here.

To prove the semi-concavity of $\varphi$, we need to assume that \eqref{lower bound on the dynamic}, \eqref{Hy1}, \eqref{H2}, \eqref{Smoothness of the boundary}, \eqref{Smoothness of the gradient of the dynamic}, \eqref{Smoothness of the boundary cost}, and \eqref{Lipschitz regularity of the gradient with respect TO x} are satisfied. In addition, we suppose that there exists a constant $\ell > 0$ such that, for every $x \in \Omega$, $t \mapsto k(t, x)$ is absolutely continuous and, almost everywhere in $t \in \mathbb R^+$,
\begin{hypothesis} \label{lower bound of the derivative of the dynamic with respect to time}
\partial_t k \geq -\ell.
\end{hypothesis} 
Moreover, we assume that
\begin{hypothesis} \label{boundary cost semiconcave}
g \text{ is semi-concave on } \partial\Omega.
\end{hypothesis}
Then, we have the following result.

\begin{theorem} \label{Theorem semiconcavity}
The value function $\varphi$ is semi-concave w.r.t.\ $x$, and its semi-concavity constant depends only on $\lambda$, $k_{\min}$, $k_{\max}$, $\kappa$, $L_1$, $L_2$, $M$, and $\ell$, where $\kappa$ is a bound on the curvatures of $\partial\Omega$ and $M$ is the semi-concavity constant of $g$.
\end{theorem}
 
\begin{proof}
Along this proof, $c$ is used to denote positive constants depending only on $\lambda$, $k_{\min}$, $k_{\max}$, $\kappa$, $L_1$, $L_2$, $M$, and $\ell$, and the value of these constants may change from one expression to another. Some parts of this proof, in particular Case 1 and the first arguments in Case 2, are treated exactly as in the corresponding parts of the proof of \cite[Theorem 8.2.7]{CanSin}, and we only detail them here for the sake of completeness.

Let $(t_0, x) \in \mathbb R^+ \times \Omega$. For simplicity of exposition, we suppose that $t_0 = 0$. Let $h \in \mathbb R^d$ be such that $x - h, x + h \in \Omega$ and $u$ be an optimal control for $x$, at time $0$. We consider the trajectories $\gamma^{0, x, u}$, $\gamma^{0, x - h, u}$, and $\gamma^{0, x + h, u}$, and split the proof into cases according to which of these trajectories arrives first at $\partial\Omega$.

\case{Case 1: $\tau_0:=\tau^{0,x,u} \leq \min\{\tau^{0,x-h,u},\tau^{0,x+h,u}\}$}.

Since $u$ is optimal for $x$, at time $0$, it follows from Lemma \ref{dynamic programming principle} that
\begin{equation}
\label{EQSC-F}
\varphi(0,x-h) + \varphi(0,x+h) - 2 \varphi(0,x) \leq \varphi(\tau_0,x^-) + \varphi(\tau_0,x^+) - 2 g(\gamma^{0,x,u}_{\tau}),
\end{equation}
where
$$ x^+:=\gamma^{0,x+h,u}(\tau_0) \qquad \text{and} \qquad x^-:=\gamma^{0,x-h,u}(\tau_0).$$
Let $u^+$, $u^-$ be two optimal controls for $x^+$ and $x^-$, at time $\tau_0$, respectively, and define $y^\pm = \gamma^{\tau_0,x^\pm,u^\pm}_{\tau}$ and $\tau^\pm:= \tau^{\tau_0,x^\pm,u^\pm}$. Then
\begin{equation}
\label{EQSC-E}
\varphi(\tau_0,x^-) + \varphi(\tau_0,x^+) - 2 g(\gamma^{0,x,u}_{\tau})=\tau^- + g(y^-) + \tau^+ + g(y^+) - 2 g(\gamma^{0,x,u}_{\tau}).
\end{equation}
Yet, by Proposition \ref{PropBoundTau}, we have
\begin{equation}
\label{EQSC-D}
\tau^\pm \leq c \mathbf d(x^\pm,\overline{\mathbb{R}^d \setminus \Omega}).
\end{equation}
As the distance function $\mathbf d(\cdot, \overline{\mathbb{R}^d \setminus \Omega})$ is $1$-Lipschitz, semi-concave in $\bar{\Omega}$, and its semi-concavity constant is bounded by $\kappa$, and taking into account that $\gamma^{0,x,u}_{\tau} \in\partial\Omega$, we obtain that 
\begin{equation}
\label{EQSC-C}
\begin{aligned}
\mathbf d(x^+,\overline{\mathbb{R}^d \setminus \Omega}) + \mathbf d(x^-,\overline{\mathbb{R}^d \setminus \Omega}) & = \mathbf d(x^+,\overline{\mathbb{R}^d \setminus \Omega}) + \mathbf d(x^-,\overline{\mathbb{R}^d \setminus \Omega}) - 2 \mathbf d \biggl(\frac{x^+ + x^-}{2},\overline{\mathbb{R}^d \setminus \Omega}\biggr) \\
& \hphantom{{} = {}} {} + 2\biggl( \mathbf d \biggl(\frac{x^+ + x^- }{2},\overline{\mathbb{R}^d \setminus \Omega}\biggr) - \mathbf d(\gamma^{0,x,u}_{\tau},\overline{\mathbb{R}^d \setminus \Omega})\biggr) \\
& \leq c\abs{x^+ - x^-}^2 + \abs{x^+ + x^- - 2\gamma^{0,x,u}_{\tau}} \leq c \abs{h}^2,
\end{aligned}
\end{equation}
where the last inequality follows from Proposition \ref{estimates trajectory}. On the other hand, from the assumptions on $g$, we have
\begin{equation}\label{090}
\begin{split}
 g(y^+) + g(y^-)  - 2 g(\gamma^{0,x,u}_{\tau}) & =  g(y^+) + g(y^-) - 2 g\biggl(\frac{y^+ + y^-}{2}\biggr) + 2 \biggl(g \biggl(\frac{y^+ + y^-}{2}\biggr) - g(\gamma^{0,x,u}_{\tau})\biggr) \\
 & \leq c\left(\abs{y^+ - y^-}^2 + \abs{y^+ + y^- - 2\gamma^{0,x,u}_{\tau}}\right). 
\end{split}
\end{equation}
Yet, 
$$ \abs{y^+ -y^-} \leq \abs{y^+ - x^+} + \abs{x^+ - x^-} + \abs{x^--y^-} \leq \abs{y^+ - x^+} + \abs{x^- - y^-} + c \abs{h}.$$
In addition, we have 
$$\abs{y^\pm - x^\pm} = \abs[\bigg]{ \int_{\tau_0}^{\tau_0 + \tau^\pm} k(s,\gamma^{\tau_0,x^\pm,u^\pm}(s)) u^\pm (s)\diff s } \leq k_{\max} \tau^\pm \leq c \abs{h}^2,$$
which implies that 
\begin{equation}
\label{EQSC-B}
\abs{y^+ -y^-} \leq c \abs{h}.
\end{equation}
For the second term in \eqref{090}, we have  
\begin{equation}
\label{EQSC-A}
\abs{y^+ + y^- - 2 \gamma^{0,x,u}_{\tau}} \leq \abs{y^+ - x^+} + \abs{x^+ + x^- - 2 \gamma^{0,x,u}_{\tau}} + \abs{x^- - y^-} \leq c\abs{h}^2.
\end{equation}
Consequently, inserting \eqref{EQSC-B} and \eqref{EQSC-A} into \eqref{090} and combining this with \eqref{EQSC-F}, \eqref{EQSC-E}, \eqref{EQSC-D}, and \eqref{EQSC-C}, we conclude that 
$$\varphi(0,x-h) + \varphi(0,x+h) - 2 \varphi(0,x) \leq c\abs{h}^2.$$

\case{Case 2: $\tau_0:=\tau^{0,x-h,u} \leq \min\{\tau^{0,x,u},\tau^{0,x+h,u}\}$}.
 
It suffices to treat this case to conclude the proof, since the other remaining case $\tau^{0,x+h,u} \leq \min\{\tau^{0,x,u},\tau^{0,x-h,u}\}$ is identical up to exchanging $h$ and $-h$. Let
$$ x_0=\gamma^{0,x-h,u}(\tau_0),\qquad x_1=\gamma^{0,x,u}(\tau_0),\qquad x_2=\gamma^{0,x+h,u}(\tau_0).$$
By Lemma \ref{dynamic programming principle}, we have 
\begin{equation}
\label{EQSC-H}
\varphi(0,x-h) + \varphi(0,x+h) - 2 \varphi(0,x) \leq \varphi(\tau_0,x_2) - 2 \varphi(\tau_0,x_1) + g(x_0),
\end{equation}
By Proposition \ref{PropRestriction}, $u$ is also an optimal control starting from $x_1$, at time $\tau_0$, with $\tau_1:=\tau^{\tau_0,x_1,u}= \tau^{0,x,u} - \tau_0$. As $x_0 \in \partial\Omega$, then, by Propositions \ref{PropBoundTau} and \ref{estimates trajectory}, we get that
\begin{equation}
\label{EQSC-G}
\tau_1 \leq c \mathbf d(x_1,\overline{\mathbb{R}^d \setminus \Omega}) \leq c \abs{x_1 - x_0} \leq c\abs{h}.
\end{equation}
Let $u^\star$ be the control defined for $t \geq \tau_0$ by $u^\star(t):=u(\frac{t + \tau_0}{2})$ and consider the trajectory $\gamma^{\tau_0, x_2, u^\star}$. We split the remainder of the proof into two cases requiring separate analyses.

\case{Case 2(a): $\tau_1 < \frac{\tau^{\tau_0,x_2,u^\star}}{2}$}.

By Lemma \ref{dynamic programming principle},
\begin{equation}
\label{EQSC-I}
\varphi(\tau_0,x_2) - 2 \varphi(\tau_0,x_1) + g(x_0) \leq \varphi(\tau_0 + 2 \tau_1,z_2) + g(x_0) - 2 g(z_1),
\end{equation}
where $z_1=\gamma^{\tau_0,x_1,u}_{\tau} \in \partial\Omega$ and $z_2=\gamma^{\tau_0,x_2,u^\star}(\tau_0 + 2 \tau_1)$. Let $v$ be an optimal control for $z_2$, at time $\tau_0 + 2 \tau_1$, and set $w_2 = \gamma_\tau^{\tau_0 + 2 \tau_1, z_2, v}$. Then, by Proposition \ref{PropBoundTau}, we have
\begin{align*}
\varphi(\tau_0 + 2 \tau_1,z_2) + g(x_0) - 2g(z_1) & = \tau^{\tau_0 + 2 \tau_1,z_2,v} + g(w_2) + g(x_0) - 2g(z_1) \displaybreak[0] \\
 & \leq c \mathbf d(z_2,\overline{\mathbb{R}^d \setminus \Omega}) + g(w_2) + g(x_0) - 2g(z_1) \displaybreak[0] \\
 & = c \mathbf d(z_2,\overline{\mathbb{R}^d \setminus \Omega}) + g(w_2) + g(x_0) \\
 & \hphantom{{} = {}} {} - 2 g\biggl(\frac{x_0 + w_2}{2}\biggr)  + 2\biggl(g\biggl(\frac{x_0 + w_2}{2}\biggr)-g(z_1)\biggr).
\end{align*} 
From \eqref{H2} $\&$ \eqref{boundary cost semiconcave}, we infer that
\begin{equation}
\label{EQSC-J}
\varphi(\tau_0 + 2 \tau_1,z_2) + g(x_0) - 2g(z_1) \leq c\left[\mathbf d(z_2,\overline{\mathbb{R}^d \setminus \Omega}) + \abs{w_2 - x_0}^2 + \abs{x_0 + w_2 - 2 z_1}\right].
\end{equation}
 Yet, using Proposition \ref{estimates trajectory}, we have
\begin{equation}
\label{EQSC-K}
\abs{w_2 - x_0} \leq \abs{w_2 - z_2} + \abs{z_2 - x_2} + \abs{x_2-x_0} \leq \abs{w_2 - z_2} + \abs{z_2 - x_2} + c\abs{h}.
\end{equation}
In addition, by Proposition \ref{PropBoundTau}, one has 
\begin{equation}
\label{eq1000}
\begin{aligned}
\abs{w_2 - z_2} & = \abs[\bigg]{\int_{\tau_0 + 2 \tau_1}^{\tau_0 + 2 \tau_1 +\tau^{\tau_0 + 2 \tau_1,z_2,v}} k\biggl(s,\gamma^{\tau_0 + 2 \tau_1,z_2,v}(s)\biggr) v(s) \diff s} \\
& \leq k_{\max} \tau^{\tau_0 + 2 \tau_1,z_2,v} \leq c \mathbf d(z_2,\overline{\mathbb{R}^d \setminus \Omega}).
\end{aligned}
\end{equation}
In the same way, we have, using \eqref{EQSC-G}, that
\begin{equation}
\label{EQSC-L}
\abs{z_2 -x_2}= \abs[\bigg]{\int_{\tau_0}^{\tau_0 + 2 \tau_1} k\biggl(s,\gamma^{\tau_0,x_2,u^\star}(s)\biggr) u^\star(s)\diff s} \leq 2 k_{\max} \tau_1 \leq c \abs{h}.
\end{equation}
Moreover, 
\begin{equation}
\label{EQSC-M}
\abs{x_0 + w_2 - 2 z_1} \leq \abs{x_0 + z_2 - 2 z_1} + \abs{w_2 - z_2}.
\end{equation}
Hence, inserting \eqref{eq1000} and \eqref{EQSC-L} into \eqref{EQSC-K}, and again \eqref{eq1000} into \eqref{EQSC-M}, it follows from \eqref{EQSC-H}, \eqref{EQSC-I}, and \eqref{EQSC-J} that the proof of Case 2(a) is completed if one shows that
\begin{equation}
\label{EQSC-N}
\mathbf d(z_2,\overline{\mathbb{R}^d \setminus \Omega}) + \abs{x_0 + z_2 - 2 z_1} \leq c \abs{h}^2.
\end{equation}

Note that
\begin{equation}
\label{EQSC-AAA}
\mathbf d(z_2,\overline{\mathbb{R}^d \setminus \Omega}) \leq \abs{z_2- 2z_1 + x_0} + \mathbf d(2z_1-x_0,\overline{\mathbb{R}^d \setminus \Omega}).
\end{equation}
Yet,
$$\mathbf d(2z_1-x_0,\overline{\mathbb{R}^d \setminus \Omega})=\mathbf d(2z_1-x_0,\overline{\mathbb{R}^d \setminus \Omega}) + \mathbf d(x_0,\overline{\mathbb{R}^d \setminus \Omega}) - 2 \mathbf d(z_1,\overline{\mathbb{R}^d \setminus \Omega}),$$
as $x_0, z_1 \in \partial\Omega$. Hence, by the semi-concavity of the distance function $\mathbf d(\cdot, \overline{\mathbb{R}^d \setminus \Omega})$ in $\bar{\Omega}$,
$$\mathbf d(2 z_1 - x_0,\overline{\mathbb{R}^d \setminus \Omega}) \leq c\abs{z_1-x_0}^2.$$
Now, using Proposition \ref{estimates trajectory}, we have
\begin{equation}
\label{EQSC-R}
\abs{z_1 - x_0} \leq \abs{z_1 - x_1} + \abs{x_1 - x_0} \leq c \abs{h}
\end{equation}
since, by \eqref{EQSC-G}, we have
\begin{equation}
\label{EQSC-S}
\abs{z_1-x_1}= \abs[\bigg]{\int_{\tau_0}^{\tau_0 +\tau_1} k\biggl(s,\gamma^{\tau_0,x_1,u}(s)\biggr) u(s) \diff s} \leq k_{\max} \tau_1 \leq c \abs{h}.
\end{equation}
Then $\mathbf d(2 z_1 - x_0, \overline{\mathbb R^d \setminus \Omega}) \leq c \abs{h}^2$. Hence, by \eqref{EQSC-AAA}, in order to prove \eqref{EQSC-N}, it suffices to show that
\begin{equation}
\label{EQSC-P}
\abs{z_2 - 2 z_1 + x_0} \leq c \abs{h}^2.
\end{equation}

Let $\mathbf{n}$ be the unit outward normal vector at $z_1$ and let $w:=u(\tau_0 + \tau_1)=-\frac{\nabla g (z_1) - \mu \mathbf{n}}{\abs{\nabla g (z_1) - \mu \mathbf{n}}}$ be the unit optimal control vector at $z_1$, at time $\tau_0 + \tau_1$ (where $\mu$ is the unique constant so that $k(\tau_0 + \tau_1,z_1)\abs{\nabla g(z_1) - \mu \mathbf{n}}=1$; see Lemma \ref{8.4.2}). If $d = 1$, then there exists $\alpha \in \mathbb R$ such that $2 z_1 - x_0 - z_2 = \alpha \mathbf{n}$. Otherwise, for $d \geq 2$, notice that, by Proposition \ref{PropLowerBoundScalarProduct}, $\mathbf n \cdot w \geq c > 0$, which shows that $\mathbf n$ and $w$ are not orthogonal, and thus there exists a unit vector $e$ orthogonal to $w$ such that
\begin{equation}
\label{EQSC-O}
2z_1 - x_0 - z_2 =\alpha \mathbf{n} + \beta e.
\end{equation}
We also write \eqref{EQSC-O} when $d = 1$ using the convention $e = 0$ for this case. Notice that $\abs{\mathbf n}^2 \geq \abs{\mathbf n \cdot w}^2 + \abs{\mathbf n \cdot e}^2$, and thus
\begin{equation} \label{change of basis}
1 - \abs{\mathbf n \cdot e}^2 \geq c.
\end{equation}

We have
\begin{align*}
(2z_1 - x_0 - z_2) \cdot \mathbf{n}& =\alpha + \beta e \cdot \mathbf{n}, \\
(2z_1 - x_0 - z_2) \cdot e & = \alpha e \cdot \mathbf{n} + \beta.
\end{align*}
Then,
$$\begin{pmatrix}
(2z_1 - x_0 - z_2) \cdot \mathbf{n} \\
 (2z_1 - x_0 - z_2) \cdot e
 \end{pmatrix} = \begin{pmatrix}
 1 & e \cdot \mathbf{n}\\
 e \cdot \mathbf{n} & 1
 \end{pmatrix} \begin{pmatrix}
 \alpha \\
 \beta
 \end{pmatrix}$$
 or equivalently,
$$ \begin{pmatrix}
 \alpha \\
 \beta
 \end{pmatrix} = {\begin{pmatrix}
 1 & e \cdot \mathbf{n}\\
 e \cdot \mathbf{n} & 1
 \end{pmatrix}}^{-1} \begin{pmatrix}
(2z_1 - x_0 - z_2) \cdot \mathbf{n} \\
 (2z_1 - x_0 - z_2) \cdot e
 \end{pmatrix}.$$
Thus
\begin{equation}
\label{EQSC-Q}
\abs{ 2z_1 - x_0 - z_2} \leq \abs{\alpha} + \abs{\beta} \leq \frac{2}{1-(e \cdot \mathbf{n})^2} \biggl(\abs{(2z_1 - x_0 - z_2) \cdot \mathbf{n}} +
\abs{(2z_1 - x_ 0 - z_2) \cdot e}\biggr),
\end{equation}
where the denominator can be estimated thanks to \eqref{change of basis}. We note that
\begin{gather*}
z_2 - 2z_1 + x_0 \\
=  x_0 + x_2 - 2 x_1 + \int_{\tau_0}^{\tau_0 + 2 \tau_1} k(s,\gamma^{\tau_0,x_2,u^\star}(s)) u^{\star}(s)\diff s
- 2 \int_{\tau_0}^{\tau_0 + \tau_1} k(s,\gamma^{\tau_0,x_1,u}(s)) u(s)\diff s \\
=  x_0 + x_2 - 2 x_1 + \int_{\tau_0}^{\tau_0 + 2 \tau_1} k(s,\gamma^{\tau_0,x_2,u^\star}(s)) u\biggl(\frac{s+ \tau_0}{2}\biggr)\diff s
- 2 \int_{\tau_0}^{\tau_0 + \tau_1} k(s,\gamma^{\tau_0,x_1,u}(s)) u(s)\diff s \\
=  x_0 + x_2 - 2 x_1 + 2 \int_{\tau_0}^{\tau_0 + \tau_1} \biggl(k(2s - \tau_0,\gamma^{\tau_0,x_2,u^\star}(2s - \tau_0)) - k(s,\gamma^{\tau_0,x_1,u}(s))\biggr) u(s)\diff s.
\end{gather*}
Hence, 
\begin{align} 
 (z_2 - 2 z_1 + x_0) \cdot e = {} & (x_0 + x_2 - 2 x_1) \cdot e \notag \\
& {} + 2 \int_{\tau_0}^{\tau_0 + \tau_1} \biggl(k(2s-\tau_0,\gamma^{\tau_0,x_2,u^\star}(2s - \tau_0)) - k(s,\gamma^{\tau_0,x_1,u}(s))\biggr) u(s) \cdot e\diff s. \label{eq1001}
\end{align}
Yet, from Proposition \ref{estimates trajectory}, we have
$$\abs{(x_0 + x_2 - 2 x_1) \cdot e} \leq \abs{x_0 + x_2 - 2 x_1} \leq c \abs{h}^2.$$
To estimate the second term in \eqref{eq1001}, we first observe that, since $u$ is $L_1$-Lipschitz continuous by Proposition \ref{RegCurve}, we have, for all $s \in [\tau_0, \tau_0 + \tau_1]$,
\begin{equation}
\label{EQSC-V}
\abs{u(s) - w} = \abs{u(s) - u(\tau_0 + \tau_1)} \leq c(\tau_0 + \tau_1 - s) \leq c \abs{h}
\end{equation}
using \eqref{EQSC-G}. This implies that
$$\abs{u(s) \cdot e}=\abs{(u(s) - w)\cdot e} \leq c \abs{h}.$$
Hence, using again \eqref{EQSC-G}, we get
\begin{multline*}
\abs[\bigg]{\int_{\tau_0}^{\tau_0 + \tau_1} \biggl(k(2s-\tau_0,\gamma^{\tau_0,x_2,u^\star}(2s - \tau_0)) - k(s,\gamma^{\tau_0,x_1,u}(s))\biggr)u(s) \cdot e\diff s} \\
 \leq 2 k_{\max} \int_{\tau_0}^{\tau_0 + \tau_1} \abs{u(s) \cdot e} \diff s \leq c \abs{h} \tau_1 \leq c \abs{h}^2.
\end{multline*}
Consequently, 
$$\abs{(z_2 - 2z_1 + x_0) \cdot e } \leq c \abs{h}^2.$$

To complete the proof of \eqref{EQSC-P}, it now suffices, by \eqref{EQSC-Q}, to show that
$$\abs{(2z_1 - x_0 - z_2) \cdot \mathbf{n}} \leq c \abs{h}^2.$$
We have
$$( 2z_1 - x_0 - z_2) \cdot \mathbf{n} = (z_1 - x_0)\cdot \mathbf{n} + (z_1 - z_2) \cdot \mathbf{n}.$$
Let $d^{\pm}$ be defined by \eqref{EqDefiDPM} and recall that $d^{\pm}$ is $C^{1, 1}$ in a neighborhood of $\partial\Omega$. Hence, we have 
$$d^\pm(x_0)=d^\pm(z_1) + \nabla d^\pm(z_1) \cdot (x_0 - z_1) + O(\abs{z_1-x_0}^2).$$
As $x_0, z_1 \in \partial\Omega$ and $\nabla d^{\pm} (z_1) = \mathbf n$, we get
$$(x_0 - z_1) \cdot {\bf{n}}=O(\abs{z_1-x_0}^2).$$
Yet, by \eqref{EQSC-R}, $\abs{z_1 - x_0} \leq c \abs{h}$. Then 
$$\abs{(z_1 - x_0) \cdot {\bf{n}}} \leq c \abs{h}^2.$$
Moreover, notice that
\begin{equation}
\label{EQSC-T}
\abs{z_2 - z_1} \leq \abs{z_2 - x_2} + \abs{x_2 - x_1} + \abs{x_1 - z_1} \leq c \abs{h}
\end{equation}
by \eqref{EQSC-L}, Proposition \ref{estimates trajectory}, and \eqref{EQSC-S}. We have
$$d^\pm(z_2)=d^\pm(z_1) + \nabla d^\pm(z_1) \cdot (z_2 - z_1) + O(\abs{z_2-z_1}^2).$$
As $z_2 \in \Omega$ and $z_1 \in \partial\Omega$, we get
$$ - {\bf{n}} \cdot (z_2 - z_1) + O(\abs{z_2-z_1}^2) \geq 0.$$
Then \eqref{EQSC-T} implies that
$$(z_1 -z_2) \cdot \mathbf{n} \geq -c \abs{h}^2.$$
Consequently,
$$\abs{(2 z_1 - x_0 - z_2) \cdot \mathbf{n}} \leq (2z_1 - x_0 -z_2) \cdot \mathbf{n} + c \abs{h}^2.$$

To complete the proof of Case 2(a), we are now left to prove that
\[
(2z_1 - x_0 -z_2) \cdot \mathbf{n} \leq c \abs{h}^2.
\]
As in \eqref{eq1001},
\begin{align*}
(2z_1 - x_0 -z_2) \cdot \mathbf{n} = {} & -(x_0 + x_2 - 2 x_1) \cdot \mathbf{n} \\
 & {} - 2 \int_{\tau_0}^{\tau_0 + \tau_1} \biggl(k(2s-\tau_0,\gamma^{\tau_0,x_2,u^\star}(2s - \tau_0)) - k(s,\gamma^{\tau_0,x_1,u}(s))\biggr)u(s) \cdot \mathbf{n}\diff s.
\end{align*}
From Proposition \ref{estimates trajectory}, we get again that
$$-(x_0 + x_2 - 2 x_1) \cdot \mathbf{n} \leq \abs{x_0 + x_2 - 2 x_1} \leq c \abs{h}^2.$$
For the second term, we have
\begin{multline} 
\label{EQSC-U}
-\int_{\tau_0}^{\tau_0 + \tau_1} \biggl(k(2s-\tau_0,\gamma^{\tau_0,x_2,u^\star}(2s - \tau_0)) 
- k(s,\gamma^{\tau_0,x_1,u}(s))\biggr) u(s) \cdot \mathbf{n}\diff s \\
{} = - \int_{\tau_0}^{\tau_0 + \tau_1} \biggl(k(2s - \tau_0,\gamma^{\tau_0,x_2,u^\star}(2s - \tau_0)) - k(s,\gamma^{\tau_0,x_2,u^\star}(2s - \tau_0))\biggr)u(s) \cdot \mathbf{n} \diff s \\
- \int_{\tau_0}^{\tau_0 + \tau_1} \biggl(k(s,\gamma^{\tau_0,x_2,u^\star}(2s - \tau_0)) - k(s,\gamma^{\tau_0,x_1,u}(s))\biggr)u(s) \cdot \mathbf{n} \diff s.
\end{multline} 
From \eqref{Hy1}, we have
\begin{multline}
\label{EQSC-AC}
\abs[\bigg]{ \int_{\tau_0}^{ \tau_0 + \tau_1} \biggl(k(s,\gamma^{\tau_0,x_2,u^\star}(2s - \tau_0)) - k(s,\gamma^{\tau_0,x_1,u}(s))\biggr)u(s) \cdot \mathbf{n}\diff s} \\
\leq c \int_{\tau_0}^{\tau_0 + \tau_1} \abs{\gamma^{\tau_0,x_2,u^\star}(2s - \tau_0) - \gamma^{\tau_0,x_1,u}(s)}\diff s.
\end{multline}
Yet, by Proposition \ref{estimates trajectory}, one has
\begin{equation}
\label{EQSC-AD}
\begin{gathered}
\abs[\bigg]{\gamma^{\tau_0,x_2,u^\star}(2s - \tau_0) - \gamma^{\tau_0,x_1,u}(s)} \\
=  \abs[\bigg]{x_2 + \int_{\tau_0}^{2s - \tau_0}k(t,\gamma^{\tau_0,x_2,u^\star}(t)) u^\star(t) \diff t - x_1 - \int_{\tau_0}^{s}k(t,\gamma^{\tau_0,x_1,u}(t)) u(t) \diff t } \\
\leq  \abs{x_2 - x_1} + \int_{\tau_0}^{2s - \tau_0}k(t,\gamma^{\tau_0,x_2,u^\star}(t)) \diff t + \int_{\tau_0}^{s}k(t,\gamma^{\tau_0,x_1,u}(t)) \diff t \\
\leq  c\abs{h} + 3 k_{\max}(s- \tau_0).
\end{gathered}
\end{equation}
Hence, we get by \eqref{EQSC-G} that
\begin{multline*}
\abs[\bigg]{ \int_{\tau_0}^{ \tau_0 + \tau_1} \biggl(k(s,\gamma^{\tau_0,x_2,u^\star}(2s - \tau_0)) - k(s,\gamma^{\tau_0,x_1,u}(s))\biggr)u(s) \cdot \mathbf{n}\diff s } \\
 \leq c \int_{\tau_0}^{\tau_0 +\tau_1}(\abs{h} + (s-\tau_0))\diff s \leq c \abs{h}^2.
\end{multline*}
We are left to consider the first term of the right-hand side of \eqref{EQSC-U}. Recalling that $w \cdot \mathbf n \geq c$ by Proposition \ref{PropLowerBoundScalarProduct}, we finally obtain, using \eqref{lower bound of the derivative of the dynamic with respect to time}, \eqref{EQSC-G}, and \eqref{EQSC-V}, that
\begin{multline*}
- \int_{\tau_0}^{\tau_0 + \tau_1} \biggl(k(2s - \tau_0,\gamma^{\tau_0,x_2,u^\star}(2s - \tau_0)) - k(s,\gamma^{\tau_0,x_2,u^\star}(2s - \tau_0))\biggr)u(s) \cdot \mathbf{n}\diff s\\ 
=- \int_{\tau_0}^{\tau_0 + \tau_1} \biggl(k(2s - \tau_0,\gamma^{\tau_0,x_2,u^\star}(2s - \tau_0)) - k(s,\gamma^{\tau_0,x_2,u^\star}(2s - \tau_0))\biggr)w \cdot \mathbf{n}\diff s\\ 
- \int_{\tau_0}^{\tau_0 + \tau_1} \biggl(k(2s - \tau_0,\gamma^{\tau_0,x_2,u^\star}(2s - \tau_0)) - k(s,\gamma^{\tau_0,x_2,u^\star}(2s - \tau_0))\biggr)(u(s) - w) \cdot \mathbf{n}\diff s \\ 
\leq \int_{\tau_0}^{\tau_0+\tau_1} \int_{s}^{2s - \tau_0}-k_t(t,\gamma^{\tau_0,x_2,u^\star}(2s - \tau_0)) w \cdot \mathbf{n} \diff t\diff s + c \abs{h}^2 \leq c(\tau_1^2 + \abs{h}^2) \leq c \abs{h}^2.
\end{multline*}

\case{Case 2(b): $\tau_2:=\tau^{\tau_0,x_2,u^\star} \leq 2 \tau_1$}.

Set 
$$z_1:=\gamma^{\tau_0,x_1,u}_\tau, z_2:=\gamma^{\tau_0,x_2,u^\star}_\tau \in \partial\Omega.$$
Recall that, by \eqref{EQSC-H}, it suffices to estimate $\varphi(\tau_0, x_2) - 2 \varphi(\tau_0, x_1) + g(x_0)$. Using Lemma \ref{dynamic programming principle} and \eqref{boundary cost semiconcave}, we have
\begin{equation}
\label{EQSC-W}
\begin{gathered}
\varphi(\tau_0,x_2) - 2 \varphi(\tau_0,x_1) + g(x_0) \\
\leq \tau_2 + g(z_2) - 2 \tau_1 - 2g(z_1) + g(x_0) \\
 =  \tau_2 - 2 \tau_1 + 2 \biggl(g\biggl(\frac{x_0 +z_2}{2}\biggr) - g(z_1)\biggr)
+ g(z_2) + g(x_0) - 2 g\biggl(\frac{x_0 +z_2}{2}\biggr) \\
\leq \tau_2 - 2 \tau_1 + 2 \biggl(g\biggl(\frac{x_0 +z_2}{2}\biggr) - g(z_1)\biggr) + c \abs{z_2 - x_0}^2.
\end{gathered}
\end{equation}
Using Proposition \ref{estimates trajectory}, we obtain that 
\begin{equation}
\label{EQSC-X}
\abs{z_2 - x_0} \leq \abs{z_2 - x_2} + \abs{x_2-x_0} \leq \abs{z_2 - x_2} + c\abs{h}.
\end{equation}
Yet, using \eqref{EQSC-G},
\begin{equation}
\label{EQSC-Y}
\abs{z_2 - x_2}=\abs[\bigg]{\int_{\tau_0}^{\tau_0 +\tau_2} k(s,\gamma^{\tau_0,x_2,u^\star}(s)) u^\star(s)\diff s} \leq 2 k_{\max} \tau_1 \leq c \abs{h}.
\end{equation}
On the other hand, using \eqref{boundary cost semiconcave}, we have
\begin{equation}
\label{EQSC-Z}
g\biggl(\frac{x_0 + z_2}{2}\biggr) - g(z_1) \leq \frac{1}{2} \nabla g(z_1) \cdot (x_0 + z_2 - 2 z_1) + O(\abs{x_0 + z_2 - 2 z_1}^2).
\end{equation}
But it is clear that
\begin{align*}
\abs{x_0 + z_2 - 2 z_1} \leq {} & \abs{x_0 + x_2 - 2 x_1} + 2 \int_{\tau_0}^{\tau_0 + \tau_1}\abs{ k(s,\gamma^{\tau_0,x_1,u}(s)) u(s)}\diff s \\
& {} + \int_{\tau_0}^{\tau_0 + \tau_2} \abs{ k(s,\gamma^{\tau_0,x_2,u^\star}(s)) u^\star(s)}\diff s,
\end{align*}
which implies, using Proposition \ref{estimates trajectory} and \eqref{EQSC-G}, that
\begin{equation}
\label{EQSC-AA}
\abs{x_0 + z_2 - 2 z_1} \leq c \abs{h}^2 + 4 k_{\max} \tau_1 \leq c\abs{h}.
\end{equation}
So, inserting \eqref{EQSC-Y} into \eqref{EQSC-X} and \eqref{EQSC-AA} into \eqref{EQSC-Z}, we conclude from \eqref{EQSC-H} and \eqref{EQSC-W} that the proof of Case 2(b) is completed if one shows that
$$\tau_2 - 2 \tau_1 + \nabla g(z_1) \cdot (x_0 + z_2 - 2 z_1) \leq c\abs{h}^2.$$

Let $\mathbf{n}$ be the unit outward normal vector at $z_1$. If $d = 1$, there exists $\alpha \in [-\lambda, \lambda]$ such that $\nabla g(z_1) = \alpha \mathbf n$. Otherwise, for $d \geq 2$, there exist a unit vector $e$ orthogonal to $\mathbf n$ and $\alpha, \beta \in \mathbb R$ such that
\begin{equation}
\label{EQSC-AB}
\nabla g(z_1) = \alpha \mathbf{n} + \beta e.
\end{equation}
We write \eqref{EQSC-AB} also when $d = 1$ by setting $e = 0$ in this case. Notice that $\alpha^2 + \beta^2 = \abs{\nabla g(z_1)}^2 \leq \lambda^2$. We have
\begin{align*}
\nabla g(z_1) \cdot (x_0 + z_2 - 2 z_1) &= (\alpha \mathbf{n} + \beta e) \cdot (x_0 + z_2 - 2 z_1) \\ 
&= \alpha \mathbf{n} \cdot (x_0 + z_2 - 2 z_1) + \beta e \cdot (x_0 + z_2 - 2 z_1).
\end{align*}
Similarly to \eqref{EQSC-R} and \eqref{EQSC-T} from Case 2(a), one can show that
$$ \abs{x_0 - z_1} + \abs{z_1 - z_2} \leq c \abs{h}.$$
From \eqref{Smoothness of the boundary} and the fact that $x_0, z_1, z_2 \in \partial\Omega$, we infer that
$$ \alpha \mathbf{n} \cdot (x_0 + z_2 - 2 z_1) \leq c \abs{h}^2.$$
We are now left to prove that
\[
\tau_2 - 2 \tau_1 + \beta e \cdot (x_0 + z_2 - 2 z_1) \leq c \abs{h}^2.
\]
Set
$$z:=\gamma^{\tau_0,x_1,u}\biggl(\tau_0 + \frac{\tau_2}{2}\biggr).$$
Then, one has 
\[
\tau_2 - 2 \tau_1 + \beta e \cdot (x_0 + z_2 - 2 z_1) = \tau_2 - 2 \tau_1 + \beta e \cdot (x_0 + z_2 - 2 z) + 2 \beta e \cdot(z-z_1).
\]
Let us observe that 
$$\abs{z-z_1} = \abs[\bigg]{\int_{\tau_0+\frac{\tau_2}{2}}^{\tau_0 + \tau_1} k(s,\gamma^{\tau_0,x_1,u}(s)) u(s)\diff s} \leq k_{\max} \biggl(\tau_1 - \frac{\tau_2}{2}\biggr).$$
Using $k_{\max} \abs{\beta} \leq k_{\max} \lambda < 1$, we infer that
\[
\tau_2 - 2 \tau_1 + \beta e \cdot (x_0 + z_2 - 2 z) + 2 \beta e \cdot(z-z_1) \leq \beta e \cdot (x_0 + z_2 - 2 z).
\]
So, the aim, now, is to prove that 
\begin{equation}
\label{EQSC-AE}
\beta e \cdot (x_0 + z_2 - 2 z) \leq c\abs{h}^2.
\end{equation}

Let us observe that 
\begin{multline}
\label{EQSC-AF}
x_0 + z_2 - 2 z \\
 =  x_0 + x_2 - 2 x_1 - 2 \int_{\tau_0}^{\tau_0 + \frac{\tau_2}{2}} k(s,\gamma^{\tau_0,x_1,u}(s))u(s)\diff s + \int_{\tau_0}^{\tau_0 + \tau_2} k(s,\gamma^{\tau_0,x_2,u^\star}(s)) u^\star(s)\diff s \displaybreak[0] \\
 =  x_0 + x_2 - 2 x_1 - 2 \int_{\tau_0}^{\tau_0 + \frac{\tau_2}{2}} k(s,\gamma^{\tau_0,x_1,u}(s))u(s)\diff s + \int_{\tau_0}^{\tau_0 + \tau_2} k(s,\gamma^{\tau_0,x_2,u^\star}(s)) u\biggl(\frac{s+\tau_0}{2}\biggr)\diff s \displaybreak[0] \\
 =  x_0 + x_2 - 2 x_1 + 2 \int_{\tau_0}^{\tau_0 + \frac{\tau_2}{2}}\biggl(k(2s-\tau_0,\gamma^{\tau_0,x_2,u^\star}(2s-\tau_0)) - k(s,\gamma^{\tau_0,x_1,u}(s))\biggr)u(s)\diff s \displaybreak[0] \\
 =  x_0 + x_2 - 2 x_1 + 2 \int_{\tau_0}^{\tau_0 +\frac{\tau_2}{2}}\biggl(k(2s-\tau_0,\gamma^{\tau_0,x_2,u^\star}(2s-\tau_0)) - k(s,\gamma^{\tau_0,x_2,u^\star}(2s-\tau_0))\biggr)u(s)\diff s \\
{} + 2 \int_{\tau_0}^{\tau_0+\frac{\tau_2}{2}}\biggl(k(s,\gamma^{\tau_0,x_2,u^\star}(2s-\tau_0)) - k(s,\gamma^{\tau_0,x_1,u}(s))\biggr)u(s)\diff s.
\end{multline}
Recall that, by Proposition \eqref{estimates trajectory},
\begin{equation}
\label{EQSC-AG}
\abs{x_2 + x_0 - 2x_1} \leq c\abs{h}^2.
\end{equation}
From \eqref{Hy1} and proceeding as in \eqref{EQSC-AC} and \eqref{EQSC-AD}, we infer that 
\begin{multline*}
\abs[\bigg]{\int_{\tau_0}^{\tau_0 + \frac{\tau_2}{2}}\biggl(k(s,\gamma^{\tau_0,x_2,u^\star}(2s-\tau_0)) - k(s,\gamma^{\tau_0,x_1,u}(s))\biggr)u(s)\diff s }\\ 
 \leq c \int_{\tau_0}^{\tau_0 + \frac{\tau_2}{2}}\abs[\bigg]{\gamma^{\tau_0, x_2,u^\star}(2s-\tau_0) - \gamma^{\tau_0,x_1,u}(s)}\diff s,
\end{multline*}
and, by Proposition \ref{estimates trajectory},
\begin{gather*}
\abs[\bigg]{\gamma^{\tau_0,x_2,u^\star}(2s - \tau_0) - \gamma^{\tau_0,x_1,u}(s)} \\
= \biggl\lvert x_2 + \int_{\tau_0}^{2s - \tau_0} k(t,\gamma^{\tau_0,x_2,u^\star}(t)) u^\star(t)\diff t 
 - x_1 - \int_{\tau_0}^{s} k(t,\gamma^{\tau_0,x_1,u}(t)) u(t)\diff t \\
\leq \abs{x_2 - x_1} + \abs[\bigg]{\int_{\tau_0}^{2s - \tau_0} k(t,\gamma^{\tau_0,x_2,u^\star}(t)) u^\star(t)\diff t} + \abs[\bigg]{\int_{\tau_0}^{s} k(t,\gamma^{\tau_0,x_1,u}(t)) u(t)\diff t } \\
\leq c\abs{h} + 3 k_{\max}(s-\tau_0).
\end{gather*}
Consequently, we get 
\begin{equation}
\label{EQSC-AH}
\abs[\bigg]{\int_{\tau_0}^{\tau_0 + \frac{\tau_2}{2}}\biggl(k(s,\gamma^{\tau_0,x_2,u^\star}(2s - \tau_0)) - k(s,\gamma^{\tau_0,x_1,u}(s))\biggr) u(s)\diff s } \leq c\abs{h}^2.
\end{equation}
On the other hand, using the fact from Proposition \ref{RegCurve} that $u$ is $L_1$-Lipschitz continuous, we get that
\begin{multline}
\label{EQSC-AJ}
\int_{\tau_0}^{\tau_0 +\frac{\tau_2}{2}}\biggl(k(2s - \tau_0,\gamma^{\tau_0,x_2,u^\star}(2s - \tau_0))
-k(s,\gamma^{\tau_0,x_2,u^\star}(2s-\tau_0))\biggr) u(s)\diff s \displaybreak[0] \\
{} = \int_{\tau_0}^{\tau_0 +\frac{\tau_2}{2}}\biggl(k(2s - \tau_0,\gamma^{\tau_0,x_2,u^\star}(2s - \tau_0))
-k(s,\gamma^{\tau_0,x_2,u^\star}(2s-\tau_0))\biggr) u(\tau_0 + \tau_1)\diff s \\
{} + \int_{\tau_0}^{ \tau_0 +\frac{\tau_2}{2}}\biggl(k(2s - \tau_0,\gamma^{\tau_0,x_2,u^\star}(2s - \tau_0))
-k(s,\gamma^{\tau_0,x_2,u^\star}(2s-\tau_0))\biggr) (u(s) - u(\tau_0 + \tau_1))\diff s \displaybreak[0] \\
{} \leq \int_{\tau_0}^{ \tau_0 +\frac{\tau_2}{2}}\biggl(k(2s - \tau_0,\gamma^{\tau_0,x_2,u^\star}(2s - \tau_0))
-k(s,\gamma^{\tau_0,x_2,u^\star}(2s-\tau_0))\biggr) u(\tau_0 + \tau_1)\diff s + c\abs{h}^2.
\end{multline}
We recall that
$$u(\tau_0 + \tau_1)=-\frac{\nabla g(z_1) - \mu \mathbf{n}}{\abs{\nabla g(z_1) - \mu \mathbf{n}}} \qquad \text{ and } \qquad k(\tau_0 + \tau_1, z_1) \abs{\nabla g(z_1) - \mu \mathbf n} = 1,$$
and so, using \eqref{lower bound of the derivative of the dynamic with respect to time} and \eqref{EQSC-G}, we get
\begin{equation}
\label{EQSC-AI}
\begin{aligned}
 & \int_{\tau_0}^{ \tau_0 +\frac{\tau_2}{2}}\biggl(k(2s - \tau_0,\gamma^{\tau_0,x_2,u^\star}(2s - \tau_0))
-k(s,\gamma^{\tau_0,x_2,u^\star}(2s-\tau_0))\biggr) u(\tau_0 + \tau_1) \cdot \beta e\diff s \\ 
{} = {} & - k(\tau_0 + \tau_1,z_1) \beta^2 \int_{\tau_0}^{ \tau_0 +\frac{\tau_2}{2}}\biggl(k(2s - \tau_0,\gamma^{\tau_0,x_2,u^\star}(2s - \tau_0)) - k(s,\gamma^{\tau_0,x_2,u^\star}(2s-\tau_0))\biggr)\diff s\\ 
{} = {} & k(\tau_0 + \tau_1,z_1) \beta^2 \int_{\tau_0}^{\tau_0 + \frac{\tau_2}{2}} \int_{s}^{2s - \tau_0}-k_t(t,\gamma^{\tau_0,x_2,u^\star}(2s - \tau_0)) \diff t\diff s \leq c \abs{h}^2.
\end{aligned}
\end{equation}
We then obtain \eqref{EQSC-AE} by combining \eqref{EQSC-AF}, \eqref{EQSC-AH}, \eqref{EQSC-AJ}, and \eqref{EQSC-AI}.
\end{proof}

We finish this section by a remark on the importance of assuming \eqref{lower bound of the derivative of the dynamic with respect to time}.
 
\begin{remark}
One can give an example in the case $g=0$ showing that a lower bound on the derivative of the dynamic $k$ with respect to $t$ is a sharp condition to obtain semi-concavity of $\varphi$. To see that, let $\Omega$ be the unit ball in $\mathbb{R}^d$. Let $\zeta$ be a differentiable real function with $0< \zeta_{\min} \leq \zeta \leq \zeta_{\max} < +\infty$. Set $k(t,x):=\zeta(t)$, for every $(t,x) \in \mathbb{R}^+ \times \Omega$. For a given $x \in \Omega$, the optimal trajectory for $x$, at time $0$, will be given by
$$\gamma^\prime(s) =k(s,\gamma(s)) e(x)=\zeta(s) e(x),$$
where $e(x):= x/\abs{x}$. Let $\varphi$ be the value function associated with this optimal control problem. We observe easily that
$$\int_0^{\varphi(0,x)} \zeta(s)\diff s = 1 - \abs{x}.$$
Now, set 
$$G(T):=\int_0^T \zeta(s)\diff s, \quad \text{for all } T \geq 0$$ and $H:=G^{-1}$.
This yields that
$$\varphi(0,x)=H(1 - \abs{x}).$$
Consequently, we have 
$$\nabla^2 \varphi(0,x)=H^{\prime\prime}(1 - \abs{x}) e(x) \otimes e(x) - \frac{H^\prime (1-\abs{x})}{\abs{x}} (I - e(x) \otimes e(x)),$$
where
$$H^\prime= \frac{1}{\zeta} \circ H \qquad \text{and} \qquad H^{\prime\prime}=-\frac{\zeta^\prime}{\zeta^3} \circ H.$$
This shows that
$\nabla^2 \varphi$ cannot be bounded from above unless $\zeta^\prime$ is bounded from below.
\end{remark} 

\subsection{Differentiability of the value function}
\label{SecDifferentiabilityVarphi}

We prove in this section an extra result on our exit-time optimal control problem, namely that the value function $\varphi$ is differentiable along optimal trajectories. This kind of result is classical (see \cite{CanSin}) and one can even obtain more (for instance, in \cite{CanFrank} smoothness of the value function in a neighborhood of optimal trajectories is proven under some suitable conditions). Yet, most of the literature is concerned with the autonomous case (see in particular \cite{CanSin}), which motivates us to provide a detailed proof here dealing with the subtleties of our non-autonomous setting.

The results of this subsection will be of use in Section \ref{SecExistenceMFG} in order to obtain the continuity equation in \eqref{MFGSystem}. They require a result stronger than Theorem \ref{Theorem semiconcavity}, namely the semi-concavity of $\varphi$ on both variables $(t, x)$, and not only on $x$. We then need the stronger assumption that
\begin{hypothesis} \label{Smoothness of the gradient of the dynamic with respect to time}
k \in C^{1,1}(\mathbb{R}^+ \times \Omega).
\end{hypothesis}

\begin{remark}
The sharper assumption \eqref{lower bound of the derivative of the dynamic with respect to time} will be of use in Section \ref{SecExistenceLessRegular} when studying a less regular MFG model. Its study is carried out by an approximation procedure, with approximated dynamics $k_\varepsilon \in C^{1, 1}(\mathbb R^+ \times \Omega)$ for $\varepsilon > 0$ but with no uniform bounds on their $C^{1, 1}$ behavior, except for an uniform lower bound on $\partial_t k_{\varepsilon}$, which is the motivation for \eqref{lower bound of the derivative of the dynamic with respect to time}. Since the differentiability of $\varphi$ along optimal trajectories plays no particular role in this approximation procedure, one may assume the stronger assumption \eqref{Smoothness of the gradient of the dynamic with respect to time} for the purposes of this section.
\end{remark}

Our first result concerns the semi-concavity of $\varphi$ on $(t, x)$.

\begin{proposition} \label{semi-concav}
Under assumptions \eqref{lower bound on the dynamic}, \eqref{H2}, \eqref{Smoothness of the boundary}, \eqref{boundary cost semiconcave}, and \eqref{Smoothness of the gradient of the dynamic with respect to time}, the value function $\varphi$ is semi-concave on $\mathbb{R}^+ \times \Omega$.
\end{proposition}

\begin{proof}
We apply the classical semi-concavity result from \cite[Theorem 8.2.7]{CanSin} to the augmented system $z^\prime = \widetilde k(z, u)$, where $z = (t, \gamma)$ and $\widetilde k$ is given by $\widetilde k(z, u) = (1, k(z) u)$.
\end{proof}

As a consequence of the semi-concavity of $\varphi$ on $\mathbb R^+ \times \Omega$ and the standard properties of semi-concave functions recalled in Proposition \ref{PropSemiconcave}, one obtains the following result.

\begin{proposition}
\label{PropLowerBoundOnGrad}
Let $c > 0$ be the constant from Corollary \ref{CoroGradNotZero}. Let $(t_0, x_0) \in \mathbb R^+ \times \accentset\circ\Omega$ and assume that $\nabla\varphi(t_0, x_0)$ exists. Then $\abs{\nabla\varphi(t_0, x_0)} \geq c$.
\end{proposition}

Notice that this improves the result of Corollary \ref{CoroGradNotZero} concerning $\nabla\varphi$, since one does not assume differentiability of $\varphi$ on $(t_0, x_0)$ in the statement of Proposition \ref{PropLowerBoundOnGrad}, but only the existence of $\nabla\varphi$.

\begin{proof}
By Proposition \ref{LipValueFunction}, $\varphi$ is Lipschitz continuous on $\mathbb R^+ \times \Omega$, and hence it is differentiable almost everywhere. Then there exists sequences $(t_n)_{n \in \mathbb N^\ast}$ in $\mathbb R^+$ and $(x_n)_{n \in \mathbb N^\ast}$ in $\Omega$ such that $\varphi$ is differentiable at $(t_n, x_n)$ for every $n \in \mathbb N^\ast$ and $t_n \to t_0$ and $x_n \to x_0$ as $n \to \infty$. In particular, one has $\abs{\nabla\varphi(t_n, x_n)} \geq c$ for every $n \in \mathbb N$.

Let $p_n = D \varphi(t_n, x_n)$. Since $\varphi$ is Lipschitz continuous, $p_n$ is bounded, and hence, up to the extraction of a subsequence, $p_n$ converges to some $p = (p_t, p_x) \in \mathbb R \times \mathbb R^d$. Then $p \in D^\star \varphi(t_0, x_0)$ and thus $p_x \in \Pi_x(D^\star \varphi(t_0, x_0)) \subset \Pi_x (D^+ \varphi(t_0, x_0)) \subset \nabla^+\varphi(t_0, x_0) = \{\nabla\varphi(t_0, x_0)\}$. Hence $\abs{\nabla\varphi(t_0, x_0)} = \abs{p_x} = \lim_{n \to \infty} \abs{\nabla\varphi(t_n, x_n)} \geq c$, as required.
\end{proof}

Another consequence of the semi-concavity of $\varphi$ is the following.

\begin{proposition} \label{propagation of diff}
Let $\gamma$ be an optimal trajectory for $x_0$, at time $t_0$, and $u$ be the associated optimal control. If $\varphi$ is differentiable at $(t_0,x_0)$, then $\varphi$ is differentiable at $(t,\gamma(t))$, for all $t \in [t_0,t_0 + \tau_0)$, where $\tau_0=\tau^{t_0,x_0,u}$.
\end{proposition}

\begin{proof}
Fix $t \in [t_0,t_0 + \tau_0)$. If $\varphi$ is differentiable at $(t_0,x_0)$, then the subdifferential $\nabla^- \varphi(t_0,x_0)$ is a singleton, say $\nabla^- \varphi(t_0,x_0)=\{p_0\}$. Now, let $p$ be a solution of \eqref{adjoint equation} with initial condition $p(t_0)=p_0$. By Proposition \ref{propagation of subdiff}, $p(t) \in \nabla^- \varphi(t,\gamma(t))$, which implies, in particular, that $\nabla^- \varphi(t,\gamma(t)) \neq \emptyset$. On the other hand, as $\varphi$ is semi-concave, then $D^+ \varphi(t,\gamma(t)) \neq \emptyset$ and so, $\nabla^+ \varphi(t,\gamma(t)) \neq \emptyset$. Hence, $\varphi$ is differentiable with respect to $x$ at $(t,\gamma(t))$. Now, take $(p_t,p_x) \in D^+ \varphi(t,\gamma(t))$. Then, $p_x = \nabla \varphi(t,\gamma(t))$. Yet, by Proposition \ref{Hamilton in the interior}, we have
\begin{equation} \label{corollary 3.2.1}
-p_t + k(t,\gamma(t)) \abs{p_x}=1,
\end{equation}
which implies that $p_t$ is uniquely determined by $\nabla \varphi(t,\gamma(t))$. Consequently, $D^+ \varphi(t,\gamma(t))$ is a singleton and so, $\varphi$ is differentiable at $(t,\gamma(t))$ (thanks again to the semi-concavity of the value function $\varphi$).
\end{proof}

\begin{remark}
The proof of Proposition \ref{propagation of diff} cannot be extended to include the final time $t_0 + \tau_0$ as \eqref{corollary 3.2.1} does not hold a priori at the endpoint of an optimal trajectory.
\end{remark}

\begin{proposition} \label{propagation of reachable}
Let $t_0$, $x_0$, $\gamma$, $u$, $\tau_0$, and $p$ be as in Proposition \ref{propagation of subdiff}. Fix $t_1 \in (t_0,t_0 +\tau_0)$ and set $x_1:=\gamma(t_1)$. Suppose that $p(t_1) \in \Pi_x (D^\star\varphi(t_1,x_1))$, then $p(t) \in \Pi_x(D^\star \varphi(t,\gamma(t)))$, for all $t \in [t_1,t_0 + \tau_0]$.
\end{proposition}

\begin{proof}
If $p(t_1) \in \Pi_x (D^\star\varphi(t_1,x_1))$, then there is a sequence $(t_{1,n},x_{1,n}) \in \mathbb{R}^+ \times \Omega$ such that $t_{1,n} \to t_1$, $x_{1,n} \to x_1$ and $\varphi$ is differentiable at $(t_{1,n},x_{1,n})$ with $\nabla \varphi(t_{1,n},x_{1,n}) \to p(t_1)$. As $\varphi$ is differentiable at $(t_{1,n},x_{1,n})$, then, by Proposition \ref{propagation of diff}, $\varphi$ is differentiable at $(t,\gamma_n(t))$, for all $t \in [t_{1,n},t_{1,n} + \tau_{1,n})$, where $\gamma_n$ is an optimal trajectory for $x_{1,n}$, at time $t_{1,n}$, and $\tau_{1,n}=\tau^{t_{1,n},x_{1,n},u_n}$, $u_n$ being the optimal control associated with $\gamma_n$. Let $p_n$ be the solution of
\begin{equation}
\left\{
\begin{aligned}
p_n^\prime(t) & = -\nabla k(t,\gamma_n(t)) u_n(t) \cdot p_n(t), & \quad & t \in [t_{1,n},t_{1,n} + \tau_{1,n}], \\ 
p_n(t_{1,n}) & = \nabla \varphi(t_{1,n},x_{1,n}).
\end{aligned}
\right.
\end{equation}
By Proposition \ref{propagation of subdiff}, we have $p_n(t)= \nabla \varphi(t,\gamma_n(t))$ for all $t \in [t_{1,n},t_{1,n} + \tau_{1,n})$. Yet, it is clear, from Lemma \ref{Uniform convergence trajectory} $\&$ Proposition \ref{PropSingletonAfterStartingTime}, that $u_n \to u$ and $\gamma_n \to \gamma$ uniformly, where $u$ is the unique optimal control for $x_1$, at time $t_1$, and $\gamma$ is its associated optimal trajectory. So, we also have $p_n \to p$ uniformly. Now, fix $t \in [t_1,t_0 + \tau_0]$ and let $(t_n)_n$ be any sequence such that $t_n \in (t_{1,n},t_{1,n} + \tau_{1,n})$, for all $n$, and $t_n \to t$. As $p_n(t_n)= \nabla \varphi(t_n,\gamma_n(t_n))$, we get that $p(t)= \lim_n \nabla \varphi(t_n,\gamma_n(t_n))$, which means that $p(t) \in \Pi_x(D^\star \varphi(t,\gamma(t)))$.
\end{proof}

We are now ready to prove the main result of this subsection.

\begin{theorem} \label{Differentiability of the value function}
Given $(t_0,x_0) \in \mathbb{R}^+ \times \Omega$, let $\gamma:[t_0,t_0 + \tau_0] \to \Omega$ be an optimal trajectory for $x_0$, at time $t_0$, where $\tau_0=\tau^{t_0,x_0,u}$; $u$ being the associated optimal control. Then, $\varphi$ is differentiable at all points $(t,\gamma(t))$, with $t \in (t_0,t_0+\tau_0)$.
\end{theorem}

\begin{proof}
Let us argue by contradiction and suppose that $D^+ \varphi(t,\gamma(t))$ is not a singleton for some $t \in (t_0,t_0+\tau_0)$. Then, thanks to Proposition \ref{PropSemiconcave}, $D^\star \varphi(t,\gamma(t))$ contains at least two elements, say $(p_{t,0},p_{x,0})$ and $(p_{t,1},p_{x,1})$. Yet, from Proposition \ref{Hamilton in the interior}, we see that different elements of $D^\star \varphi(t,\gamma(t))$ have different space components, i.e.\ $p_{x,0} \neq p_{x,1}$. For any $\theta \in (0,1)$, we have $(1-\theta)(p_{t,0},p_{x,0}) + \theta (p_{t,1},p_{x,1}) \in D^+\varphi(t,\gamma(t))$ and so, recalling again Proposition \ref{Hamilton in the interior}, one has
\begin{gather*}
-p_{t,0} + k(t,\gamma(t)) \abs{p_{x,0}} - 1=0, \\
-p_{t,1} + k(t,\gamma(t)) \abs{p_{x,1}} - 1=0,
\end{gather*}
and $$-(1-\theta)p_{t,0} - \theta p_{t,1} + k(t,\gamma(t)) \abs{(1-\theta) p_{x,0} + \theta p_{x,1}} - 1=0.$$
Hence, we have
\begin{equation}
\label{EqCombinationPx}
\abs{(1-\theta) p_{x,0} + \theta p_{x,1}} = (1-\theta) \abs{p_{x,0}} + \theta \abs{p_{x,1}},
\end{equation}
which implies that $p_{x,1}=\alpha p_{x,0}$, for some $\alpha>0$, $\alpha \neq 1$. Now, let $p_0$ and $p_1$ be the solutions of \eqref{adjoint equation}, associated with the optimal $(\gamma,u)$, with initial condition $p_0(t)=p_{x,0}$ and $p_1(t)=p_{x,1}$, respectively. Then $p_1=\alpha p_0$. In particular, we have $p_1(t_0+\tau_0)=\alpha p_0(t_0 + \tau_0)$. Yet, by Proposition \ref{propagation of reachable}, we know that both $p_0(t_0+\tau_0)$ and $p_1(t_0+\tau_0)$ belong to $\Pi_x(D^\star \varphi(t_0+\tau_0,\gamma(t_0 + \tau_0)))$. As $\varphi(t,x)=g(x)$ at every $(t,x) \in \mathbb{R}^+ \times \partial\Omega$, then $\varphi$ is differentiable with respect to $t$ on $\mathbb{R}^+ \times \partial\Omega$ and, $\partial_t \varphi=0$. This implies that $\Pi_t(D^\star \varphi(t_0+\tau_0,\gamma(t_0 + \tau_0)))=\{0\}$. Hence, we obtain, using Proposition \ref{Hamilton in the interior}, that if $q_0,q_1 \in \Pi_x(D^\star \varphi(t_0+\tau_0,\gamma(t_0 + \tau_0)))$, then $\abs{q_0}=\abs{q_1}$. This implies that $\abs{p_0(t_0+\tau_0)}=\abs{p_1(t_0+\tau_0)}=\alpha \abs{p_0(t_0+\tau_0)}$, which is a contradiction as $\alpha \neq 1$. Hence, $\varphi$ is differentiable at $(t,\gamma(t))$, for all $t \in (t_0,t_0+\tau_0)$.
\end{proof}

\begin{remark}
Contrarily to other classical results on the differentiability of the value function along optimal trajectories such as \cite[Theorem 8.4.6]{CanSin}, we cannot conclude the proof of Theorem \ref{Differentiability of the value function} using only local information on the superdifferential at the point $(t, \gamma(t))$. The main conclusion we obtain from local information is \eqref{EqCombinationPx}, which allows us to deduce that $p_{x, 0}$ and $p_{x, 1}$ are colinear and point to the same direction. In order to obtain the desired contradiction, we need to propagate this information to the boundary, using Proposition \ref{propagation of reachable}, and exploit the additional information that $\partial_t \varphi$ vanishes on $\partial\Omega$ in order to conclude.
\end{remark}

As a consequence of Propositions \ref{maximum principle} $\&$ \ref{propagation of surdiff} and Theorem \ref{Differentiability of the value function}, one can characterize an optimal control $u$ in terms of the normalized gradient, with respect to $x$, of the value function $\varphi$.

\begin{corollary}
\label{CoroAlmostVelocityField}
Let $(t_0, x_0) \in \mathbb R^+ \times \Omega$ and $\gamma=\gamma^{t_0,x_0,u}$ be an optimal trajectory for $x_0$, at time $t_0$, where $u$ is the associated optimal control. Then, for all $t \in (t_0, t_0 + \tau_0)$, where $\tau_0:=\tau^{t_0,x_0,u}$, one has
\begin{equation}
\label{NablaTauQIsOptimalControl}
\gamma^\prime(t) = -k(t, \gamma(t)) \frac{\nabla \varphi(t, \gamma(t))}{\abs{\nabla \varphi(t, \gamma(t))}}. \newline
\end{equation}
\end{corollary} 

Our final result of this section provides a converse to Corollary \ref{CoroAlmostVelocityField}, proving that any solution of \eqref{NablaTauQIsOptimalControl} is an optimal trajectory, and also that such solutions are unique for almost every initial condition.

\begin{proposition}
\label{PropUnique}
Fix $(t_0,x_0) \in \mathbb{R}^+\times \Omega$ and let $\gamma: [t_0, +\infty) \to \mathbb R^d$ be an absolutely continuous function satisfying, for almost every $t \in [t_0, +\infty)$,
\begin{equation} \label{eqUni}
\begin{aligned}
\gamma^\prime(t) & = 
\begin{dcases*}
-k(t, \gamma(t)) \frac{\nabla \varphi(t, \gamma(t))}{\abs{\nabla \varphi(t, \gamma(t))}}, & if $\gamma(t) \in \accentset\circ\Omega$, \\
0, & otherwise,
\end{dcases*} \\
\gamma(t_0) & = x_0.
\end{aligned}
\end{equation}
Then $\gamma$ is an optimal trajectory for $x_0$ at time $t_0$. Moreover, for every $t_0 \in \mathbb{R}^+$ and for a.e.\ $x_0 \in \Omega$, \eqref{eqUni} admits a unique solution $\gamma$.
\end{proposition}

\begin{proof}
Let $t\geq t_0$ be such that \eqref{eqUni} holds at $t$. So, this implies, in particular, that $\varphi$ is differentiable with respect to $x$ at $(t,\gamma(t))$. Thanks to Proposition \ref{semi-concav}, we infer that $\varphi$ is also differentiable with respect to $t$. This follows from the fact that if $(p_t,p_x)\in D^\star\varphi(t,\gamma(t))$ then, using Proposition \ref{MainTheoHJ}, $p_t$ is uniquely determined by $p_x = \nabla\varphi(t, \gamma(t))$. Then $D^\star \varphi(t, \gamma(t))$ is a singleton, which implies, by Proposition \ref{PropSemiconcave}, that $D^+\varphi(t, \gamma(t))$ is also a singleton.

Hence, we have
$$-\frac{\diff}{\diff t}\varphi(t,\gamma(t))=-\partial_t\varphi(t,\gamma(t)) - \nabla\varphi(t,\gamma(t)) \cdot \gamma^\prime(t)=-\partial_t\varphi(t,\gamma(t)) + k(t,\gamma(t)) \abs{\nabla\varphi(t,\gamma(t))} = 1.$$
Integrating the above inequality over $[t_0,t_0 + \tau^{t_0,x_0,u}]$ we finally obtain, since $\varphi(t_0 + \tau^{t_0, x_0, u},\allowbreak\gamma^{t_0,x_0,u}_\tau)=g(\gamma^{t_0,x_0,u}_\tau)$ where $u$ is the control associated with $\gamma$,
$$\varphi(t_0,x_0)=\tau^{t_0,x_0,u} + g(\gamma^{t_0,x_0,u}_\tau).$$
Therefore $u$ is optimal.

For a given $t_0 \in \mathbb R^+$, $x \mapsto \varphi(t_0, x)$ is Lipschitz continuous by Proposition \ref{LipValueFunction}, and then $\nabla\varphi(t_0, x_0)$ exists for almost every $x_0 \in \Omega$. The last statement of the proposition is then a direct consequence of Proposition \ref{UniqueOptimalTrajectory}.
\end{proof}

\section{Optimal-exit mean field games}
\label{SectionMFG}

After the preliminary study of the corresponding optimal control problem in Section \ref{SecOptimalControl}, we are ready to consider in this section the mean field game model treated in this paper, which we briefly recall. Let $\Omega \subset \mathbb R^d$ be compact and $k: \mathcal P(\Omega) \times \Omega \to \mathbb R^+$ and $g: \partial\Omega \to \mathbb R^+$ be continuous (recall that $\mathcal P(\Omega)$ is endowed with the topology of weak convergence of measures). We consider the mean field game in which agents evolve in $\Omega$, their distribution at time $t$ being given by a probability measure $\rho_t \in \mathcal P(\Omega)$. We assume the initial distribution $\rho_0$ to be known. The goal of each agent is to leave $\Omega$ through its boundary $\partial\Omega$ minimizing the sum of their exit time with the cost $g(z)$ at their exit position $z \in \partial\Omega$. The speed of an agent at the position $x$ at time $t$ is assumed to be bounded by $k(\rho_t, x)$, which means that, for a given agent, their trajectory $\gamma$ satisfies $\abs{\gamma^\prime(t)} \leq k(\rho_t, \gamma(t))$, and thus depends on the distribution of all agents $\rho_t$. On the other hand, the distribution of the agents $\rho_t$ itself depends on how agents choose their trajectories $\gamma$. We are interested here in \emph{equilibrium} situations, i.e., in situations where, starting from a time evolution of the density of agents $\rho: \mathbb{R}^+ \to \mathcal P(\Omega)$, the trajectories $\gamma$ chosen by agents induce an evolution of the initial distribution of agents $\rho_0$ that is precisely given by $\rho$.

In this section, we first provide a precise definition of equilibrium and prove existence of equilibria, obtaining as well a system of PDEs, called the \emph{MFG system}, satisfied by the time-dependent measure $\rho_t$ and the value function of the corresponding optimal control problem. We then prove that, if $\rho_0$ is absolutely continuous with $L^p$ density, the same holds for $\rho_t$ for $t \geq 0$, with a control on its $L^p$ norm. Finally, thanks to these $L^p$ estimates, we extend the result of existence of equilibria and the corresponding MFG system to a case where $k$ is less regular.

\subsection{Existence of equilibria and the MFG system}
\label{SecExistenceMFG}

In order to provide the definition of equilibrium used in this paper, let us introduce some notation. Let $\Gamma = C(\mathbb R^+, \Omega)$. For a given $\gamma \in \Gamma$, we define its arrival time at $\partial\Omega$ by
\[\tau_{\gamma} := \inf\{s \geq 0 \suchthat \gamma(s) \in \partial\Omega \},\]
and, if $\tau_\gamma < +\infty$, we write
\[\gamma_\tau :=\gamma(\tau_\gamma) \in \partial\Omega.\]
Given $\rho: \mathbb R^+ \to \mathcal P(\Omega)$ and $x \in \Omega$, we define the set $\Gamma[\rho, x]$ of \emph{admissible trajectories} from $x$ by
\begin{align*}
\Gamma[\rho,x] := \bigl\{\gamma \in \Gamma \suchthat & \gamma(0)=x,\;\abs{{\gamma}^\prime(s)} \leq k(\rho_s,\gamma(s)) \text{ for a.e.\ } s \in (0,\tau_\gamma), \\
 & \text{ and } \gamma^\prime(s)=0 \text{ for every } s > \tau_\gamma\bigr\}.
\end{align*}
With these definitions, one can write the optimal control problem solved by each agent of the mean field game as
\begin{equation} \label{Minimal exit time}
\inf\left\{J(\gamma)\suchthat \gamma \in \Gamma[\rho,x]\right\},
\end{equation}
where
\[
J(\gamma) = 
\begin{dcases*}
\tau_{\gamma} + g(\gamma_\tau) & if $\tau_\gamma < +\infty$, \\
+\infty & otherwise.
\end{dcases*}
\]

\begin{remark} \label{RemkControlSyst}
If $\gamma \in \Gamma[\rho,x]$, then there is a measurable control $u: \mathbb{R}^+ \to \bar{B}(0,1)$ such that
\begin{equation}
\label{AdmissibleIsControlSystem}
\left\{
\begin{aligned}
\gamma^\prime(t) & = k(\rho_t, \gamma(t))u(t), & \quad & \text{for a.e.\ }t, \\ 
\gamma(0) & = x.
\end{aligned}
\right.
\end{equation}
System \eqref{AdmissibleIsControlSystem} can be seen as a control system under the form \eqref{control system} where the dynamic is given by $\widetilde{k}(t,x) = k(\rho_t,x)$ for every $(t,x) \in \mathbb{R}^+ \times \Omega$. This point of view allows one to formulate \eqref{Minimal exit time} as an optimal control problem when $\rho$ is fixed.
\end{remark}

\begin{remark}
\label{RemkConcentrationOnBoundary}
Due to the interaction between agents stemming from $k$, which may be non-local, the behavior of players who have not yet arrived at $\partial\Omega$ may be influenced by the players who already arrived. However, after arriving at $\partial\Omega$, players are no longer submitted to the minimization criterion \eqref{Minimal exit time}, and thus their trajectory $\gamma$ might in principle be arbitrary after their arrival time $\tau_\gamma$. The condition that $\gamma^\prime(s)=0$ for every $s > \tau_\gamma$ is imposed on admissible trajectories $\gamma$ in order to avoid ambiguity.

The above choice leads to a concentration of agents on the boundary, which is quite artificial from a modeling point of view. For this reason, one may consider, for modeling purposes, that, for $k$ given by \eqref{IntroK}, the function $\psi$ is a cut-off function, equal to $1$ everywhere on $\Omega$ except on a neighborhood of $\partial\Omega$ and vanishing at $\partial\Omega$ together with all its derivatives. In this way, the interaction term does not take into account agents who already left $\Omega$. Notice, however, that such assumptions on $k$ are not necessary for the results proved in this paper.
\end{remark}

We use in this paper a relaxed notion of MFG equilibrium based on a Lagrangian formulation, following the ideas in \cite{BenCarSan, CanCap, MazantiMinimal, Cardaliaguet2015Weak, Cardaliaguet2016First}, for which we give existence result. Such a formulation consists of replacing curves of probability measures on $\Omega$ with measures on arcs in $\Omega$. For any $t \in \mathbb{R}^+$, we denote by $e_t: \Gamma \to \Omega$ the evaluation map defined by 
$$e_t(\gamma)=\gamma(t), \quad \text{for all } \gamma \in \Gamma.$$
For any $\eta \in \mathcal{P}(\Gamma)$, we define the curve $\rho^\eta$ of probability measures on $\Omega$ as
$$\rho^\eta(t)=(e_t)_{\#} \eta, \quad \text{for all } t \in \mathbb{R}^+.$$
Since $e_t: \Gamma \to \Omega$ is continuous, we observe that if, for any $n \in \mathbb{N}$, $\eta_n \in \mathcal{P}(\Gamma)$ and $\eta \in \mathcal{P}(\Gamma)$  are such that $\eta_n \deb \eta$, then $\rho^{\eta_n}(t) \deb \rho^\eta(t)$ for all $t \in \mathbb{R}^+$. For any fixed $\rho_0 \in \mathcal{P}(\Omega)$, we denote by $\mathcal{P}_{\rho_0}(\Gamma)$ the set of all Borel probability measures $\eta$ on $\Gamma$ such that $(e_0)_{\#} \eta = \rho_0$. Notice that $\mathcal{P}_{\rho_0}(\Gamma)$ is nonempty, since it contains $j_{\#} \rho_0$, where $j: \Omega \to \Gamma$ is the continuous map defined by $j(x)(t)=x$ for all $t \in \mathbb{R}^+$. For all $x \in \Omega$ and $\eta \in \mathcal{P}_{\rho_0}(\Gamma)$, we define the set $\Gamma^\prime[\rho^\eta,x]$ of \emph{optimal trajectories} from $x$ by
$$\Gamma^\prime[\rho^\eta,x]:=\left\{\gamma \in \Gamma[\rho^\eta,x]\suchthat J(\gamma)=\min_{\Gamma[\rho^\eta,x]} J\right\}.$$
We also find it useful to introduce the set $\Gamma_{k_{\max}}$ of $k_{\max}$-Lipschitz trajectories $\gamma \in \Gamma$, i.e.,
$$\Gamma_{k_{\max}}=\{\gamma \in \Gamma \suchthat \abs{\gamma^\prime(t)} \leq k_{\max} \text{ for a.e.\ }t \in \mathbb R^+\}.$$
Recall that $\Gamma_{k_{\max}}$ is a compact subset of $\Gamma$, and, for every $\eta \in \mathcal P(\Gamma)$ and $x \in \Omega$, one has $\Gamma^\prime[\rho^\eta,x] \subset \Gamma[\rho^\eta,x] \subset \Gamma_{k_{\max}}$.

The definition of equilibrium used in this paper is the following.

\begin{definition}
\label{DefiEquilibrium}
Let $\rho_0 \in \mathcal P(\Omega)$. We say that $\eta \in \mathcal{P}_{\rho_0}(\Gamma)$ is a \emph{MFG equilibrium} for $\rho_0$ if
$$\spt(\eta) \subset \bigcup_{x \in \Omega} \Gamma^\prime[\rho^\eta,x].$$
\end{definition}

Let us state the assumptions used to guarantee existence of equilibria. The function $k: \mathcal P(\Omega) \times \Omega \to \mathbb R^+$ is assumed to be continuous. It is reasonable to suppose that $k$ is bounded from above, since it is not natural to assume that an agent's speed might approach $+\infty$. For simplicity, and in order to affirm that there is at least one admissible trajectory $\gamma$ starting from a point $x$ that reaches the boundary in finite time, we also suppose that $k$ is bounded from below. Hence, as in Section \ref{SecOptimalControl}, we assume that $k$ satisfies \eqref{lower bound on the dynamic}. We also suppose that the counterpart of \eqref{Hy1} holds, namely,
\begin{hypothesis} 
\label{HypoKLip-MFG}
\exists L_1 > 0 \quad \text{such that} \quad \abs{k(\mu,x_0) - k(\mu,x_1)} \leq L_1 \abs{x_0 - x_1} \quad \text{for all } x_0, x_1 \in \Omega,\;\mu \in \mathcal{P}(\Omega).
\end{hypothesis}
Notice that \eqref{lower bound on the dynamic} and \eqref{HypoKLip-MFG} are satisfied for \eqref{IntroK} if $V: \mathbb R^+ \to (0, +\infty)$ and $\chi: \mathbb R^d \to \mathbb R^+$ are Lipschitz continuous and $\psi: \mathbb R^d \to \mathbb R^+$ is continuous. Moreover, we suppose, as in Section \ref{SecOptimalControl}, that $g: \partial\Omega \to \mathbb R^+$ satisfies \eqref{H2}. In particular, from Proposition \ref{PropExistOptim}, we infer that \eqref{Minimal exit time} reaches a minimum.

We can now state our result on the existence of equilibria.

\begin{theorem} \label{equilibre}
Let $\rho_0 \in \mathcal P(\Omega)$, $k: \mathcal P(\Omega) \times \Omega \to \mathbb R^+$ be continuous, $g: \partial\Omega \to \mathbb R^+$, and assume that \eqref{lower bound on the dynamic}, \eqref{H2}, and \eqref{HypoKLip-MFG} hold. Then there exists a MFG equilibrium for $\rho_0$.
\end{theorem}

The proof of Theorem \ref{equilibre} is based on the same fixed-point strategy used in \cite{MazantiMinimal, CanCap}. Notice that the above theorem is slightly stronger than \cite[Theorem 5.1]{MazantiMinimal} since existence of equilibria is obtained under weaker assumptions. For this reason, and also for the sake of completeness, we provide a detailed proof of Theorem \ref{equilibre}. The first step is the following property of the map $(\eta, x) \mapsto \Gamma^\prime[\rho^\eta, x]$.

\begin{lemma} \label{8.4}
Let $\rho_0$, $k$, and $g$ be as the statement of Theorem \ref{equilibre}. Let $(\eta_n)_n$ be a sequence in $\mathcal P_{\rho_0}(\Gamma)$, $(x_n)_n$ a sequence in $\Omega$, and $(\gamma_n)_n$ a sequence in $\Gamma$ such that $\gamma_n \in \Gamma^\prime[\rho^{\eta_n}, x_n]$ for every $n \in \mathbb N$ and $\eta_n \deb \eta$, $x_n \to x$, and $\gamma_n \to \bar\gamma$ for some $\eta \in \mathcal P_{\rho_0}(\Gamma)$, $x \in \Omega$, and $\bar\gamma$ in $\Gamma$. Then $\bar{\gamma} \in \Gamma^\prime[\rho^{\eta},x]$. Consequently, $(\eta,x) \mapsto \Gamma^\prime[\rho^\eta,x]$ has a closed graph.
\end{lemma}

\begin{proof}
We set, for simplicity, $\tau_n:=\tau_{\gamma_n}$ and $z_n:=\gamma_{n}(\tau_n)$. Using Proposition \ref{PropBoundTau}, $(\tau_n)_n$ is bounded and, up to extracting a subsequence, $\tau_n$ converges to some $\bar{\tau}$. On the other hand, we see easily that $\bar\gamma$ is $k_{\max}$-Lipschitz continuous. In addition, for a.e.\ $t \in (0, \tau_n)$, we have $\abs{{\gamma}_n^\prime(t)} \leq k(\rho^{\eta_n}(t), \gamma_n(t))$. Letting $n \to +\infty$, we get that $\abs{{\bar{\gamma}}^\prime(t)} \leq k(\rho^{\eta}(t),\bar{\gamma}(t))$ for a.e.\ $t \in (0, \bar\tau)$. In the same way, one can prove that $\bar{\gamma}^\prime(t)=0$ for all $t > \bar{\tau}$. Moreover, we have $z_n \to \bar{\gamma}(\bar{\tau})$, which implies that $\bar{\gamma}(\bar{\tau}) \in \partial\Omega$ and $\tau:=\tau_{\bar{\gamma}} \leq \bar{\tau}$. Notice that $\bar\gamma \in \Gamma[\rho^\eta, x]$ if and only if $\tau = \bar\tau$.

Define the trajectory $\gamma \in \Gamma[\rho^\eta,x]$ by
$$\gamma(t)=\begin{cases}
\bar{\gamma}(t), & \text{if } t \leq \tau,\\
\bar{\gamma}(\tau), & \text{if } t > \tau.
\end{cases}$$
Notice that, by Lemma \ref{LemmTimePlusGIncreases}, $J(\gamma) \leq \bar\tau + g(\bar\gamma(\bar\tau))$, with a strict inequality if and only if $\tau < \bar\tau$. Suppose, to obtain a contradiction, that $\bar\gamma \notin \Gamma^\prime[\rho^\eta,x]$. Then there exists a trajectory $\widehat{\gamma} \in \Gamma^\prime[\rho^\eta,x]$ such that $J(\widehat{\gamma}) < \bar\tau + g(\bar\gamma(\bar\tau))$; indeed, this follows by the definition of $\Gamma^\prime[\rho^\eta,x]$ if $\tau = \bar\tau$ or by the fact that $\gamma \in \Gamma[\rho^\eta,x]$ and $J(\gamma) < \bar\tau + g(\bar\gamma(\bar\tau))$ if $\tau < \bar\tau$.

For each $n \in \mathbb N$, let $\widetilde{\gamma}_n: [0, \abs{x_n - x}] \to \mathbb R^d$ be the segment from $x_n$ to $x$ with $\abs{\widetilde{\gamma}_n^\prime(t)}=1$ for every $t \in [0, \abs{x_n - x}]$. Let $\phi_n: [k_{\min}^{-1} \abs{x_n - x}, +\infty) \to \mathbb R^+$ be a function satisfying 
\begin{equation}
\label{PhiNDiffEqn}
\begin{dcases}
\phi_n^\prime(t)= \frac{k(\rho^{\eta_n}(t),\widehat{\gamma}(\phi_n(t)))}{k(\rho^{\eta}(\phi_n(t)),\widehat{\gamma}(\phi_n(t)))},\\
\phi_n(k_{\min}^{-1} \abs{x_n - x})=0.
\end{dcases}
\end{equation}
Define $\widehat{\gamma}_n: \mathbb R^+ \to \mathbb R^d$ by
$$\widehat{\gamma}_n(t)=\begin{cases}
\widetilde{\gamma}_n(k_{\min} t) & \text{if } t \in [0, k_{\min}^{-1} \abs{x_n - x}],\\
\widehat{\gamma}(\phi_n(t)) & \text{otherwise}.
\end{cases}$$
One has $\widehat\gamma_n(\phi_n^{-1}(\tau_{\widehat\gamma})) = \widehat\gamma(\tau_{\widehat\gamma})$, and thus $\tau_{\widehat\gamma_n} \leq \phi_n^{-1}(\tau_{\widehat\gamma})$. Hence, by Lemma \ref{LemmTimePlusGIncreases},
\begin{equation}
\label{IneqCosts}
\tau_{\widehat\gamma_n} + g\bigl(\widehat\gamma_n(\tau_{\widehat\gamma_n})\bigr) \leq \phi_n^{-1}(\tau_{\widehat\gamma}) + g\Bigl(\widehat\gamma_n\bigl(\phi_n^{-1}(\tau_{\widehat\gamma})\bigr)\Bigr) = \phi_n^{-1}(\tau_{\widehat\gamma}) + g(\widehat\gamma_\tau).
\end{equation}
We modify $\widehat\gamma_n$ on the interval $(\tau_{\widehat\gamma_n}, +\infty)$ by setting $\widehat\gamma_n(t) = \widehat\gamma_n(\tau_{\widehat\gamma_n})$ for $t > \tau_{\widehat\gamma_n}$. This modification does not change $\tau_{\widehat\gamma_n}$ and one has now $\widehat\gamma_n \in \Gamma[\rho^{\eta_n}, x_n]$. In particular, \eqref{IneqCosts} reads
\begin{equation}
\label{8.7.1}
J(\widehat\gamma_n) \leq \phi_n^{-1}(\tau_{\widehat\gamma}) + g(\widehat\gamma_\tau).
\end{equation}

Since $(\phi_n)_n$ and $(\phi_n^{-1})_n$ are equi-Lipschitz sequences, it follows from Arzelà--Ascoli Theorem that, up to extracting subsequences, there exists a bi-Lipschitz function $\phi: \mathbb R^+ \to \mathbb R^+$ such that $\phi_n \to \phi$ and $\phi_n^{-1} \to \phi^{-1}$ uniformly on compact sets of $\mathbb R^+$. In addition, it is easy to check by integrating \eqref{PhiNDiffEqn} that, for all $t \in [k_{\min}^{-1} \abs{x_n - x}, +\infty)$,
$$\int_0^{\phi_n(t)} k(\rho^{\eta}(s),\widehat{\gamma}(s))\diff s= \int_{k_{\min}^{-1} \abs{x_n - x}}^t k(\rho^{\eta_n}(s),\widehat{\gamma}(\phi_n(s)))\diff s.$$
So, letting $n \to +\infty$, we get, for all $t \in \mathbb R^+$,
$$\int_0^{\phi(t)} k(\rho^{\eta}(s),\widehat{\gamma}(s))\diff s= \int_{0}^t k(\rho^{\eta}(s),\widehat{\gamma}(\phi(s)))\diff s.$$
Set 
$$G(\theta)=\int_0^\theta k(\rho^{\eta}(s),\widehat{\gamma}(s))\diff s, \quad \forall \theta \in \mathbb{R}^+.$$
Then $G: \mathbb R^+ \to \mathbb R^+$ is a bi-Lipschitz bijection and, for $t \in \mathbb R^+$,
\begin{align*}
\abs{\phi(t) - t } & = \abs[\bigg]{G^{-1}\biggl(\int_{0}^t k(\rho^{\eta}(s),\widehat{\gamma}(\phi(s)))\diff s\biggr) - G^{-1} \biggl(\int_{0}^t k(\rho^{\eta}(s),\widehat{\gamma}(s))\diff s\biggr)}\\ 
& \leq C \int_{0}^t \abs{k(\rho^{\eta}(s),\widehat{\gamma}(\phi(s))) - k(\rho^{\eta}(s),\widehat{\gamma}(s))}\diff s\\ 
& \leq C \int_{0}^t \abs{\phi(s) - s}\diff s.
\end{align*}
By Grönwall's lemma, we get that $\phi(t)=t$ for all $t \in \mathbb{R}^+$. Passing to the limit in \eqref{8.7.1}, we get
\begin{equation} \label{8.7.2}
\limsup_n J(\widehat{\gamma}_n) \leq \tau_{\widehat{\gamma}} + g(\widehat{\gamma}_\tau) = J(\widehat{\gamma}) < \bar\tau + g(\bar\gamma(\bar\tau)).
\end{equation}
Yet, 
$$\lim_n J(\gamma_n)= \lim_n \tau_n + g(z_n)=\bar{\tau} + g(\bar{\gamma}(\bar{\tau})).$$
Using \eqref{8.7.2}, we infer that, for $n$ large enough, 
$$J(\widehat{\gamma}_n) < J(\gamma_n),$$
which is a contradiction, as $\widehat{\gamma}_n \in \Gamma[\rho^{\eta_n},x_n]$ and $\gamma_n \in \Gamma^\prime[\rho^{\eta_n},x_n]$. Then $\bar{\gamma} \in \Gamma^\prime[\rho^{\eta},x]$.
\end{proof}

\begin{remark}
\label{RemkReformulationEquilibrium}
As a consequence of Lemma \ref{8.4}, for a given $\eta \in \mathcal P_{\rho_0}(\Gamma)$, the graph $G$ of the map $x \mapsto \Gamma^\prime[\rho^\eta, x]$ is closed in $\Omega \times \Gamma$. Since $\Gamma^\prime[\rho^\eta, x] \subset \Gamma_{k_{\max}}$, $G$ is compact, since it is a closed subset of the compact set $\Omega \times \Gamma_{k_{\max}}$. Hence, the set $\bigcup_{x \in \Omega} \Gamma^\prime[\rho^\eta,x]$, which is the projection of $G$ onto $\Gamma$, is also compact, and, in particular, a measure $\widetilde\eta \in \mathcal P(\Gamma)$ satisfies $\spt(\widetilde\eta) \subset \bigcup_{x \in \Omega} \Gamma^\prime[\rho^\eta,x]$ if and only if $\widetilde\eta\left[\bigcup_{x \in \Omega} \Gamma^\prime[\rho^\eta,x]\right] = 1$.

In particular, one can reformulate Definition \ref{DefiEquilibrium} in an equivalent way by saying that $\eta \in \mathcal P_{\rho_0}(\Gamma)$ is a MFG equilibrium for $\rho_0$ if
\[
\eta\left[\bigcup_{x \in \Omega} \Gamma^\prime[\rho^\eta,x]\right] = 1,
\]
i.e., if for $\eta$-a.e.\ $\bar{\gamma} \in \Gamma$, we have 
$$ J(\bar{\gamma}) \leq J(\gamma), \quad \text{for all } \gamma \in \Gamma[\rho^\eta,\bar{\gamma}(0)].$$ 
\end{remark}

We now reformulate the notion of equilibrium as a fixed point problem, in order to prove Theorem \ref{equilibre} using a fixed-point argument. We introduce the set-valued map $E: \mathcal{P}_{\rho_0}(\Gamma) \rightrightarrows \mathcal{P}_{\rho_0}(\Gamma)$ given, for $\eta \in \mathcal{P}_{\rho_0}(\Gamma)$, by
$$E(\eta) = \biggl\{\widetilde \eta \in \mathcal P_{\rho_0}(\Gamma) \suchthat \spt(\widetilde \eta) \subset \bigcup_{x \in \Omega} \Gamma^\prime[\rho^\eta,x]\biggr\}.$$
It follows immediately that $\eta \in \mathcal{P}_{\rho_0}(\Gamma)$ is a MFG equilibrium for $\rho_0$ if and only if $\eta \in E(\eta)$, which is precisely the definition of fixed point for a set-valued map. We will therefore prove Theorem \ref{equilibre} by showing that $E$ admits a fixed point using Kakutani's Theorem (see, e.g., \cite[\S 7, Theorem 8.6]{Granas2003Fixed}, \cite{Kakutani}), whose assumptions we verify in the next lemma.

\begin{lemma} \label{8.5}
Let $\rho_0$, $k$, and $g$ be as the statement of Theorem \ref{equilibre}. Then
\begin{enumerate}
\item\label{LemmKakutaniA} for any $\eta \in \mathcal{P}_{\rho_0}(\Gamma)$, $E(\eta)$ is a nonempty convex set; and
\item\label{LemmKakutaniB} $E: \mathcal{P}_{\rho_0}(\Gamma) \rightrightarrows \mathcal{P}_{\rho_0}(\Gamma)$ has a closed graph.
\end{enumerate}
\end{lemma}

\begin{proof}
To prove \ref{LemmKakutaniA}, fix $\eta \in \mathcal P_{\rho_0}(\Gamma)$. Using Remark \ref{RemkReformulationEquilibrium}, one immediately verifies that $E(\eta)$ is convex. To see that it is nonempty, notice that, by Lemma \ref{8.4} and \cite[Theorem 8.1.3]{Aubin}, the map $x \mapsto \Gamma^\prime[\rho^\eta,x]$ has a Borel measurable selection $\pmb\gamma^\eta: x \mapsto \gamma_x^\eta \in \Gamma^\prime[\rho^\eta,x]$, and one immediately verifies that $\pmb\gamma^\eta_{\#} \rho_0 \in E(\eta)$.

Now, to prove \ref{LemmKakutaniB}, let $(\eta_n)_n$ and $(\widehat\eta_n)_n$ be sequences in $\mathcal{P}_{\rho_0}(\Gamma)$ and $\eta, \widehat\eta \in \mathcal{P}_{\rho_0}(\Gamma)$ such that $\widehat{\eta}_n \in E(\eta_n)$ for every $n \in \mathbb N$, $\eta_n \deb \eta$, and $\widehat{\eta}_n \deb \widehat{\eta}$. For $k \in \mathbb{N}^\star$, let $V_k:=\{\gamma \in \Gamma \suchthat \mathbf d(\gamma,\bigcup_{x}\Gamma^\prime[\rho^\eta,x])\leq \frac{1}{k}\}$, where we recall that $\mathbf d(\gamma, A)$ denotes the usual distance between a curve $\gamma \in \Gamma$ and a set $A \subset \Gamma$ and the metric $\mathbf d$ in $\Gamma$ is compatible with the topology of uniform convergence of compact sets (e.g., the metric defined in \eqref{DefiDistGamma}). Notice that the graph of the set-valued map $\widetilde\eta \mapsto \bigcup_{x} \Gamma^\prime[\rho^{\widetilde{\eta}},x]$ is closed, since it is the projection onto $\mathcal P_{\rho_0}(\Gamma) \times \Gamma$ of the graph of the set-valued map from Lemma \ref{8.4}. Then, using \cite[Proposition 1.4.8]{Aubin}, it follows that there exists a neighborhood $W$ of $\eta$ such that $\bigcup_{x} \Gamma^\prime[\rho^{\widetilde{\eta}},x] \subset V_k$ for every $\widetilde \eta \in W$. Then, for $n$ large enough, $\bigcup_{x} \Gamma^\prime[\rho^{\eta_n},x] \subset V_k$. Since $\widehat{\eta}_n (\bigcup_{x} \Gamma^\prime[\rho^{\eta_n},x]) = 1$, one obtains that $\widehat{\eta}_n (V_k) = 1$, for large $n$. Yet, $\widehat{\eta}_n \deb \widehat{\eta}$ and $V_k$ is closed, hence it follows that $\widehat{\eta} (V_k) \geq \limsup_{n} \widehat{\eta}_n (V_k) = 1$ and thus, $\widehat{\eta} (V_k) = 1$. As this holds for every $k \in \mathbb{N}^\star$, one concludes that $\widehat{\eta} (\bigcup_{x} \Gamma^\prime[\rho^{\eta},x]) = 1$. Hence $\widehat{\eta} \in E(\eta)$, which proves that the graph of $E$ is closed.
\end{proof}

\begin{remark} \label{Req 3.4}
The set $\mathcal{P}_{\rho_0}(\Gamma_{k_{\max}})$ is a compact convex subset of $\mathcal{P}_{\rho_0}(\Gamma)$. Indeed, the convexity of $\mathcal{P}_{\rho_0}(\Gamma_{k_{\max}})$ follows immediately. As for compactness, if $(\eta_k)_k$ is a sequence in $\mathcal{P}_{\rho_0}(\Gamma_{k_{\max}})$, then, since $\Gamma_{k_{\max}}$ is compact, $(\eta_k)_k$ is tight, and so, by Prokhorov's Theorem, one finds a subsequence which converges weakly to some probability measure $\eta \in \mathcal{P}_{\rho_0}(\Gamma_{k_{\max}})$.
\end{remark}

Notice that, by the definition of $E$, we have
$$E(\eta) \subset \mathcal{P}_{\rho_0}(\Gamma_{k_{\max}}), \quad \text{for all } \eta \in \mathcal{P}_{\rho_0}(\Gamma).$$
In particular, any fixed point of $E$ belongs to $\mathcal{P}_{\rho_0}(\Gamma_{k_{\max}})$. We will thus restrict our domain of interest to $\mathcal{P}_{\rho_0}(\Gamma_{k_{\max}})$ with no loss of generality, denoting hereafter by $E$ the restriction $E\rvert_{\mathcal{P}_{\rho_0}(\Gamma_{k_{\max}})}$. Notice that Lemma \ref{8.5} still holds for this restriction. One can now complete the proof of Theorem \ref{equilibre}.

\begin{proof}[Proof of Theorem \ref{equilibre}]
Lemma \ref{8.5} guarantees that the set-valued map $E$ has a closed graph and $E(\eta)$ is a nonempty convex set for any $\eta \in \mathcal{P}_{\rho_0}(\Gamma_{k_{\max}})$. Since, by Remark \ref{Req 3.4}, $\mathcal{P}_{\rho_0}(\Gamma_{k_{\max}})$ is a nonempty compact convex set, all assumptions of Kakutani's Theorem are satisfied and thus there exists $\eta \in \mathcal{P}_{\rho_0}(\Gamma_{k_{\max}})$ such that $\eta \in E(\eta)$, i.e., $\eta$ is a MFG equilibrium for $\rho_0$.
\end{proof}

\begin{remark}
Given $\rho_0 \in \mathcal P(\Omega)$, one may have several MFG equilibria for $\rho_0$, as one may see from the following example taken from \cite[Remark 7.1]{MazantiMinimal}. Let $\Omega = (0, 1)$, $g = 0$, and $k = 1$. Assume that $\rho_0$ is the Dirac delta measure on the point $\frac{1}{2}$. Let $\gamma_l, \gamma_r \in \Gamma$ be given for $t \in \mathbb R^+$ by $\gamma_l(t) = \max\left(\frac{1}{2} - t, 0\right)$ and $\gamma_r(t) = \min\left(\frac{1}{2} + t, 1\right)$. Then any $\eta \in \mathcal P(\Gamma)$ concentrated on $\gamma_l, \gamma_r$ (i.e., satisfying $\eta(\{\gamma_l, \gamma_r\}) = 1$) is a MFG equilibrium for $\rho_0$. This example can be generalized for any $\Omega \subset \mathbb R^d$, taking $g = 0$ and $k = 1$, by considering initial distributions $\rho_0 \in \mathcal P(\Omega)$ concentrated on the set where the distance function $\mathbf d(\cdot, \partial\Omega)$ is not differentiable.

For other models of mean field games, uniqueness of equilibria is typically obtained under some monotonicity assumptions on functions appearing in the cost of each player (see, e.g., \cite[Proposition 2.9 and Theorem 3.6]{CardaliaguetNotes}, \cite[Theorem 4.1]{Lasry2006}, and \cite[Theorem 3.1]{Lasry2006Jeux}). Typically, these monotonicity assumptions mean that players tend to avoid congested regions, and they are important for uniqueness since games in which players tend to aggregate may present several equilibria (see, e.g., \cite{Cirant2019Time}). In our setting, it is not clear whether suitable congestion-avoidance assumptions should be sufficient for obtaining uniqueness of equilibria.
\end{remark}

Now that existence of a MFG equilibrium $\eta \in \mathcal P_{\rho_0}(\Gamma)$ has been established, we wish to prove that, similarly to most mean field game models, the corresponding time-dependent measure $\rho_t = \rho^\eta(t)$ satisfies, together with the value function $\varphi$ of the corresponding optimal control problem, a system of PDEs, known as MFG system, composed of a continuity equation under the form $\partial_t \rho + \nabla \cdot (\rho v) = 0$ for some velocity field $v$ and a Hamilton--Jacobi equation on $\varphi$. The Hamilton--Jacobi equation on $\varphi$ is the one from Proposition \ref{MainTheoHJ}, and one can easily obtain that $\rho$ satisfies some continuity equation (for instance, by proving that $t \mapsto \rho_t$ is Lipschitz continuous with respect to the Wasserstein distance $W_p$ for $p > 1$, as in \cite[Proposition 5.2(a)]{MazantiMinimal}, and then applying \cite[Theorem 8.3.1]{Ambrosio2005Gradient}). The main point here is to identify the velocity field of the continuity equation. To do so, we shall use the results from Section \ref{SecDifferentiabilityVarphi}, which in particular require assumption \eqref{Smoothness of the gradient of the dynamic with respect to time}. We then introduce the following notion.

\begin{definition}
Let $k: \mathcal P(\Omega) \times \Omega \to \mathbb R^+$ be continuous, $g: \partial\Omega \to \mathbb R^+$, $\rho_0 \in \mathcal P(\Omega)$, and assume that \eqref{lower bound on the dynamic}, \eqref{H2}, and \eqref{HypoKLip-MFG} hold. We say that $k$ is \emph{$C^{1, 1}$ on MFG equilibria for $\rho_0$} if, for every MFG equilibrium $\eta \in \mathcal P_{\rho_0}(\Gamma)$ for $\rho_0$, the function $(t, x) \mapsto k(\rho^\eta(t), x)$ is $C^{1, 1}$ on $\mathbb R^+ \times \Omega$.
\end{definition}

To motivate this definition, we prove that the function $k$ given by \eqref{IntroK},
\begin{equation*} 
k(\mu, x) = V\left(\int_\Omega \chi(x - y) \psi(y) \diff \mu(y)\right),
\end{equation*}
is $C^{1, 1}$ on MFG equilibria under suitable regularity assumptions on $V$, $\chi$, $\psi$, $\partial\Omega$, and $g$.

\begin{proposition}
\label{PropKC11MFGEquil}
Let $V \in C^{1, 1}(\mathbb R^+, (0, +\infty))$ be Lipschitz continuous, $\chi \in C^{1, 1}(\mathbb R^d,\allowbreak \mathbb R^+)$, $\psi \in C^{1, 1}(\mathbb R^d, \mathbb R^+)$, $k: \mathcal P(\Omega) \times \mathbb R^d \to \mathbb R^+$ be given by \eqref{IntroK}, $k_{\max} = \sup_{\mathbb R^+ \times \Omega} k$, and $\rho_0 \in \mathcal P(\Omega)$. Suppose that $\psi(x) = 0$ and $\nabla\psi(x) = 0$ for every $x \in \partial\Omega$, \eqref{Smoothness of the boundary} holds, and $g: \partial\Omega \to \mathbb R^+$ satisfies \eqref{H2} and \eqref{Smoothness of the boundary cost}. Then $k$ is $C^{1, 1}$ on MFG equilibria for $\rho_0$.
\end{proposition}

\begin{proof}
Notice first that $k: \mathcal P(\Omega) \times \mathbb R^d \to \mathbb R^+$ is continuous and satisfies \eqref{lower bound on the dynamic} and \eqref{HypoKLip-MFG}. Let $\rho_0 \in \mathcal P(\Omega)$, $\eta \in \mathcal P_{\rho_0}(\Gamma_{k_{\max}})$ be a MFG equilibrium for $\rho_0$, and $\rho_t = \rho^\eta(t)$ for $t \geq 0$. Let $\theta: \mathbb R^+ \times \mathbb R^d \to \mathbb R^+$ be given by
\[
\theta(t, x) = \int_{\Omega} \chi(x - y) \psi(y) \diff \rho_t(y).
\]
Since $V \in C^{1, 1}(\mathbb R^+, (0, +\infty))$, it suffices to prove that $\theta \in C^{1, 1}(\mathbb R^+ \times \mathbb R^d, \mathbb R^+)$.

Set $\Gamma^\prime = \bigcup_{x \in \Omega} \Gamma^\prime[\rho^\eta, x] \subset \Gamma_{k_{\max}}$. Since $\eta$ is a MFG equilibrium, one has $\eta(\Gamma^\prime) = 1$. Notice that
\[
\theta(t, x) = \int_{\Gamma^\prime} \chi(x - \gamma(t)) \psi(\gamma(t)) \diff \eta(\gamma),
\]
and, since every $\gamma \in \Gamma^\prime$ is $k_{\max}$-Lipschitz, one obtains that $(t, x) \mapsto k(\rho_t, x)$ is Lipschitz continuous.

For $(t, x) \in \mathbb R^+ \times \mathbb R^d$,
\[
\nabla \theta(t, x) = \int_{\Omega} \nabla\chi(x - y) \psi(y) \diff \rho_t(y) = \int_{\Gamma^\prime} \nabla\chi(x - \gamma(t)) \psi(\gamma(t)) \diff\eta(\gamma),
\]
and this function can be easily seen to be Lipschitz continuous on $\mathbb R^+ \times \mathbb R^d$. In particular, the function $(t, x) \mapsto k(\rho_t, x)$ satisfies \eqref{Smoothness of the gradient of the dynamic} and \eqref{Lipschitz regularity of the gradient with respect TO x}. Hence, the results of Section \ref{SecOC-PMP} apply to the optimal control problem \eqref{Minimal exit time}, and, in particular, by Proposition \ref{RegCurve}, one obtains that $\gamma \in C^{1, 1}([0, \tau_\gamma], \Omega)$ for every $\gamma \in \Gamma^\prime$.

For every $\gamma \in \Gamma^\prime$, the function $t \mapsto \chi(x - \gamma(t)) \psi(\gamma(t))$ is differentiable everywhere on $\mathbb R^+$, except possibly at $t = \tau_\gamma$, with
\begin{equation*}
\frac{\diff}{\diff t} \bigl[\chi(x - \gamma(t)) \psi(\gamma(t))\bigr] = -\nabla\chi(x - \gamma(t)) \cdot \gamma^\prime(t) \psi(\gamma(t)) + \chi(x - \gamma(t)) \nabla\psi(\gamma(t)) \cdot \gamma^\prime(t).
\end{equation*}
Since $\psi(x) = 0$ and $\nabla\psi(x) = 0$ for $x \in \partial\Omega$ and $\gamma(t) \in \partial\Omega$ for $t = \tau_\gamma$, one can also prove that the above function is differentiable and its derivative is zero at $t = \tau_\gamma$. Moreover, its derivative is Lipschitz continuous and upper bounded, and thus $\partial_t \theta(t, x)$ exists, with
\[
\partial_t \theta(t, x) = \int_{\Gamma^\prime} \Bigl[-\nabla\chi(x - \gamma(t)) \cdot {\gamma}^\prime(t) \psi(\gamma(t)) + \chi(x - \gamma(t)) \nabla\psi(\gamma(t)) \cdot {\gamma}^\prime(t)\Bigr] \diff \eta(\gamma),
\] 
and one immediately verifies using the previous assumptions that $\partial_t \theta$ is Lipschitz continuous in $\mathbb R^+ \times \mathbb R^d$. Together with the corresponding property for $\nabla \theta$, we obtain that $\theta \in C^{1, 1}(\mathbb R^+ \times \mathbb R^d, \mathbb R^+)$.
\end{proof}

We now show that, for every MFG equilibrium $\eta$, $\rho^\eta$ and the corresponding value function satisfy a MFG system.

\begin{theorem}
\label{TheoMFGSystem}
Let $k: \mathcal P(\Omega) \times \Omega \to \mathbb R^+$ be continuous, $g: \partial\Omega \to \mathbb R^+$, $\rho_0 \in \mathcal P(\Omega)$, and assume that \eqref{lower bound on the dynamic}, \eqref{H2}, \eqref{Smoothness of the boundary}, \eqref{boundary cost semiconcave}, and \eqref{HypoKLip-MFG} hold. Suppose that $k$ is $C^{1, 1}$ on MFG equilibria for $\rho_0$. Let $\eta \in \mathcal P_{\rho_0}(\Omega)$ be a MFG equilibrium for $\rho_0$, $\rho = \rho^\eta$, and $\varphi$ be the value function of the optimal control problem \eqref{Minimal exit time} with dynamic $(t, x) \mapsto k(\rho_t, x)$. Then $(\rho, \varphi)$ solve the \emph{MFG system}
\begin{equation}
\label{MFGSystem}
\left\{
\begin{aligned}
& \partial_t \rho(t, x) - \nabla \cdot \left(\rho(t, x) k(\rho_t, x) \frac{\nabla\varphi (t, x)}{\abs{\nabla \varphi (t,x)}}\right) = 0, & \quad & (t, x) \in (0,\infty) \times \accentset\circ{\Omega}, \\ 
& -\partial_t \varphi(t, x) + k(\rho_t, x) \abs{\nabla \varphi(t, x)} - 1 = 0, & & (t, x) \in \mathbb{R}^+ \times \Omega, \\ 
& \varphi(t, x) = g(x), & & (t, x) \in \mathbb R^+ \times \partial\Omega,\\
& \rho(0,x) = \rho_0(x), & & x \in \Omega,
\end{aligned}
\right.
\end{equation}
where the first and second equations are satisfied, respectively, in the sense of distributions and in the viscosity sense.
\end{theorem}

\begin{proof}
The second equation in \eqref{MFGSystem} and the corresponding boundary condition have already been established in Proposition \ref{MainTheoHJ}. We are left to prove that $\rho$ satisfies the continuity equation in \eqref{MFGSystem}.

Let $\phi \in C^\infty_{\mathrm c}((0,\infty) \times \accentset\circ{\Omega})$ and set $\Gamma^\prime = \bigcup_{x \in \Omega} \Gamma^\prime[\rho^\eta, x]$. Then, recalling Theorem \ref{Differentiability of the value function} and Corollary \ref{CoroAlmostVelocityField}, we have
\begin{align*}
& -\int_0^{+\infty} \int_{\Omega} \partial_t \phi(t, x) \diff \rho_t(x) \diff t + \int_0^{+\infty} \int_{\Omega} k(\rho_t, x) \nabla \phi(t, x) \cdot \frac{\nabla\varphi(t,x)}{\abs{\nabla\varphi(t,x)}}\diff \rho_t(x) \diff t \displaybreak[0] \\
{} = {} & -\int_0^{+\infty} \int_{\Gamma^\prime} \partial_t \phi(t, \gamma(t)) \diff \eta(\gamma) \diff t + \int_0^{+\infty} \int_{\Gamma^\prime} k(\rho_t, \gamma(t)) \nabla \phi(t, \gamma(t)) \cdot \frac{\nabla\varphi(t, \gamma(t))}{\abs{\nabla\varphi(t, \gamma(t))}}\diff \eta(\gamma) \diff t \displaybreak[0] \\
{} = {} & -\int_0^{+\infty} \int_{\Gamma^\prime} \partial_t \phi(t, \gamma(t)) \diff \eta(\gamma) \diff t - \int_0^{+\infty} \int_{\Gamma^\prime} \nabla \phi(t, \gamma(t)) \cdot \gamma^\prime(t) \diff \eta(\gamma) \diff t \displaybreak[0] \\
{} = {} & -\int_{\Gamma^\prime}\int_0^{+\infty} \frac{\diff}{\diff t}\bigl[\phi(t, \gamma(t))\bigr]\diff t \diff \eta(\gamma) = 0. \qedhere
\end{align*}
\end{proof}

\subsection{\texorpdfstring{$L^p$}{Lp} estimates}
\label{SecLp}

Recall that our motivation for the mean field game model in this paper comes from crowd motion, where a reasonable expression for $k$ is \eqref{IntroK}. In order to apply the existence result from Theorem \ref{equilibre} to this setting, one should require the function $\psi$ in \eqref{IntroK} to be at least continuous. On the other hand, as stated in Remark \ref{RemkConcentrationOnBoundary}, agents concentrate on the boundary. A reasonable feature of our model would be to assume that agents do not take into account in their congestion term other agents that have already left the domain, which can be done by assuming that $\psi(x) = 0$ for $x \in \partial\Omega$. However, due to the continuity of $\psi$, this implies that agents that are too close to the boundary, but have not yet left, will also be somehow discounted.

From a modeling point of view, an interesting choice would be to take $\psi = \mathbbm 1_{\accentset\circ\Omega}$, but this yields a function $k$ that is discontinuous on measures $\mu$ such that $\mu(\partial\Omega) > 0$, and the arguments used in the proof of Theorem \ref{equilibre} do not apply. On the other hand, one may still expect to have existence of equilibria, at least when $\rho_0$ is absolutely continuous with respect to the Lebesgue measure. The goal of this section and the following is to establish a result on the existence of equilibria in this setting. We first prove that, as soon as $\rho_0$ is absolutely continuous and with an $L^p$ density, $\rho_t\rvert_{\accentset\circ\Omega}$ is also absolutely continuous and with an $L^p$ density, with a control on the $L^p$ norm that is, in some sense, independent of $\psi$. This will be a key result for the proof of existence of an equilibrium with $\psi = \mathbbm 1_{\accentset\circ\Omega}$ in Section \ref{SecExistenceLessRegular}, which is based on a limit argument on a sequence $\psi_\varepsilon$ converging to $\mathbbm 1_{\accentset\circ\Omega}$ as $\varepsilon \to 0$.

The main difficulty in providing an $L^p$ control of the norm of $\rho_t$ comes from the fact that the velocity field $(t, x) \mapsto - k(\rho_t, x) \frac{\nabla\varphi(t, x)}{\abs{\nabla\varphi(t, x)}}$ of the continuity equation in \eqref{MFGSystem} is not smooth. Solutions of continuity equations with smooth velocity fields can be represented, using classical arguments, as the push-forward of the corresponding initial condition through the flow of the ordinary differential equation defined by the velocity field (see, e.g., \cite[Chapter 8]{Ambrosio2005Gradient}), and hence $L^p$ estimates can be obtained from lower bounds on the Jacobian of this flow, which in turn follow from lower bounds on the divergence of the velocity field. Major results for transport equations with velocity fields in Sobolev spaces have been obtained in \cite{DiPerna1989Ordinary} based on the method of renormalized solutions, with further results for velocity fields with less regularity provided in \cite{Ambrosio2008Transport}. However, several results of \cite{Ambrosio2008Transport} require the divergence of the velocity field to be absolutely continuous with respect to the Lebesgue measure, which is not necessarily the case in our setting due to the lack of regularity of $\varphi$. Our strategy relies instead on regularizing the velocity field and obtaining $L^p$ bounds for the solution of the corresponding continuity equation, the desired $L^p$ bounds on $\rho_t$ being obtained by a limit procedure on the regularization parameter.

The control of the $L^p$ norm we prove in this section depends essentially on the semi-concavity constant of the value function $\varphi$ at equilibrium. On the other hand, for $k$ given by \eqref{IntroK}, it follows from Theorem \ref{Theorem semiconcavity} that, for uniformly bounded functions $\psi$, the semi-concavity constant of $\varphi$ may depend on $\psi$ only through a lower bound on $\partial_t k$. We then start by proving that, for reasonable choices of $\psi$, one can obtain a lower bound on $\partial_t k$ independent of $\psi$. We shall consider as reasonable choices of $\psi$ those belonging to the class $\Psi_\delta$ defined for $\delta > 0$ by
\begin{align*}
\Psi_\delta = \{\psi: \mathbb R^d \to [0, 1] \mid {} & \exists \alpha \in C^{1, 1}(\mathbb R, [0, 1]) \text{ such that } \alpha \text{ is non-increasing, } \\
& \alpha(x) = 0 \text{ for } x \geq 0,\; \alpha^\prime(0) = 0,\; \alpha(x) = 1 \text{ for } x \leq -\delta, \\
& \text{and } \psi(x) = \alpha(d^{\pm}(x))\}.
\end{align*}

\begin{proposition}
\label{PropUnifLowerBoundPartialTK}
Let $V \in C^{1, 1}(\mathbb R^+, (0, +\infty))$ be Lipschitz continuous and non-increasing, $\chi \in C^{1, 1}(\mathbb R^d,\allowbreak \mathbb R^+)$ be Lipschitz continuous, and $g: \partial\Omega \to \mathbb R^+$ satisfy \eqref{H2} and \eqref{Smoothness of the boundary cost}. Suppose also that \eqref{Smoothness of the boundary} holds. Then there exist $C, \delta > 0$ such that, for every $\psi \in \Psi_\delta$, if $k$ is given by \eqref{IntroK} and $\eta$ is a  MFG equilibrium, defining $\tilde k$ by $\tilde k(t, x) = k(\rho^\eta(t), x)$, one has
\[
\partial_t \tilde k(t, x) \geq -C, \qquad \forall (t, x) \in \mathbb R^+ \times \Omega.
\]
\end{proposition}

\begin{proof}
Notice first that, for every $\delta > 0$ small enough, $d^{\pm}$ is $C^{1, 1}$ in a closed $\delta$-neighborhood of $\partial\Omega$, and thus one has $\psi \in C^{1, 1}(\mathbb R^d, \mathbb R^+)$ for every $\psi \in \Psi_\delta$. Then, by Proposition \ref{PropKC11MFGEquil}, $k$ is $C^{1, 1}$ on MFG equilibria.

Let $M > 0$ be such that $\sup_{x, y \in \Omega} \chi(x - y) \leq M$ and $\sup_{x, y \in \Omega} \abs{\nabla\chi(x - y)} \leq M$. Let $\underline V^\prime = -\inf_{x \in [0, M]} V^\prime(x) \geq 0$ and $\overline V = \sup_{x \in [0, M]} V(x) > 0$. Let $c > 0$ and $\delta > 0$ be as in the statement of Proposition \ref{PropLowerBoundScalarProduct}. Notice that, for every $\delta > 0$, $\psi \in \Psi_\delta$, $x \in \Omega$, and $\mu \in \mathcal P(\Omega)$, one has $\int_\Omega \chi(x-y) \psi(y) \diff \mu(y) \leq M$, and then $k(\mu, x) \leq \overline V$.

Let $\psi \in \Psi_\delta$, $k$ be given by \eqref{IntroK}, $\eta$ be a MFG equilibrium, and $\tilde k$ be defined from $k$ as in the statement. Let $\theta: \mathbb R^+ \times \mathbb R^d \to \mathbb R^+$ be given by
\[
\theta(t, x) = \int_{\Omega} \chi(x - y) \psi(y) \diff \rho_t(y).
\]
Notice that $\tilde k(t, x) = V(\theta(t, x))$ and $\theta(t, x) \in [0, M]$ for every $(t, x) \in \mathbb R^+ \times \Omega$. Since $V$ is non-increasing and $V^\prime(\theta(t, x)) \geq -\underline V^\prime$ for every $(t, x) \in \mathbb R^+ \times \Omega$, the proposition is proved if one obtains an upper bound on $\partial_t \theta(t, x)$.

Let $\alpha \in C^{1, 1}(\mathbb R, [0, 1])$ be a non-increasing function with $\alpha(x) = 0$ for $x \geq 0$, $\alpha(x) = 1$ for $x \leq -\delta$, $\alpha^\prime(0) = 0$, and $\psi(x) = \alpha(d^{\pm}(x))$ for $x \in \mathbb R^d$. As in the proof of Proposition \ref{PropKC11MFGEquil}, $\theta$ is $C^{1, 1}$ and
\begin{equation}
\label{PartialTKappa}
\partial_t \theta(t, x) = \int_{\Gamma^\prime} \Bigl[-\nabla\chi(x - \gamma(t)) \cdot {\gamma}^\prime(t) \psi(\gamma(t)) + \chi(x - \gamma(t)) \nabla\psi(\gamma(t)) \cdot {\gamma}^\prime(t)\Bigr] \diff \eta(\gamma),
\end{equation}
where $\Gamma^\prime = \bigcup_{x \in \Omega} \Gamma^\prime[\rho^\eta, x] \subset \Gamma_{k_{\max}}$. For every $\gamma \in \Gamma^\prime$, one has $\abs{\gamma^\prime(t)} \leq \overline V$. On the other hand, denoting by $u$ the optimal control associated with $\gamma$, one has
\begin{equation}
\label{TermeAEstimer}
\nabla\psi(\gamma(t)) \cdot {\gamma}^\prime(t) = k(t, \gamma(t)) \nabla\psi(\gamma(t)) \cdot u(t) = k(t, \gamma(t)) \alpha^\prime(d^{\pm}(\gamma(t))) \nabla d^{\pm}(\gamma(t)) \cdot u(t).
\end{equation}
If $\mathbf d(\gamma(t), \partial\Omega) > \delta$, then $\alpha^\prime(d^{\pm}(\gamma(t))) = 0$ and thus $\nabla\psi(\gamma(t)) \cdot {\gamma}^\prime(t) = 0$. Otherwise, by Proposition \ref{PropLowerBoundScalarProduct}, one has $\nabla d^{\pm}(\gamma(t)) \cdot u(t) \geq c$, and, since $\alpha^\prime(x) \leq 0$ for every $x \in \mathbb R$, one has $\nabla\psi(\gamma(t)) \cdot {\gamma}^\prime(t) \leq 0$. It then follows from \eqref{PartialTKappa} and \eqref{TermeAEstimer} that
\[
\partial_t \theta(t, x) \leq M \overline V,
\]
providing the required upper bound.
\end{proof}

Our main result of this section is the following.

\begin{theorem} \label{L^p MFG}
Let $p \in (1, +\infty]$, $k: \mathcal P(\Omega) \times \Omega \to \mathbb R^+$ be continuous, $g: \partial\Omega \to \mathbb R^+$, and assume that \eqref{lower bound on the dynamic}, \eqref{H2}, \eqref{Smoothness of the boundary}, \eqref{boundary cost semiconcave}, and \eqref{HypoKLip-MFG} hold. Suppose that $k$ is $C^{1, 1}$ on MFG equilibria. Let $\rho_0 \in \mathcal P(\Omega)$, $\eta \in \mathcal P_{\rho_0}(\Omega)$ be a MFG equilibrium for $\rho_0$, $\rho = \rho^\eta$, and $\varphi$ be the value function of the optimal control problem \eqref{Minimal exit time} with dynamic $(t, x) \mapsto k(\rho_t, x)$. There exists $C > 0$ such that, if $\rho_0$ is absolutely continuous and $\rho_0 \in L^p(\accentset\circ\Omega)$, then, for every $t \geq 0$, $\rho_t\rvert_{\accentset\circ\Omega}$ is absolutely continuous, $\rho_t \in L^p(\accentset\circ\Omega)$, and
\begin{equation}
\label{EqLpBound}
\norm{\rho_t}_{L^p(\accentset\circ\Omega)} \leq C \norm{\rho_0}_{L^p(\accentset\circ\Omega)}.
\end{equation}
Moreover, $C$ depends only on $\lambda$, $k_{\min}$, $k_{\max}$, $\diam(\Omega)$, a bound $\kappa$ on the curvature of $\partial\Omega$, $L_1$, $L_2$, $\ell$, the semi-concavity constant of $g$, and the semi-concavity constant w.r.t.\ $x$ of the value function $\varphi$.
\end{theorem}

Before proving Theorem \ref{L^p MFG}, we need the following auxiliary results.

\begin{lemma}
\label{LemmBoundDiverg}
Let $O \subset \mathbb R^d$ be a bounded open set, $\alpha: O \to \mathbb R$ be a semi-concave function with semi-concavity constant $C \geq 0$, and $\beta: \mathbb R^d \to \mathbb R$ be a $C^2$ convex function with $\nabla\beta$ Lipschitz continuous and bounded by some constant $C^\prime \geq 0$. Then $\nabla \beta \circ \nabla\alpha$ is a function of locally bounded variation and $\nabla \cdot (\nabla\beta \circ \nabla\alpha) \leq C C^\prime$ in the sense of distributions.
\end{lemma}

\begin{proof}
Since $\alpha$ is semi-concave with semi-concavity constant $C$, $\nabla\alpha: O \to \mathbb R^d$ is a function of locally bounded variation and $\nabla^2\alpha \leq C$ in the sense of measures (see, e.g., \cite[Proposition 1.1.3 and Theorem 2.3.1]{CanSin}). Then, by \cite[Theorem 3.96]{Ambrosio2000Functions}, $\nabla \beta \circ \nabla\alpha$ is a function of locally bounded variation, with
\[
\nabla(\nabla \beta \circ \nabla \alpha) = \xi \nabla^2 \alpha
\]
and $\xi: O \to \mathcal M_d(\mathbb R)$ given by
\[
\xi(x) = \int_0^1 \nabla\beta(t \nabla\alpha^+(x) + (1 - t) \nabla\alpha^-(x)) \diff t,
\]
where $\nabla\alpha^+$ and $\nabla\alpha^-$ have their usual definitions at jump points (see, e.g., \cite[Section 3.6]{Ambrosio2000Functions}) and are defined at points $x \in O$ where $\nabla\alpha$ is approximately continuous by setting $\nabla\alpha^+(x) = \nabla\alpha^-(x) = \nabla\alpha(x)$. In particular, since $\beta$ is convex, $\nabla\beta(y)$ is a positive semidefinite matrix for every $y \in \mathbb R^d$, and then $\xi(x)$ is also positive semidefinite for every $x \in O$ and bounded by $C^\prime$. Then $\xi(x)$ admits a positive semidefinite square root $\sqrt{\xi(x)}$, bounded by $\sqrt{C^\prime}$, and one has, in the sense of distributions,
\[
\nabla \cdot (\nabla \beta \circ \nabla \alpha) = \trace(\xi \nabla^2 \alpha) = \trace(\sqrt{\xi} \nabla^2 \alpha \sqrt{\xi}) \leq C C^\prime,
\]
as required.
\end{proof}

\begin{lemma}
\label{LemmBoundBetaEpsilon}
Let $\beta \in C^\infty(\mathbb R^d, \mathbb R^+)$ be such that $\spt(\beta) \subset B(0, 1)$, $\inf_{x \in B(0, 1/2)} \beta(x) > 0$, and $\int_{\mathbb R^d}\beta(x) \diff x = 1$. For $\varepsilon > 0$, let $\beta_\varepsilon \in C^\infty(\mathbb R^d, \mathbb R^+)$ be defined by $\beta_\varepsilon(x) = \varepsilon^{-d} \beta(\frac{x}{\varepsilon})$. Then
\[
\inf_{\substack{x \in \Omega \\ \varepsilon \in (0, 1]}} \int_{B(x, \varepsilon) \cap \Omega} \beta_{\varepsilon}(x - y) \diff y > 0.
\]
\end{lemma}

The proof of Lemma \ref{LemmBoundBetaEpsilon} follows from straightforward arguments and its details are omitted here. By taking minimizing sequences $(x_n)_{n \in \mathbb N}$ and $(\varepsilon_n)_{n \in \mathbb N}$, one may split the proof according to whether, up to extracting subsequences, $(x_n)_{n \in \mathbb N}$ converges to a point in the interior or the boundary of $\Omega$, the proof being easy in the first case and relying on the regularity of $\partial\Omega$ stated in \eqref{Smoothness of the boundary} in the second case.

\begin{proof}[Proof of Theorem \ref{L^p MFG}]
Let $T = \frac{1 + \lambda k_{\max}}{1 - \lambda k_{\max}} k_{\min}^{-1} \sup_{x \in \Omega} \mathbf d(x, \partial\Omega)$. It follows from Proposition \ref{PropBoundTau} that $\rho_t|_{\accentset\circ\Omega} = 0$ for $t \geq T$, and thus it suffices to prove \eqref{EqLpBound} for $t \in [0, T]$.

For $t \in [0, T]$, define the vector field $v_t: \mathbb R^d \to \mathbb R^d$ by
\[
v_t(x) = 
\begin{dcases}
-k(\rho_t,x)\frac{\nabla \varphi(t,x)}{\abs{\nabla \varphi(t,x)}}, & \text{if } x \in \accentset\circ{\Omega},\\
0, & \text{otherwise}.
\end{dcases}
\]
Notice that $v_t$ is well-defined almost everywhere since $x \mapsto \varphi(t, x)$ is Lipschitz continuous and, by Proposition \ref{PropLowerBoundOnGrad}, $\nabla\varphi(t, x) \neq 0$ wherever it exists. Let $c > 0$ be the constant from Corollary \ref{CoroGradNotZero} and let $F: \mathbb R^d \to \mathbb R$ be a convex $C^2$ function such that $F(x) = \abs{x}$ for every $x \in \mathbb R^d$ with $\abs{x} \geq c$. Notice that $F$ can be chosen in such a way that $\nabla F$ is bounded by some constant $c^\prime$ depending only on $c$. It follows from Proposition \ref{PropLowerBoundOnGrad} that, for almost every $x \in \accentset\circ\Omega$, one has $v_t(x) = - k(\rho_t, x) \nabla F(\nabla \varphi(t, x))$. Since $x \mapsto k(\rho_t, x)$ is Lipschitz continuous, it follows from \cite[Proposition 3.2(b)]{Ambrosio2000Functions} that $v_t$ is of locally bounded variation and that its divergence satisfies, in the sense of distributions,
\[
\nabla \cdot v_t = - \nabla k \cdot (\nabla F \circ \nabla\varphi) - k \nabla \cdot (\nabla F \circ \nabla\varphi).
\]
It then follows from \eqref{lower bound on the dynamic}, \eqref{HypoKLip-MFG}, and Lemma \ref{LemmBoundDiverg} that there exists $C > 0$ depending on $c$, the constant $k_{\max}$ from \eqref{lower bound on the dynamic}, the constant $L_1$ from \eqref{HypoKLip-MFG}, and the semi-concavity constant of $\varphi$ such that
\begin{equation}
\label{EqBoundDivergVt}
\nabla \cdot v_t \geq - C
\end{equation}
in the sense of distributions.

For $\varepsilon > 0$, let $\beta, \beta_\varepsilon \in C^\infty(\mathbb R^d, \mathbb R^+)$ be defined as in the statement of Lemma \ref{LemmBoundBetaEpsilon}, so that $\spt(\beta_\varepsilon) \subset B(0, \varepsilon)$ and $\int_{\mathbb R^d} \beta_{\varepsilon}(x) \diff x = 1$. Let $\Omega_\varepsilon = \{x \in \Omega \suchthat \mathbf d(x, \partial\Omega) > \varepsilon\}$. For $t \in [0, T]$, define $v_t^\varepsilon: \mathbb R^d \to \mathbb R^d$ by $v_t^\varepsilon = v_t \ast \beta_\varepsilon$. It then follows from \eqref{EqBoundDivergVt} and \cite[Proposition 3.2(c)]{Ambrosio2000Functions} that
\begin{equation}
\label{EqBoundDivergVtEpsilon}
\nabla \cdot v_t^\varepsilon(x) \geq -C \qquad \forall (t, x) \in [0, T] \times \Omega_\varepsilon.
\end{equation}
Notice also that, for every $q_t, q_x \in [1, +\infty)$, one has $v^\varepsilon \to v$ in $L^{q_t}([0, T], L^{q_x}(\mathbb R^d))$ as $\varepsilon \to 0$. Let $d^{\pm}: \mathbb R^d \to \mathbb R$ be the signed distance to $\partial\Omega$ defined in \eqref{EqDefiDPM}.

\begin{claim}
\label{ClaimBoundVarepsilon}
There exists $\bar c > 0$ and $\bar\varepsilon > 0$ such that, for every $\varepsilon \in (0, \bar\varepsilon]$, $t \in [0, T]$, and $x \in \Omega$ with $\mathbf d(x, \partial\Omega) \leq \varepsilon$, one has
\begin{equation}
\label{BoundScalarProductVarepsilon}
\nabla d^{\pm}(x) \cdot v_t^\varepsilon(x) \geq \bar c.
\end{equation}
\end{claim}

\begin{proof}
Let $c > 0$ and $\delta > 0$ be as in the statement of Proposition \ref{PropLowerBoundScalarProduct}. Up to reducing $\delta > 0$, $d^{\pm}$ is $C^{1, 1}$ on the set of all points at a distance at most $\delta > 0$ from $\partial\Omega$. Let $L_d$ be a Lipschitz constant for $\nabla d^{\pm}$ on this set and define
\[
c^\prime = \inf_{\substack{x \in \Omega \\ \varepsilon \in (0, 1]}} \int_{B(x, \varepsilon) \cap \Omega} \beta_\varepsilon(x - y) \diff y,
\]
which is positive by Lemma \ref{LemmBoundBetaEpsilon}. By \eqref{LowerBoundScalarProduct-Varphi}, one deduces that, for every $t \in [0, T]$, one has $\nabla d^{\pm}(x) \cdot v_t(x) \geq c k_{\min}$ for almost every $x \in \Omega$ with $\mathbf d(x, \partial\Omega) \leq \delta$. Let $\bar\varepsilon = \min\left\{\delta/2, 1, \frac{c c^\prime k_{\min}}{2 L_d k_{\max}}\right\}$ and fix $\varepsilon \in (0, \bar\varepsilon]$, $t \in [0, T]$, and $x \in \Omega$ with $\mathbf d(x, \partial\Omega) \leq \varepsilon$. Then
\begin{align*}
\nabla d^{\pm}(x) \cdot v_t^\varepsilon(x) & = \nabla d^{\pm}(x) \cdot \int_{B(x, \varepsilon) \cap \Omega} v_t(y) \beta_\varepsilon(x - y) \diff y \\
& = \int_{B(x, \varepsilon) \cap \Omega} \nabla d^{\pm}(y) \cdot v_t(y) \beta_\varepsilon(x - y) \diff y \\
& \hphantom{{} = {}} {} + \int_{B(x, \varepsilon) \cap \Omega} \left[\nabla d^{\pm}(x) - \nabla d^{\pm}(y)\right] \cdot v_t(y) \beta_\varepsilon(x - y) \diff y \\
& \geq c c^\prime k_{\min} - L_d k_{\max} \varepsilon \geq \frac{1}{2} c c^\prime k_{\min}.
\end{align*}
Hence \eqref{BoundScalarProductVarepsilon} holds with $\bar c = \frac{1}{2} c c^\prime k_{\min}$.
\end{proof}

Let $X_\varepsilon: \mathbb R^+ \times \mathbb R^d \to \mathbb R^d$ satisfy
\begin{equation}
\label{DiffEqnXVarepsilon}
\left\{
\begin{aligned}
\partial_t X_\varepsilon(t, x) & = v_t^\varepsilon(X_\varepsilon(t, x)), & \qquad & (t, x) \in [0, T] \times \mathbb R^d, \\
X_\varepsilon(0, x) & = x, & & x \in \mathbb R^d,
\end{aligned}
\right.
\end{equation}
i.e., $X_\varepsilon$ is the flow of the differential equation $\gamma^\prime = v_t^\varepsilon(\gamma)$ restricted to the fixed initial time $0$. By standard properties of flows, for every $t \in [0, T]$, the map $X_\varepsilon(t, \cdot): \mathbb R^d \to \mathbb R^d$ is invertible, and, with a slight abuse of notation, we denote its inverse by $X_\varepsilon^{-1}(t, \cdot)$. Since $v^\varepsilon_t \in C^\infty(\mathbb R^d, \mathbb R^d)$, one has $X_\varepsilon(t, \cdot) \in C^\infty(\mathbb R^d, \mathbb R^d)$ and, in particular, 
\[
\left\{
\begin{aligned}
\partial_t \nabla X_\varepsilon(t, x) & = \nabla v_t^\varepsilon(X_\varepsilon(t, x)) \nabla X_\varepsilon(t, x), & \qquad & (t, x) \in [0, T] \times \mathbb R^d, \\
\nabla X_\varepsilon(0, x) & = I, & & x \in \mathbb R^d.
\end{aligned}
\right.
\]
Let $J_\varepsilon: \mathbb R^+ \times \mathbb R^d \to \mathbb R$ be given by $J_\varepsilon(t, x) = \det(\nabla X_\varepsilon(t, x))$. Then $J_\varepsilon$ satisfies
\[
\left\{
\begin{aligned}
\partial_t J_\varepsilon(t, x) & = \nabla \cdot v_t^\varepsilon(X_\varepsilon(t, x)) J_\varepsilon(t, x), & \qquad & (t, x) \in [0, T] \times \mathbb R^d, \\
J_\varepsilon(0, x) & = 1, & & x \in \mathbb R^d,
\end{aligned}
\right.
\]
which yields
\[
J_\varepsilon(t, x) = \exp\left(\int_0^t \nabla \cdot v^\varepsilon_s(X_\varepsilon(s, x)) \diff s\right).
\]

Let $\rho_t^\varepsilon = X_\varepsilon(t, \cdot)_{\#} \rho_0$. Then $\rho^\varepsilon$ satisfies, in the sense of distributions in $[0, T) \times \mathbb R^d$,
\begin{equation}
\label{RhoVarepsilonContinuityEquation}
\left\{
\begin{aligned}
& \partial_t \rho^\varepsilon + \nabla \cdot (\rho^\varepsilon v^\varepsilon) = 0, \\
& \rho^\varepsilon_0 = \rho_0.
\end{aligned}
\right.
\end{equation}
Moreover, since $\rho_0$ is absolutely continuous with respect to the Lebesgue measure, so is $\rho^\varepsilon_t$, and their densities (also denoted by $\rho_0$ and $\rho^\varepsilon_t$ for simplicity) satisfy
\[
\rho^\varepsilon_t(x) = \frac{\rho_0(X_\varepsilon^{-1}(t, x))}{J_\varepsilon(t, X_\varepsilon^{-1}(t, x))}.
\]

Let $K \subset \accentset\circ\Omega$ be compact. For $t \in [0, T]$ and $\varepsilon > 0$, set
\[
K_{t}^\varepsilon = \left\{X_\varepsilon(s, x) \suchthat s \in [0, t],\; x \in \mathbb R^d \text{ and } X_\varepsilon(t, x) \in K\right\},
\]
i.e., $K_t^\varepsilon$ is the set of all points which belong to some trajectory of \eqref{DiffEqnXVarepsilon} passing through $K$ at time $t$.

\begin{claim}
\label{ClaimKSubsetOmega}
There exists $\varepsilon_0 > 0$ such that, for every $t \in [0, T]$ and $\varepsilon \in (0, \varepsilon_0)$, one has $K_t^\varepsilon \subset \Omega_\varepsilon$.
\end{claim}

\begin{proof}
Assume, to obtain a contradiction, that there exist sequences $(t_n)_{n \in \mathbb N}$ and $(\varepsilon_n)_{n \in \mathbb N}$ with $\varepsilon_n \to 0$ as $n \to \infty$ such that $t_n \in [0, T]$ and $K_{t_n}^{\varepsilon_n} \not\subset \Omega_{\varepsilon_n}$ for every $n \in \mathbb N$. Then, for every $n \in \mathbb N$, there exists $s_n \in [0, t_n]$ and $x_n \in \mathbb R^d$ such that, setting $z_n = X_{\varepsilon_n}(s_n, x_n)$, one has $z_n \notin \Omega_{\varepsilon_n}$. Notice that $\mathbf d(x_n, \Omega) \leq \varepsilon_n$, for otherwise one would have $v_s^{\varepsilon_n}(x_n) = 0$ for every $s \geq 0$ and then $X_{\varepsilon_n}(s, x_n) = x_n$ for every $s \geq 0$, contradicting the fact that $X_{\varepsilon_n}(t_n, x_n) \in K \subset \accentset\circ\Omega$. For the same reason, one must have $\mathbf d(z_n, \Omega) \leq \varepsilon_n$, and, since $z_n \notin \Omega_{\varepsilon_n}$, this implies that $\mathbf d(z_n, \partial\Omega) \leq \varepsilon_n$.

Let $\bar\varepsilon > 0$ and $\bar c > 0$ be such that \eqref{BoundScalarProductVarepsilon} holds for every $\varepsilon \in (0, \bar\varepsilon]$, $t \in [0, T]$, and $x \in \Omega$ with $\mathbf d(x, \partial\Omega) \leq \varepsilon$. Up to reducing $\bar\varepsilon$, one may assume that $\mathbf d(K, \partial\Omega) > \bar\varepsilon$.

Fix $n \in \mathbb N$ such that $\varepsilon_n \leq \bar\varepsilon$. Let $\alpha: [0, T] \to \mathbb R$ be defined for $s \in [0, T]$ by $\alpha(s) = d^{\pm}(X_{\varepsilon_{n}}(s, x_{n}))$. Then $\alpha^\prime(s) = \nabla d^{\pm}(X_{\varepsilon_{n}}(s, x_{n})) \cdot v_s^{\varepsilon_{n}}(X_{\varepsilon_{n}}(s, x_{n}))$. In particular, by \eqref{BoundScalarProductVarepsilon}, $\alpha^\prime(s) \geq \bar c > 0$ whenever $\alpha(s) \in [-\varepsilon_n, 0]$ (i.e., whenever $X_{\varepsilon_n}(s, x_n) \in \Omega$ and $\mathbf d(X_{\varepsilon_n}(s, x_n), \partial\Omega) \leq \varepsilon_n$). Since $\mathbf d(z_n, \partial\Omega) \leq \varepsilon_n$, one has $\alpha(s_{n}) = d^{\pm}(z_{n}) \in [-\varepsilon_n, \varepsilon_n]$, and thus $\alpha(s) \geq -\varepsilon_n$ for every $s \in [s_n, T]$. This is a contradiction, since $\alpha(t_n) = d^{\pm}(X_{\varepsilon_n}(t_n, x_n)) < -\varepsilon_n$ due to the fact that $X_{\varepsilon_n}(t_n, x_n) \in K \subset \accentset\circ\Omega$ and $\mathbf d(K, \partial\Omega) > \bar\varepsilon \geq \varepsilon_n$.
\end{proof}

Let $\varepsilon_0 > 0$ be as in the statement of Claim \ref{ClaimKSubsetOmega}. We consider here only the case $p \in (1, \infty)$, the remaining case $p = \infty$ following from the fact that our constants do not depend on $p$. For $t \in [0, T]$ and $\varepsilon \in (0, \varepsilon_0)$, one has
\begin{align*}
\norm{\rho_t^\varepsilon}_{L^p(K)}^p & = \int_K \frac{\rho_0(X_\varepsilon^{-1}(t, x))^p}{J_\varepsilon(t, X_\varepsilon^{-1}(t, x))^p} \diff x = \int_{K_0} \frac{\rho_0(x)^p}{J_\varepsilon(t, x)^{p-1}} \diff x \\
& = \int_{K_0} \rho_0(x)^p \left[\exp\left(\int_0^t \nabla \cdot v_s^\varepsilon(X_\varepsilon(s, x)) \diff s\right)\right]^{1 - p} \diff x,
\end{align*}
where $K_0 = \{x \in \mathbb R^d \suchthat X_\varepsilon(t, x) \in K\}$. For every $x \in K_0$ and $s \in [0, t]$, one has $X_\varepsilon(s, x) \in K_t^\varepsilon \subset \Omega_\varepsilon$, and thus one obtains from the above expression and \eqref{EqBoundDivergVtEpsilon} that
\begin{equation}
\label{EqLpBoundVarepsilon}
\norm{\rho_t^\varepsilon}_{L^p(K)}^p \leq e^{C (p-1) T} \norm{\rho_0}_{L^p(\accentset\circ\Omega)}^p.
\end{equation}

Let $(K_n)_{n \in \mathbb N}$ be an increasing sequence of compact subsets of $\accentset\circ\Omega$ such that $\accentset\circ\Omega = \bigcup_{n \in \mathbb N} K_n$. For $i \in \mathbb N$, we construct by induction on $i$ a sequence $(\varepsilon^{i}_{n})_{n \in \mathbb N}$ such that $\varepsilon^i_n \to 0$ as $n \to \infty$. Let $K = K_0$ and take $\varepsilon_0 > 0$ as in the statement of Claim \ref{ClaimKSubsetOmega}. Since, by \eqref{EqLpBoundVarepsilon}, $(\rho^\varepsilon)_{\varepsilon \in (0, \varepsilon_0]}$ is bounded in $L^\infty([0, T], L^p(K_0))$, there exists a sequence $(\varepsilon^0_n)_{n \in \mathbb N}$ in $(0, \varepsilon_0]$ with $\varepsilon^0_n \to 0$ as $n \to \infty$ such that $(\rho^{\varepsilon^0_n})_{n \in \mathbb N}$ converges weakly-$\ast$ in $L^\infty([0, T], L^p(K_0))$. Now, assume that $i \in \mathbb N$ is such that $(\varepsilon^i_n)_{n \in \mathbb N}$ is constructed and $\varepsilon^i_n \to 0$ as $n \to \infty$. Since, by \eqref{EqLpBoundVarepsilon}, $(\rho^{\varepsilon^i_n})_{n \in \mathbb N}$ is bounded in $L^\infty([0, T], L^p(K_{i+1}))$, there exists a subsequence $(\varepsilon^{i+1}_n)_{n \in \mathbb N}$ of $(\varepsilon^i_n)_{n \in \mathbb N}$ such that $(\rho^{\varepsilon^{i+1}_n})_{n \in \mathbb N}$ converges weakly-$\ast$ in $L^\infty([0, T], L^p(K_{i+1}))$.

For $n \in \mathbb N$, let $\varepsilon_n = \varepsilon^n_n$. Then $(\rho^{\varepsilon_n})_{n \in \mathbb N}$ converges weakly-$\ast$ in $L^\infty([0, T], L^p(K_i))$ for every $i \in \mathbb N$. Let $\bar\rho \in L^\infty([0, T], L^p_{\text{loc}}(\accentset\circ\Omega))$ denote the weak-$\ast$ limit of $(\rho^{\varepsilon_n})_{n \in \mathbb N}$. One deduces from the weak convergence of $(\rho^{\varepsilon_n})_{n \in \mathbb N}$ and \eqref{EqLpBoundVarepsilon} that, for every $i \in \mathbb N$ and almost every $t \in [0, T]$,
\begin{equation*}
\norm{\bar\rho_t}_{L^p(K_i)}^p \leq e^{C (p-1) T} \norm{\rho_0}_{L^p(\accentset\circ\Omega)}^p,
\end{equation*}
and thus
\[
\norm{\bar\rho_t}_{L^p(\accentset\circ\Omega)} = \lim_{i \to \infty} \norm{\bar\rho_t}_{L^p(K_i)} \leq e^{C \left(1-\frac{1}{p}\right) T} \norm{\rho_0}_{L^p(\accentset\circ\Omega)}
\]
for almost every $t \in [0, T]$. In particular, one obtains that $\bar\rho \in L^\infty([0, T], L^p(\accentset\circ\Omega))$.

Since $v^\varepsilon \to v$ in $L^1([0, T], L^{p^\prime}(\mathbb R^d))$ as $\varepsilon \to 0$, one obtains from \eqref{RhoVarepsilonContinuityEquation} that $\bar\rho$ satisfies, in the sense of distributions in $[0, T) \times \accentset\circ\Omega$,
\[
\left\{
\begin{aligned}
& \partial_t \bar\rho + \nabla \cdot (\bar\rho v) = 0, \\
& \bar\rho_0 = \rho_0.
\end{aligned}
\right.
\]
On the other hand, the measure $\rho = \rho^\eta$ obtained from the MFG equilibrium $\eta$ also satisfies the continuity equation $\partial_t \rho + \nabla \cdot (\rho v) = 0$ with initial condition $\rho_0$. It follows from Proposition \ref{PropUnique} and \cite[Theorem 3.1]{Ambrosio2008Transport} that solutions to this equation are unique, and thus $\bar\rho = \rho$. In particular,
\begin{equation}
\label{LpBoundAlmostThere}
\norm{\rho_t}_{L^p(\accentset\circ\Omega)}^p \leq e^{C (p-1) T} \norm{\rho_0}_{L^p(\accentset\circ\Omega)}^p
\end{equation}
for almost every $t \in [0, T]$. To conclude that \eqref{LpBoundAlmostThere} holds for every $t \in [0, T]$, let $\bar t \in [0, T]$ and $(t_n)_{n \in \mathbb N}$ be a sequence in $[0, T]$ such that $t_n \to t$ as $n \to \infty$ and \eqref{LpBoundAlmostThere} holds at $t_n$ for every $n \in \mathbb N$. The sequence $(\rho_{t_n})_{n \in \mathbb N}$ is bounded in $L^p(\accentset\circ\Omega)$, and thus, up to the extraction of a subsequence, it admits a weak limit $\widetilde\rho$. On the other hand, $t \mapsto \rho_t = (e_t)_{\#} \eta$ is continuous with respect to the weak convergence of measures, and thus $\widetilde\rho = \rho_{\bar t}$. One concludes that \eqref{LpBoundAlmostThere} holds for $\bar t$ by the $L^p$-weak convergence of $(\rho_{t_n})_{n \in \mathbb N}$ to $\rho_{\bar t}$ and the weak lower semi-continuity of the $L^p$ norm.
\end{proof}
 
\subsection{Equilibria in a less regular model}
\label{SecExistenceLessRegular}

In this section, we use the $L^p$ estimates on $\rho_t$ from Theorem \ref{L^p MFG} to study equilibria of the MFG model with $k$ given by \eqref{IntroK} and $\psi = \mathbbm 1_{\accentset\circ\Omega}$, i.e.,
\begin{equation}
\label{LimitDynamic}
k(\mu,x)=V\biggl(\int_\Omega \chi(x-y) \mathbbm 1_{\accentset\circ{\Omega}}(y) \diff \mu(y)\biggr), \quad \text{for all }(\mu,x) \in \mathcal{P}(\Omega) \times \Omega.
\end{equation}
Notice that the lack of continuity of the dynamic $k$ with respect to $\mu$ prevents us from using the result of Section \ref{SecExistenceMFG}. So, the idea is to consider a sequence of \emph{cut-off} functions $(\psi^\varepsilon)_{\varepsilon >0}$ taken in $\Psi_\delta$, for $\delta$ as in Proposition \ref{PropUnifLowerBoundPartialTK}, and converging as $\varepsilon \to 0$ to $\mathbbm 1_{\accentset\circ{\Omega}}$ in $L^q(\mathbb R^d)$ for all $q \in [1, +\infty)$, and to replace the dynamic $k$ with $k_\varepsilon$ defined from $\psi^\varepsilon$ as in \eqref{IntroK}, i.e.,
\begin{equation}
\label{SequenceDynamic}
k_\varepsilon(\mu,x) = V\biggl(\int_\Omega \chi(x-y) \psi^\varepsilon(y) \diff \mu(y)\biggr), \quad \text{for all } (\mu,x) \in \mathcal{P}(\Omega) \times \Omega.
\end{equation}
Our first result of this section shows that, under some suitable convergence assumptions on $k_\varepsilon$ as $\varepsilon \to 0$, one has uniform convergence of the value functions of the corresponding optimal control problems and that the limit of MFG equilibria is a MFG equilibrium for the limiting model.

\begin{proposition} \label{Uniform convergence of the value function}
Let $\rho_0 \in \mathcal P(\Omega)$. For $n \in \mathbb N$, let $k_n, k :\mathcal{P}(\Omega) \times \Omega \to \mathbb{R}^+$ be such that $k_n$ is continuous on $\mathcal{P}(\Omega) \times \Omega$ and Lipschitz continuous with respect to the second variable. Let $\eta_n$ be a MFG equilibrium for $\rho_0$ associated with the control problem with dynamic $k_n$. In addition, assume the following:
\begin{itemize}
\item As $n \to \infty$, $(\eta_n)_{n \in \mathbb N}$ converges weakly in $\mathcal P(\Gamma)$ to some measure $\eta$.
\item There exist two constants $k_{\min}$ and $k_{\max}$ such that $0 < k_{\min} \leq k_n \leq k_{\max} < + \infty$.
\item There exists a constant $M$ independent of $n$ such that $\abs{\nabla k_n} \leq M$.
\item For every $(t, x) \in \mathbb{R}^+ \times \Omega$, we have $k_n((e_t)_{\#} \eta_n, x) \to k((e_t)_{\#} \eta, x)$ as $n \to \infty$.
\item For every $x \in \Omega$, $t \mapsto k((e_t)_{\#}\eta, x)$ is continuous on $\mathbb{R}^+$.
\end{itemize}
For $n \in \mathbb N$, let $\varphi_n$ (resp.\ $\varphi$) be the value function associated with the control problem with dynamic $k_n$ (resp.\ $k$). Then
\begin{enumerate}
\item\label{UnifConvVarphi} $\varphi_n \rightarrow \varphi$ as $n \to \infty$ uniformly in $\mathbb{R}^+ \times \Omega$, and
\item\label{LimitIsEquilibrium} $\eta$ is a MFG equilibrium for $\rho_0$ with dynamic $k$.
\end{enumerate}
\end{proposition}

\begin{proof}
Let us first prove \ref{UnifConvVarphi}. Notice that, up to extracting a subsequence, $(\varphi_n)_{n \in \mathbb N}$ converges uniformly to some function $\widetilde{\varphi}$ on $\mathbb{R}^+ \times \Omega$. Indeed, from Proposition \ref{PropBoundTau}, the sequence $(\varphi_n)_n$ is equibounded. Moreover, by Proposition \ref{LipValueFunction}, the value function $\varphi_n$ is Lipschitz in $\mathbb{R}^+ \times \Omega$ with a Lipschitz constant depending only on the Lipschitz constant of the dynamic $k_n$ with respect to $x$, which is independent of $n$. Then, by Arzelà--Ascoli Theorem, $(\varphi_n)_{n \in \mathbb N}$ admits a uniform limit $\widetilde\varphi$ up to the extraction of a subsequence.

We now prove that the limit of $(\varphi_n)_{n \in \mathbb N}$ is $\varphi$. Fix $(t,x) \in \mathbb{R}^+ \times \Omega$. For every $n \in \mathbb N$, let $\gamma_n$ be an optimal trajectory for $x$, at time $t$, in the control problem with dynamic $k_n$. It is easy to observe that, up to extracting a subsequence, $\gamma_n \rightarrow \gamma$ uniformly for some $\gamma \in \Gamma_{k_{\max}}$. Yet, this $\gamma$ is, in fact, an admissible trajectory for $x$, at time $t$, in the control problem with dynamic $k$. Indeed, for a.e.\ $s \in (t,\infty)$, we have $\abs{\gamma_n^\prime(s)} \leq k_n((e_s)_{\#}\eta_n,\gamma_n(s))$. So, letting $n \to \infty$, we get
\begin{equation*} 
\abs{\gamma^\prime(s)} \leq k((e_s)_{\#}\eta,\gamma(s)), \quad \text{for a.e.\ } s \in (t, \infty).
\end{equation*}
Let $u_n$ be the optimal control associated with $\gamma_n$ and $u$ the control associated with $\gamma$. Set
$$\tau_n:=\tau^{t,x,u_n}, \qquad \tau_\gamma:= \tau^{t, x, u}, \qquad z_n:=\gamma_n(t+\tau_n) \in \partial\Omega.$$
It is clear that there exist $\bar\tau \geq 0$ and $z \in \partial\Omega$ such that, up to extracting subsequences, $\tau_n \to \bar{\tau}$ and $z_n \to z$ as $n \to \infty$. In particular, we have $z=\gamma(t + \bar{\tau})$ and then $\tau_\gamma \leq \bar{\tau}$. Consequently, by Lemma \ref{LemmTimePlusGIncreases}, $\varphi(t,x) \leq \bar{\tau} + g(z)=\lim_{n \to \infty} \varphi_n(t,x)=\widetilde{\varphi}(t,x)$.

To prove the converse inequality, let $\gamma$ be an optimal trajectory for $x$, at time $t$, in the control problem with dynamic $k$, and $u$ be the associated optimal control with $\gamma$. For $n \in \mathbb N$, let $\phi_n$ be a solution of
\begin{equation} \label{00}
\begin{dcases}
\phi_n^\prime(s)=\frac{k_n((e_s)_{\#}\eta_n,\gamma(\phi_n(s)))}{k((e_{\phi_n (s)})_{\#} \eta,\gamma(\phi_n(s)))},\\
\phi_n(t)=t.
\end{dcases}
\end{equation}
Set 
$$\gamma_n(s)=\gamma(\phi_n(s)), \quad \text{for all } s \in [t,\infty).$$
It is clear that $\gamma_n$ is admissible for $x$, at time $t$, in the control problem with dynamic $k_n$, and its corresponding control $u_n$ is given by $u_n(s) = u(\phi_n(s))$ for $s \geq t$. Let $\tau_n = \tau^{t,x,u_n}$. Hence, we have 
\begin{equation} \label{010}
\varphi_n(t,x) \leq \tau_n + g(\gamma_n(t+\tau_n)).
\end{equation}
Yet, we observe easily that $\tau_n=\phi_n^{-1}(t+\tau) - t$, where $\tau:=\tau^{t,x,u}$. From \eqref{00}, we have
$$\int_t^{\phi_n(s)} k((e_r)_{\#}\eta,\gamma(r)) \diff r = \int_t^s k_n((e_r)_{\#}\eta_n,\gamma(\phi_n(r)))\diff r.$$
Set 
$$\Psi(\theta)=\int_t^\theta k((e_r)_{\#} \eta,\gamma(r))\diff r, \quad \text{for all } \theta \in [t,\infty).$$
Then $\Psi$ is a bijective map from $[t, +\infty)$ to $[0, +\infty)$, whose inverse is $k_{\min}^{-1}$-Lipschitz continuous. We have
\begin{align*}
\abs{\phi_n(s) - s} & = \abs[\bigg]{\Psi^{-1} \biggl(\int_t^s k_n((e_r)_{\#}\eta_n,\gamma(\phi_n(r)))\diff r\biggr) - \Psi^{-1} \biggl(\int_t^s k((e_r)_{\#}\eta,\gamma(r))\diff r\biggr)} \displaybreak[0] \\
 & \leq k_{\min}^{-1} \int_t^s \abs{k_n((e_r)_{\#}\eta_n,\gamma(\phi_n(r))) - k((e_r)_{\#}\eta,\gamma(r)) }\diff r \displaybreak[0] \\
 & \leq k_{\min}^{-1} \int_t^s \biggl(\abs{k_n((e_r)_{\#}\eta_n,\gamma(\phi_n(r))) - k((e_r)_{\#}\eta,\gamma(\phi_n(r)))} \\
 & \hphantom{{} \leq k_{\min}^{-1} \int_t^s \biggl(} {} + \abs{k((e_r)_{\#}\eta,\gamma(\phi_n(r))) - k((e_r)_{\#}\eta,\gamma(r)) }\biggr)\diff r \displaybreak[0] \\
 & \leq k_{\min}^{-1} \biggl(\int_t^s \abs{k_n((e_r)_{\#}\eta_n,\gamma(\phi_n(r))) - k((e_r)_{\#} \eta,\gamma(\phi_n(r)))}\diff r \\
 & \hphantom{{}\leq k_{\min}^{-1} \biggl(} {} + M k_{\max} \int_t^s\abs{\phi_n(r) - r}\diff r \biggr).
\end{align*}
Yet, 
\begin{align*}
 & \abs{k_n((e_r)_{\#}\eta_n,\gamma(\phi_n(r))) - k((e_r)_{\#} \eta,\gamma(\phi_n(r)))} \\
{} \leq {} & \abs[\bigg]{\biggl(k_n((e_r)_{\#}\eta_n,\gamma(\phi_n(r))) - k_n((e_r)_{\#}\eta_n,\gamma(r))\biggr) - \biggl(k((e_r)_{\#} \eta,\gamma(\phi_n(r))) - k((e_r)_{\#} \eta,\gamma(r))\biggr) } \\
 & {} + \abs[\bigg]{k_n((e_r)_{\#}\eta_n,\gamma(r)) - k((e_r)_{\#} \eta,\gamma(r))} \displaybreak[0] \\
{} \leq {} & 2 M k_{\max} \abs{\phi_n(r) - r} + \abs[\bigg]{k_n((e_r)_{\#}\eta_n,\gamma(r)) - k((e_r)_{\#} \eta,\gamma(r))}.
\end{align*}
Hence, one has
\[\abs{\phi_n(s) - s} \leq C \biggl( \int_t^s \abs[\bigg]{k_n((e_r)_{\#}\eta_n,\gamma(r)) - k((e_r)_{\#} \eta,\gamma(r))} \diff r + \int_t^s\abs{\phi_n(r) - r}\diff r \biggr),\]
where $C > 0$ depends only on $M$, $k_{\max}$, and $k_{\min}$. Using Gronwall's inequality, we get
\[
\abs{\phi_n(s) - s} \leq C e^{C(s-t)} \int_t^s \abs[\bigg]{k_n((e_r)_{\#}\eta_n,\gamma(r)) - k((e_r)_{\#} \eta,\gamma(r))}\diff r.
\]
Consequently, for every $s \geq t$, $\phi_n(s) \to s$ as $n \to \infty$. In particular, we have $\tau_n=\phi_n^{-1}(t + \tau) - t \to \tau$. So, passing to the limit in \eqref{010}, we get
$$\widetilde{\varphi}(t,x) \leq \tau + g(\gamma(t + \tau))=\varphi(t,x).$$
This concludes the proof of \ref{UnifConvVarphi}.

To prove \ref{LimitIsEquilibrium}, we define, for $k \in \mathbb N^\ast$, the set $V_k:=\{\gamma \in \Gamma \suchthat \mathbf d(\gamma,\bigcup_{x}\Gamma^\prime[\rho^\eta,x])\leq \frac{1}{k}\}$. We claim that, for every $k \in \mathbb N^\ast$, there is some $N_0 \in \mathbb N$ such that
\begin{equation}
\label{OptimTrajSubsetGammaK}
\bigcup_{x \in \Omega} \Gamma^\prime[{\rho}^{\eta_n},x] \subset V_k \qquad \text{ for every } n \geq N_0.
\end{equation}

Indeed, if this is not the case, then there exists $k \in \mathbb N^\ast$ and sequences $(n_j)_{j \in \mathbb N}$, $(x_j)_{j \in \mathbb N}$, and $(\gamma_j)_{j \in \mathbb N}$ with $n_j \to \infty$ as $j \to \infty$ and, for every $j \in \mathbb N$, $x_j \in \Omega$ and $\gamma_j \in \Gamma^\prime[{\rho}^{\eta_{n_j}},x_{j}] \setminus V_k$. Up to extracting subsequences, there exist $x \in \Omega$ and $\gamma \in \Gamma_{k_{\max}}$ such that, as $j \to \infty$, $x_{j} \to x$ and $\gamma_{j} \to \gamma$ on $\Gamma$. For $j \in \mathbb N$, set $\tau_j = \tau^{0, x_j, u_j}$, where $u_j$ is the control corresponding to $\gamma_j$. Then, using Proposition \ref{PropBoundTau}, we infer that, up to extracting subsequences, there exists $\bar\tau \geq 0$ such that $\tau_j \to \bar{\tau}$ and $\gamma_j(\tau_j) \to \gamma(\bar{\tau}) \in \partial\Omega$, which implies that $\tau_\gamma \leq \bar{\tau}$. Moreover, it is easy to check that $\gamma$ is admissible in the control problem with dynamic $k$. Yet, we have
$$\varphi_{n_j}(0,x_j) = \tau_j + g(\gamma_j(\tau_j)).$$
Then, passing to the limit when $j \to \infty$, we obtain from \ref{UnifConvVarphi} and Lemma \ref{LemmTimePlusGIncreases} that
$$\varphi(0,x)=\bar{\tau} + g(\gamma(\bar{\tau})) \geq \tau_\gamma + g(\gamma(\tau_\gamma)).$$
This implies that $\gamma \in \Gamma^\prime[\rho^\eta,x]$, which is a contradiction.

As a consequence of \eqref{OptimTrajSubsetGammaK} and the fact that $V_k$ is a closed subset of $\Gamma$, we have, for every $k \in \mathbb N^\ast$,
$$\eta(V_k) \geq \limsup_{n \to \infty} \eta_{n} (V_k) \geq \limsup_{n \to \infty} \eta_{n} \biggl(\bigcup_{x} \Gamma^\prime[\rho^{\eta_n}, x]\biggr)=1.$$
Hence, $\eta(V_k)=1$ and, since $k$ is arbitrary and $\bigcap_{k \in \mathbb N^\ast} V_k = \bigcup_{x} \Gamma^\prime[\rho^\eta,x]$, we deduce that $\eta\biggl(\bigcup_{x} \Gamma^\prime[\rho^\eta,x]\biggr)=1$, concluding the proof of \ref{LimitIsEquilibrium}.
\end{proof}

One can now use Proposition \ref{Uniform convergence of the value function} to obtain the existence of an equilibrium to the less regular dynamic $k$ defined in \eqref{LimitDynamic}.
\begin{theorem}
\label{existence of an equilibrium in less regular case}
Let $k$ be given by \eqref{LimitDynamic} and $\rho_0 \in L^p(\accentset\circ\Omega)$ for some $p \in (1, +\infty]$. Then there exists a MFG equilibrium $\eta$ for $\rho_0$. Moreover, letting $\rho = \rho^\eta$ and $\varphi$ be the value function of the optimal control problem \eqref{Minimal exit time}, then $(\rho, \varphi)$ solves the MFG system \eqref{MFGSystem}.
\end{theorem}

\begin{proof}
For $\varepsilon > 0$, let $k_\varepsilon$ be given by \eqref{SequenceDynamic} and $\eta_\varepsilon$ be a MFG equilibrium for $\rho_0$ associated with the control problem with dynamic $k_\varepsilon$. Then there exists $\eta \in \mathcal P(\Gamma)$ and a sequence $(\varepsilon_n)_{n \in \mathbb N}$ with $\varepsilon_n \to 0$ as $n \to \infty$ such that $\eta_{\varepsilon_n} \deb \eta$ as $n \to \infty$. To prove that $\eta$ is a MFG equilibrium for $\rho_0$, it suffices to show that the hypotheses of Proposition \ref{Uniform convergence of the value function} are verified for the sequences $(k_{\varepsilon_n})_{n \in \mathbb N}$ and $(\eta_{\varepsilon_n})_{n \in \mathbb N}$. For $t \geq 0$, let $\rho_t^{\varepsilon} = (e_t)_{\#} \eta_\varepsilon$.

One easily obtains from \eqref{SequenceDynamic} and Proposition \ref{PropUnifLowerBoundPartialTK} that there exist $k_{\min}, k_{\max}, M, C > 0$ such that, for every $\varepsilon > 0$, $0 < k_{\min} \leq k_{\varepsilon} \leq k_{\max} < +\infty$, $\abs{\nabla k_\varepsilon} \leq M$, and $\partial_t k_\varepsilon \geq - C$. As a consequence of that, by Theorem \ref{Theorem semiconcavity}, the value function $\varphi_\varepsilon$, associated with the control problem with the dynamic $k_\varepsilon$, is semi-concave with respect to $x$, and its semi-concavity constant is independent of $\varepsilon$. Then, Theorem \ref{L^p MFG} implies that
$$\norm{\rho^\varepsilon_t}_{L^p(\accentset\circ\Omega)} \leq C \norm{\rho_0}_{L^p(\accentset\circ\Omega)}, \quad \text{for all } t \in \mathbb{R}^+,\;\varepsilon>0,$$
where the constant $C > 0$ is independent of $t$ and $\varepsilon$.

Since $\eta_{\varepsilon_n} \deb \eta$ in $\mathcal P(\Gamma)$ as $n \to \infty$, one deduces from the above uniform $L^p$ estimate that $\rho_t^{\varepsilon_n} \deb \rho_t$ in $L^p$. In addition, $\psi^\varepsilon \rightarrow \mathbbm 1_{\accentset\circ{\Omega}}$ in $L^q$ as $\varepsilon \to 0$, for all $q \in [1, +\infty)$. Using these facts, we get, for every $(t,x) \in \mathbb{R}^+ \times \Omega$, that, as $n \to \infty$,
\[
k_{\varepsilon_n}((e_t)_{\#}\eta_{\varepsilon_n},x) = V\biggl(\int_\Omega \chi(x-y) \psi^{\varepsilon_n}(y)\rho^{\varepsilon_n}_t(y) \diff y\biggr) \to V\biggl(\int_{\accentset\circ\Omega} \chi(x-y) \rho_t(y) \diff y\biggr) = k((e_t)_{\#}\eta,x).
\]
Moreover, for any $x \in \Omega$, the function $t \mapsto k((e_t)_{\#}\eta, x)$ is continuous on $\mathbb{R}^+$. Indeed, if $(t_n)_{n \in \mathbb N}$ is a sequence with $t_n \to t$, then $\rho_{t_n} \deb \rho_t$ in $L^p$ and so we have 
\[
k((e_{t_n})_{\#}\eta,x) = V\biggl(\int_{\accentset\circ\Omega} \chi(x-y) \rho_{t_n}(y) \diff y\biggr) \to V\biggl(\int_{\accentset\circ\Omega} \chi(x-y) \rho_t(y) \diff y\biggr) = k((e_t)_{\#}\eta,x).
\]
Hence the hypotheses of Proposition \ref{Uniform convergence of the value function} are satisfied, and then $\eta$ is a MFG equilibrium for $\rho_0$.

To obtain the MFG system \eqref{MFGSystem} for this equilibrium, notice first that $(t, x) \mapsto k(\rho_t, x)$ is continuous and satisfies \eqref{lower bound on the dynamic} and \eqref{Hy1}, and thus it follows from Proposition \ref{MainTheoHJ} that $\varphi$ satisfies the Hamilton--Jacobi equation in \eqref{MFGSystem} in the viscosity sense.

By Theorem \ref{TheoMFGSystem}, for every $\varepsilon > 0$, $\rho^\varepsilon$ satisfies the continuity equation in \eqref{MFGSystem} with dynamic $k_\varepsilon$ and the corresponding value function $\varphi_\varepsilon$. This means that, for every $\phi \in C^\infty_{\mathrm c}((0, +\infty) \times \accentset\circ\Omega)$, one has
\[
- \int_0^\infty \int_{\accentset\circ\Omega} \partial_t \phi(t, x) \rho^\varepsilon_t(x) \diff x \diff t + \int_0^{\infty} \int_{\accentset\circ\Omega} \rho^\varepsilon_t(x) k_\varepsilon(\rho^\varepsilon_t, x) \frac{\nabla \varphi_\varepsilon(t, x)}{\abs{\nabla \varphi_\varepsilon(t, x)}} \cdot \nabla \phi(t, x) \diff x \diff t = 0.
\]
Recall that, by Proposition \ref{PropBoundTau}, one has $\rho_t^\varepsilon|_{\accentset\circ\Omega} = 0$ for every $\varepsilon > 0$ and $t \geq T$, where $T = \frac{1 + \lambda k_{\max}}{1 - \lambda k_{\max}} k_{\min}^{-1} \sup_{x \in \Omega} \mathbf d(x, \partial\Omega)$. Hence,
\begin{equation}
\label{ContinuityEquation-Weak}
- \int_0^T \int_{\accentset\circ\Omega} \partial_t \phi(t, x) \rho^\varepsilon_t(x) \diff x \diff t + \int_0^T \int_{\accentset\circ\Omega} \rho^\varepsilon_t(x) k_\varepsilon(\rho^\varepsilon_t, x) \frac{\nabla \varphi_\varepsilon(t, x)}{\abs{\nabla \varphi_\varepsilon(t, x)}} \cdot \nabla \phi(t, x) \diff x \diff t = 0
\end{equation}
for every $\phi \in C^\infty_{\mathrm c}((0, T) \times \accentset\circ\Omega)$. Recall that, for every $t \in \mathbb R^+$, one has $\rho_t^{\varepsilon_n} \deb \rho_t$ in $L^p$, and thus
\begin{equation}
\label{ContinuityEquation-ConvergenceDt}
\lim_{n \to \infty} \int_{\accentset\circ\Omega} \partial_t \phi(t, x) \rho^{\varepsilon_n}_t(x) \diff x = \int_{\accentset\circ\Omega} \partial_t \phi(t, x) \rho_t(x) \diff x.
\end{equation}
Moreover, $k_{\varepsilon_n}(\rho^{\varepsilon_n}_t, x) \to k(\rho_t, x)$ for every $(t, x) \in \mathbb R^+ \times \Omega$. On the other hand, for every $t \in \mathbb R^+$ and $\varepsilon > 0$, $x \mapsto \varphi_\varepsilon(t, x)$ is semi-concave and its semi-concavity constant $C$ is independent of $\varepsilon$. Then, by Proposition \ref{Uniform convergence of the value function}\ref{UnifConvVarphi}, $x \mapsto \varphi(t, x)$ is also semi-concave with the same semi-concavity constant. For every $t \in \mathbb R^+$, $n \in \mathbb N$, and almost every $x \in \accentset\circ\Omega$, $\nabla\varphi_{\varepsilon_n}(t, x)$ and $\nabla\varphi(t, x)$ exist. Then, for every $h > 0$ small enough, one has
\[
\varphi_{\varepsilon_n}(t, x + h) - \varphi_{\varepsilon_n}(t, x) - \nabla\varphi_{\varepsilon_n}(t, x) \cdot h \leq C \abs{h}^2.
\]
Letting $n \to \infty$, up to extracting a subsequence, $\nabla\varphi_{\varepsilon_n}(t, x)$ converges to some $p \in \mathbb R^d$, and then
\[
\varphi(t, x + h) - \varphi(t, x) - p \cdot h \leq C \abs{h}^2.
\]
This means that $p \in \nabla^+ \varphi(t, x)$ and, since $\nabla\varphi(t, x)$ exists, one concludes that $\nabla\varphi_{\varepsilon_n}(t, x) \to \nabla\varphi(t, x)$ as $n \to \infty$. Moreover, by Proposition \ref{PropLowerBoundOnGrad}, there exists $c > 0$ such that $\abs{\nabla\varphi_{\varepsilon_n}(t, x)} \geq c$, implying that $\frac{\nabla\varphi_{\varepsilon_n}(t, x)}{\abs{\nabla\varphi_{\varepsilon_n}(t, x)}} \to \frac{\nabla\varphi(t, x)}{\abs{\nabla\varphi(t, x)}}$ as $n \to \infty$. One then concludes that, for every $t \geq 0$,
\begin{equation}
\label{ContinuityEquation-ConvergenceDiverg}
\lim_{n \to \infty} \int_{\accentset\circ\Omega} \rho^{\varepsilon_n}_t(x) k_{\varepsilon_n}(\rho^{\varepsilon_n}_t, x) \frac{\nabla \varphi_{\varepsilon_n}(t, x)}{\abs{\nabla \varphi_{\varepsilon_n}(t, x)}} \cdot \nabla \phi(t, x) \diff x = \int_{\accentset\circ\Omega} \rho_t(x) k(\rho_t, x) \frac{\nabla \varphi(t, x)}{\abs{\nabla \varphi(t, x)}} \cdot \nabla \phi(t, x) \diff x.
\end{equation}
Combining \eqref{ContinuityEquation-ConvergenceDt} and \eqref{ContinuityEquation-ConvergenceDiverg}, one obtains from \eqref{ContinuityEquation-Weak} that
\[
- \int_0^T \int_{\accentset\circ\Omega} \partial_t \phi(t, x) \rho_t(x) \diff x \diff t + \int_0^T \int_{\accentset\circ\Omega} \rho_t(x) k(\rho_t, x) \frac{\nabla \varphi(t, x)}{\abs{\nabla \varphi(t, x)}} \cdot \nabla \phi(t, x) \diff x \diff t = 0,
\]
yielding the conclusion.
\end{proof}

\section*{Acknowledgements}

The authors would like to thank Filippo Santambrogio for the fruitful discussions that lead to this paper.

\bibliographystyle{abbrv}
\bibliography{PaperDM}

\end{document}